\documentclass[a4paper, 11 pt]{article}
\usepackage{a4wide, amsmath, amssymb, mathtools, yfonts}
\usepackage{comment}
\usepackage{tikz}
\usepackage{tikz-cd}
\usepackage[all]{xy}
\usepackage[utf8]{inputenc}
\usepackage{amsthm, mathrsfs}
\usepackage[english]{babel}
\usepackage{hyperref}
\usepackage{authblk}
\usepackage[OT2,T1]{fontenc}
\DeclareSymbolFont{cyrletters}{OT2}{wncyr}{m}{n}
\DeclareMathSymbol{\Sha}{\mathalpha}{cyrletters}{"58}

\numberwithin{equation}{section}
\newtheorem{lemma}{Lemma}[section]
\newtheorem{theorem}[lemma]{Theorem}
\newtheorem{proposition}[lemma]{Proposition}
\newtheorem{corollary}[lemma]{Corollary}

\newtheorem{remark}[lemma]{Remark}

\theoremstyle{definition}
\newtheorem{mydef}[lemma]{Definition}
\newtheorem{setup}[lemma]{Setup}
\newtheorem{example}[lemma]{Example}
\newtheorem{notat}[lemma]{Notation}

\newcommand{\Z}{\mathbb{Z}}
\newcommand{\Q}{\mathbb{Q}}
\newcommand{\C}{\mathbb{C}}
\newcommand{\QQ}{\Q}
\newcommand{\FFF}{\mathbb{F}}
\newcommand{\Qbar}{\overline{\Q}}
\newcommand{\R}{\mathbb{R}}
\newcommand{\FF}{\mathbb{F}}

\newcommand{\Hom}{\mathrm{Hom}}
\newcommand{\Frob}{\textup{Frob}}

\newcommand\Gal{\mathrm{Gal}}

\newcommand{\Sel}{\textup{Sel}}
\newcommand{\Vplac}{\mathscr{V}}
\newcommand{\CTP}{\textup{CTP}}
\newcommand{\Redei}[3]{\left[#1,\, #2,\,#3\right]}
\newcommand{\mfp}{\mathfrak{p}}

\newcommand{\mfq}{\mathfrak{q}}
\newcommand{\mfa}{\mathfrak{a}}
\newcommand{\mfb}{\mathfrak{b}}
\newcommand{\mfr}{\mathfrak{r}}
\newcommand{\mff}{\mathfrak{f}}
\newcommand{\mfh}{\mathfrak{h}}
\newcommand{\mfB}{\mathfrak{B}}
\newcommand{\ovp}{\overline{\mfp}}
\newcommand{\ovQQ}{\Qbar}
\newcommand{\class}[1]{\left[#1\right]}
\newcommand{\symb}[2]{\left[#1,\, #2\right]}
\newcommand{\isoarrow}{\xrightarrow{\,\,\sim\,\,}}
\newcommand{\res}{\textup{res}}
\newcommand{\inv}{\textup{inv}}

\title{\vspace{-\baselineskip}\sffamily\bfseries Sums of rational cubes and the $3$-Selmer group}
\author[1]{Peter Koymans\thanks{Institute for Theoretical Studies, ETH Zurich, 8092 Zurich, Switzerland, peter.koymans@eth-its.ethz.ch}}
\author[2]{Alexander Smith\thanks{Department of Mathematics, UCLA, Los Angeles, CA, asmith13@math.ucla.edu}}
\affil[1]{ETH Zurich}
\affil[2]{UCLA}
\date{\today}

\begin{document}

\maketitle

\begin{abstract}

Recently, Alp\"oge--Bhargava--Shnidman \cite{ABS} determined the average size of the $2$-Selmer group in the cubic twist family of any elliptic curve over $\Q$ with $j$-invariant $0$. We obtain the distribution of the $3$-Selmer groups in the same family. As a consequence, we improve their upper bound on the density of integers expressible as a sum of two rational cubes. Assuming a $3$-converse theorem, we also improve their lower bound on this density.

The $\sqrt{-3}$-Selmer group in this cubic twist family is well-known to be large, which poses significant challenges to the methods previously developed by the second author \cite{Smi22a, Smi22b}. We overcome this problem by strengthening the analytic core of these methods. Specifically, we prove a ``trilinear large sieve'' for an appropriate generalization of the classical R\'edei symbol, then use this to control the restriction of the Cassels--Tate pairing to the $\sqrt{-3}$-Selmer groups in these twist families.
\end{abstract}

\section{Introduction}
Given nonzero integers $d, n$, take $E_{d, n}/\Q$ to be the elliptic curve with Weierstrass form
\[E_{d, n}: y^2 = x^3 + dn^2.\]
If we fix $d$ and allow $n$ to vary, this defines a cubic twist family of rational elliptic curves. We are interested in how the ranks of these curves are distributed across this twist family. In practice, this means we are interested in understanding the distribution of Selmer groups in these families.

Much of the recent progress towards determining the distribution of Selmer groups in families of elliptic curves over number fields has come from two main methods, and both apply to this problem. These might be called the ``geometry of numbers'' method and the ``Galois invariant submodule'' method.

The geometry of numbers method was first developed by Bhargava and Shankar to determine the average size of $\Sel^n E$ over the family of all rational elliptic curves ordered by na\"{i}ve height  for $2 \le n \le 5$ \cite{BS1, BS2, BS3, BS4}. In summary, this method starts with a parameterization of some set of tuples $(E, \phi)$, where $E$ is an elliptic curve in some family and $\phi$ is a Selmer element, as the integral points in some fundamental domain in a nice algebraic variety, then uses the geometry of numbers and related analytic techniques to count the number of such tuples up to a given height.

This method was applied to the problem of finding the average size of the $2$-Selmer group in a cubic twist family by Alp\"{o}ge, Bhargava, and Shnidman \cite{ABS}. We start by recalling these results.

\begin{mydef}
Given $w \in \pm 1$, a nonzero integer $d$, and a positive real number $H$, define
\[
\mathscr{T}(w, d, H) = \left\{ n \in \Z_{>0} \,:\,\, n < H \,\text{ and }\, w(E_{d, n}) = w\right\},
\]
where $w(E_{d, n})$ denotes the global root number of $E_{d, n}$. We note that the natural density of each root number is $1/2$ in this family \cite[Theorem 1.4]{ABS}; that is
\begin{equation}
\label{eq:root_number_equi}
\lim_{H \to \infty} \frac{\#\mathscr{T}(+1, d, H)}{H} = \lim_{H \to \infty} \frac{\#\mathscr{T}(-1, d, H)}{H} = 1/2.
\end{equation}
\end{mydef}


\begin{theorem}[\cite{ABS}]
Fix $w \in \pm 1$ and a nonzero integer $d$. Then
\[\lim_{H \to \infty} \frac{\sum_{ n \in \mathscr{T}(w, d, H)} \# \Sel^2 E_{d, n}}{\sum_{ n \in \mathscr{T}(w, d, H)}  1}  = 3.\]
\end{theorem}

\subsection{The Galois invariant submodule method}
This paper instead applies the Galois invariant submodule method, which we overview now. A standard application of this method starts with the observation that, in a quadratic twist family of an elliptic curve $E$ over a number field, the $2$-torsion submodule is Galois invariant. In particular, the $2$-Selmer groups in this family may be viewed as the subgroups of $H^1(G_F, E[2])$ cut out by a varying collection of local conditions. This approach was first applied by Heath-Brown to study $2$-Selmer groups in the quadratic twist family of the congruent number curve \cite{HB}, and the extension to higher Selmer groups was later developed by the second author \cite{Smi22a, Smi22b}.

To apply this to the problem at hand, we note that $E_{d, n}$ has CM by the ring $\Z[\mu_3] \subseteq \Q(\sqrt{-3})$. Multiplication by $\sqrt{-3}$  defines a rational isogeny
\[
\sqrt{-3}: E_{d, n} \to E_{-27d, n},
\]
of degree $3$. The Galois structure of the kernel of this isogeny does not depend on $n$, so we always have a natural embedding
\[
\Sel^{\sqrt{-3}} E_{d, n} \hookrightarrow H^1\left(G_F,\, E_{d, 1}\left[\sqrt{-3}\right]\right).
\]
This Selmer group is analogous to the $2$-Selmer group in a quadratic twist family, but is unwieldy in a way that makes higher Selmer work more difficult. Specifically, we have the following proposition.

\begin{proposition}
\label{prop:blowup}
Choose a nonzero integer $d$ and a real number $\delta \in (0, 1)$. Then, if $d$ is a square, we have
\begin{equation}
\label{eq:coinflipA}    
\dim_{\FFF_3} \Sel^{\sqrt{-3}} E_{d, n} = 0 \quad \textup{and} \quad \dim_{\FFF_3} \Sel^{\sqrt{-3}} E_{-27d, n} > \left(\log \log  3n \right)^{\delta}
\end{equation}
for 100\% of positive integers $n$. Equivalently, if $-27d$ is a square, we have
\begin{equation}
\label{eq:coinflipB}
\dim_{\FFF_3} \Sel^{\sqrt{-3}} E_{d, n} > \left(\log \log  3n \right)^{\delta} \quad \textup{and} \quad \dim_{\FFF_3} \Sel^{\sqrt{-3}} E_{-27d, n} = 0
\end{equation}
for 100\% of positive integers $n$.

Otherwise, if neither $d$ nor $-27d$ is a square, and if $\delta$ lies in $(0, 1/2)$, we find that \eqref{eq:coinflipA} holds for 50\% of positive integers $n$ and that \eqref{eq:coinflipB} holds for 50\% of positive integers $n$.
\end{proposition}

This behavior, where a pair of isogenous curves will have one trivial Selmer group and one enormous Selmer group between them, is not unusual; it is encountered in the $2$-Selmer groups in the quadratic twist families of elliptic curves with partial $2$-torsion, for example \cite[Theorem 1.4]{Smi22b}. But it is a more fundamental issue here, as we cannot avoid it by focusing on the curve where the Selmer group is small. To directly bound the $3$-Selmer rank of $E_{d, n}$ using $\sqrt{-3}$-Selmer ranks, we use the exact sequence
\begin{equation}
\label{eq:Sel3decomp}
H^0\left(G_{\QQ}, E_{-27d, n}\left[\sqrt{-3}\right]\right)\to \Sel^{\sqrt{-3}} E_{d, n}  \to  \Sel^{3} E_{d, n} \to \Sel^{\sqrt{-3}} E_{-27d, n}
\end{equation}
to derive the relation
\[\dim_{\FFF_3} \Sel^3 E_{d, n} \,\le\, \dim_{\FFF_3}\Sel^{\sqrt{-3}} E_{d, n} \,+\, \dim_{\FFF_3}\Sel^{\sqrt{-3}} E_{-27d, n}.\]
This relation invokes the $\sqrt{-3}$-Selmer rank of both isogenous curves, so any higher work needs to reckon with the fact that typically one of these groups is large.

\subsection{Our results}
Nonetheless, we can make some progress towards understanding the $3$-Selmer groups. To account for the behavior of Proposition \ref{prop:blowup}, we will use the following variation on the definition of the $3$-Selmer rank.

\begin{mydef}
Given nonzero integers $d, n$, we define
\[
r_3(E_{d, n}) = \begin{cases}-1 + \dim_{\FFF_3} \Sel^3 E_{d, n} &\text{ if } d \text{ is square} \\
-1 + \dim_{\FFF_3} \Sel^3 E_{-27d, n} &\text{ if } -27d \text{ is square}\\
\min\left(\dim_{\FFF_3} \Sel^3 E_{d, n},\, \dim_{\FFF_3} \Sel^3 E_{-27d, n}\right) &\text{ otherwise. }\end{cases}
\]
This is always an upper bound for the rank of $E_{d, n}$.
\end{mydef}

\begin{theorem}
\label{thm:main}
Fix $w \in \pm 1$ and a nonzero integer $d$, and take
\[\alpha_0 = \prod_{k=0}^{\infty} ( 1- 3^{-2k-1}) \approx .6390.\]
Then, for any integer $r \ge 0$ satisfying $(-1)^r = w$, we have
\[\lim_{H \to \infty} \frac{\# \left\{ n \in \mathscr{T}(w, d, H)\,:\,\, r_3(E_{d, n}) =  r \right\} }{\#\mathscr{T}(w, d, H)} = \alpha_0 \cdot 3^{\frac{-r(r-1)}{2}}\cdot \prod_{k=1}^r \left(1 - 3^{-k}\right)^{-1}.\]
\end{theorem}

The following corollary is an immediate consequence of our main theorem.
Write $r(E)$ for the rank of an elliptic curve $E$.

\begin{corollary}
Fix a nonzero integer $d$. We then have
$$
\liminf_{H \to \infty} \frac{\# \left\{ n \in \mathscr{T}(+1, d, H)\,:\,\, r(E_{d, n}) =  0 \right\}}{\#\mathscr{T}(+1, d, H)} \geq 0.6390.
$$
\end{corollary}

Combining this with \eqref{eq:root_number_equi} gives
\[
\liminf_{H \to \infty} \frac{\# \left\{ n \in \mathscr{T}(+1, d, H)\,:\,\, r(E_{d, n}) =  0 \right\}}{H} \geq 0.3195.
\]
Since the elliptic curves $x^3 + y^3 = n$ and $y^2 = x^3 - 432n^2$ are isomorphic for any nonzero integer $n$, the case $d = -432$ of this gives

\begin{corollary}
At least $31.95\%$ of positive integers are not the sum of two rational cubes.
\end{corollary}

This improves on the bounds of  Alp\"oge--Bhargava--Shnidman \cite[Theorem 1.1]{ABS}, which proves that at least $1/6 \approx 16.67\%$ of positive integers are not the sum of two rational cubes. 

By taking advantage of a $2$-converse theorem of Burungale and Skinner \cite[Corollary A.2]{ABS}, these authors also show that at least $2/21 \approx 9.52\%$ of positive integers are the sum of two rational cubes. We do not have access to an appropriate $3$-converse theorem, so we may only provide a conditional result here.

\begin{corollary}
Suppose the implication
\begin{equation}
\label{eq:3converse}
r_3(E_{-432, n}) = 1 \implies r(E_{-432, n})  = 1 
\end{equation}
holds for all positive integers $n$. Then at least $47.92\%$ of positive integers are the sum of two rational cubes.
\end{corollary}

Unfortunately, and unlike in the case of the corresponding $2$-converse statement, the hypothesis \eqref{eq:3converse} is not known for even a positive density of positive integers $n$ such that $w(E_{-432, n}) = -1$. 

We note that assuming either the BSD conjecture or the Shafarevich--Tate conjecture would suffice to prove that every elliptic curve with global root number $-1$ has positive rank, making them inappropriate hypotheses for this corollary.

\subsection{The trilinear large sieve}
The Cassels--Tate pairing \cite{MS22a} defines an alternating pairing
\[\CTP_{d, n}: \Sel^{\sqrt{-3}} E_{-27d, n} \times \Sel^{\sqrt{-3}} E_{-27d, n} \to \tfrac{1}{3}\Z/\Z\]
whose kernel equals the image of $\Sel^{3} E_{d, n}$ in \eqref{eq:Sel3decomp}. Assuming $-27d$ is not a square, the exact sequence \eqref{eq:Sel3decomp} gives
\[\dim_{\FFF_3} \Sel^{3} E_{d, n} = \dim_{\FFF_3} \Sel^{\sqrt{-3}} E_{d, n} + \dim\left(\ker \CTP_{d, n}\right).\]
If $d$ is square, Proposition \ref{prop:blowup} implies that
\[\dim_{\FFF_3} \Sel^{3} E_{d, n} = \dim\left(\ker \CTP_{d, n}\right)\]
for 100\% of positive integers $n$. If neither $d$ nor $-27d$ is square, Proposition \ref{prop:blowup} instead gives the relationship
\[\min\left(\dim_{\FFF_3} \Sel^{3} E_{d, n}, \,\dim_{\FFF_3} \Sel^{3} E_{-27d, n}\right) \le \max\left(\dim \ker \CTP_{d, n},\, \dim \ker \CTP_{-27d, n}\right),\]
with equality for 100\% of positive integers $n$. Our main goal is then to understand how the pairing $\CTP_{d, n}$ varies as $n$ varies.

To use the Galois invariant submodule method, we should look for relationships between the Cassels--Tate pairings for collections of cubic twists $E_{d, n}$ whose $\sqrt{-3}$-Selmer groups have overlapping images in $H^1\left(G_{\QQ},\, E_{d, 1}\left[\sqrt{-3}\right]\right)$. These relationships will express the sum of Cassels--Tate pairings over some collection of twists in terms of splitting condition for some prime in some Galois extension of $\QQ$. Following the method, we wish to control these splitting conditions using a technique from analytic number theory, then use this to control the distribution of kernels of the Cassels--Tate pairing.

At this level of generality, our approach is the same as in \cite{Smi22a}. The novelty of our work lies in the analytic tool we use to control the splitting conditions.

In the preprint \cite{Smi17}, which was the initial form of the two-part \cite{Smi22a, Smi22b}, the analytic tool at the core of the argument is the unconditional Chebotarev density theorem. This is also the tool at the core of \cite{KoyPag21, KoyPag22}.

In \cite[Section 5]{Smi22a}, this tool is replaced with a result on bilinear character sums that can reasonably either be classified as an averaged form of the Chebotarev density theorem or as a generalized ``large sieve''. This tool generalizes the following result of Jutila

\begin{theorem}[{\cite[Lemma 3]{Juti79}}]
Given integers $d_1, d_2$ with $d_2$ odd, take $\left(\frac{d_1}{d_2}\right)$ to be the Jacobi symbol, the multiplicative generalization of the Legendre symbol. If $d_2$ is even, we take this symbol to be $0$.

Then, given real numbers $H_1, H_2 \ge 3$, and given any functions
\[
a_1: \Z \to \mathbb{C} \quad\textup{and}\quad a_2: \Z \to \mathbb{C}
\]
of magnitude bounded by $1$, we have
\[
\left|\sum_{|d_1| < H_1} \sum_{|d_2| < H_2}a_1(d_1)a_2(d_2) \left(\frac{d_1}{d_2}\right)\right|\ll H_1 H_2 \cdot (\log H_1H_2)^3 \cdot (H_1^{-1/4} + H_2^{-1/4}).
\]
\end{theorem}

The result \cite[Theorem 5.2]{Smi22a} generalizes this bilinear result to handle symbols defined over arbitrary number fields, rather than just $\Q$. To control the distribution of the Cassels--Tate pairing with minimal error, it was necessary to apply this result for symbols defined over number fields of degree as large as possible. This was the limiter of the strength of the equidistribution results for this method in \cite{Smi22a}, giving results stronger than can be proved with unconditional Chebotarev alone but weaker than could be useful for the problem of controlling $3$-Selmer ranks in cubic twist families of elliptic curves over $\Q$.

The key observation of this paper is that, to control $3$-Selmer ranks, we should go one step further and prove a \emph{trilinear} character bound. Unlike bilinear results, no such trilinear result has been proved before. Perhaps the simplest possible example is the following, which follows from Theorem \ref{thm:trilinear} and Remark \ref{rmk:supersolvable}.
\begin{theorem}
\label{thm:classic_trilinear_LS}
Given integers $d_1, d_2, d_3$, take $\Redei{d_1}{d_2}{d_3} \in \{-1, 0, 1\}$ to be the R\'edei symbol, as defined in Example \ref{ex:Redei}. Choose three functions
\[
a_{12}, a_{13}, a_{23}: \Z \times \Z \to \mathbb{C}
\]
of magnitude at most $1$. Then, given real numbers $H_1, H_2, H_3 \ge 3$, we have
\begin{multline*}
\left|\sum_{|d_1| < H_1} \sum_{|d_2| < H_2} \sum_{|d_3| < H_3} a_{12}(d_1,d_2)\cdot a_{13}(d_1, d_3)\cdot  a_{23}(d_2, d_3) \cdot \Redei{d_1}{d_2}{d_3}\right| \\
\ll H_1H_2H_3 \cdot \log(H_1H_2H_3)^{1792} \cdot \left( H_1^{-1/512} + H_2^{-1/512} + H_3^{-1/512}\right)
\end{multline*}
with the implicit constant absolute.
\end{theorem}

This problem was previously considered in \cite{IKZ}, where a density result was obtained conditional on GRH when $d_1, d_2, d_3$ are restricted to primes. Our work unconditionally recovers the main result in \cite{IKZ} with a power-saving error term.

In Section \ref{sec:GRS}, we give a construction for a trilinear symbol that generalizes the R\'{e}dei symbol, and we prove the appropriate generalization of Theorem \ref{thm:classic_trilinear_LS} for this symbol in Section \ref{sec:trilinear}. In Section \ref{sCubicTwist} and Section \ref{sec:Leg} we lay the algebraic foundations for studying the $\sqrt{-3}$-Selmer group and $3$-Selmer group. We use these foundations to study their statistical behavior in Section \ref{sFinal}.

\subsection*{Acknowledgements}
The first author gratefully acknowledges the support of Dr. Max R\"ossler, the Walter Haefner Foundation and the ETH Z\"urich Foundation. The second author served as as Clay Research Fellow during the writing of this paper, and would like to thank the Clay Mathematics Institute for their support.

\section{Generalized R\'edei symbols}
\label{sec:GRS}
We start with some notation.

\begin{notat}
Fix a number field $F$ and an algebraic closure $\overline{F}$. Take $G_F = \Gal(\overline{F}/F)$ to be the absolute Galois group of $F$. Given a place $v$ of $F$, fix an algebraic closure $\overline{F_v}$ of the completion $F_v$, and take $G_v = \Gal(\overline{F_v}/F_v)$ to be the corresponding absolute Galois group. 

All cohomology groups are continuous cochain cohomology. We take $Z^i$ and $C^i$ as notation for the associated spaces of $i$-cocycles and $i$-cochains.

A finite $G_F$-module $M$ is a discrete, finite abelian group with a continuous action of $G_F$. For such a module, we will define
$$
\Sha^i(G_F, M) = \ker\left( H^i(G_F, M) \to \prod_{v \text{ of } F} H^i(G_v, M)\right)
$$
for $i = 1, 2$, where the inclusion of $G_v$ in $G_F$ corresponds to some choice of an embedding of  $\overline{F}$ in $\overline{F_v}$. These groups do not depend on the choice of this inclusion.

Given a positive integer $n$, we take $\mu_n$ to be the group of $n^{th}$ roots of unity in $\overline{F}$.
\end{notat}

\subsection{Construction of the symbol}
\begin{mydef}
\label{def:Redei}
With $F$ fixed as above, choose finite $G_F$ modules $M_1$, $M_2$, and $M_3$, and choose an equivariant homomorphism
\begin{equation}
\label{eq:M1M2M3_pairing}
M_1 \otimes M_2 \otimes M_3 \to \overline{F}^{\times}.
\end{equation}
Here and for the rest of the paper, the tensor product is taken over $\Z$ and $M_1 \otimes M_2 \otimes M_3$ is endowed with the structure of a $G_F$-module via $g(m_1 \otimes m_2 \otimes m_3) := gm_1 \otimes gm_2 \otimes gm_3$.

Choose elements $\phi_i \in H^1(G_F, M_i)$ for $i\le 3$. We define a \emph{generalized R\'edei symbol}
\[
\Redei{\phi_1}{\phi_2}{\phi_3} \,\in\, \{0\} \cup \{z \in \mathbb{C}\,:\,\, |z| =  1\} 
\]
as follows:

Take $S$ to be the minimal set of places of $F$ containing all archimedean primes, all primes dividing the order of some $M_i$, all primes where the action of $G_F$ on some $M_i$ is ramified, and all primes where some $\phi_i$ is ramified. We call $S$ the \emph{bad primes} for this symbol.

For each place $v$ of $F$, we choose a corresponding inclusion of the algebraic closure $\overline{F}$ into $\overline{F_v}$, and we take $G_v \subseteq G_F$ to be the corresponding inclusion of Galois groups. We write $\text{inv}_v$ for the standard map $H^2\big(G_v,\,{\overline{F_v}}^{\times}\big) \to \Q/\Z$. 

With this set, we say the \emph{symbol condition holds at $v$} if
\begin{itemize}
\item For some $i \le 3$, $\phi_i$ has trivial restriction to $H^1(G_v, M_i)$; and
\item Given any permutation $(i, j, k)$ of $(1, 2, 3)$, and given any $x \in H^0(G_v, M_i)$, the cup product
\[
x \cup  \phi_j \cup \phi_k \in H^2\big(G_v, \overline{F_v}^{\times}\big)
\]
with respect to the properly-permuted variant of \eqref{eq:M1M2M3_pairing} is trivial. 
\end{itemize}
By local Tate duality, the second condition is equivalent to the condition that the cup product $\phi_j \cup \phi_k \in H^2(G_F, M_i^{\vee})$, defined with respect to the homomorphism $M_j \otimes M_k \rightarrow M_i^\vee$ induced by \eqref{eq:M1M2M3_pairing}, is trivial at $G_v$. This cup product is also trivial at $G_v$ for $v$ outside $S$ since $\phi_j$ and $\phi_k$ are unramified at such a $v$. So we find that the claim
\[\phi_j \cup \phi_k \in \Sha^2(G_F, M_i^{\vee}) \quad\text{ for each permutation } (i, j, k) \text{ of } (1, 2, 3)\]
is equivalent to the second part of the symbol condition holding for each $v$ in $S$.

If the symbol condition does not hold  for some $v \in S$, we define $[\phi_1, \phi_2, \phi_3] = 0$.

Otherwise, choose a cocycle $\overline{\phi_i} \in Z^1(G_F, M_i)$ representing $\phi_i$ for each $i \le 3$. Take $F_S \subset \overline{F}$ to be the maximal extension of $F$ ramified only at places over $S$, take $\mathcal{O}_{F, S}$ to be its ring of integers, and take $G_{F, S} = \textup{Gal}(F_S/F)$. From \cite[(8.3.11)]{Neuk07}, we may choose a cochain
\[
\epsilon \in C^2\left(G_{F, S}, \mathcal{O}_{F, S}^{\times}\right)
\]
satisfying
\[
d\epsilon = \overline{\phi_1} \cup \overline{\phi_2} \cup \overline{\phi_3},
\]
where the cup product is defined with respect to \eqref{eq:M1M2M3_pairing}.

For each $v \in S$, choose $i$ so $\phi_i$ is trivial at $G_v$, and choose $x \in M_i$ satisfying
\[
\text{res}_{G_v} \overline{\phi_i} = dx.
\]
Taking $j, k$ so $(i, j, k)$ is a permutation of $(1, 2, 3)$, we then define a $2$-cochain
\begin{equation}
\label{eq:fv}
f_v := (-1)^{i+1} \cdot \psi_1 \cup \psi_2 \cup \psi_3\quad\text{with}\quad \psi_i = x,\,\,\,\,\, \psi_j  = \textup{res}_{G_v}\overline{\phi_j},\,\text{ and }\, \psi_k  = \textup{res}_{G_v}\overline{\phi_k}.
\end{equation}
Then $\textup{res}_{G_v}\epsilon -f_v$ is a cocycle, and we may define
\[
\Redei{\phi_1}{\phi_2}{\phi_3} = \exp\left( 2\pi i \cdot \sum_{v \in S}\text{inv}_v\left(\textup{res}_{G_v}\epsilon -f_v\right)\right).
\]
\end{mydef}

\begin{lemma}
\label{lem:symbol_well_defined}
The symbol $\Redei{\phi_1}{\phi_2}{\phi_3}$ is well-defined. 
\end{lemma}

\begin{proof}
We need to check that the definition is independent of the choice of the cochains $f_v$, the choice of the cochain $\epsilon$, the choice of the cocycles $\overline{\phi_i}$, and the choice of the inclusion of $\overline{F}$ in $\overline{F_v}$ for each place $v$ in $S$. Except for the last part, our proof will follow the basic framework of the well-definedness of the Cassels--Tate pairing \cite[Proposition 3.3]{MS22a}.

First, given $v \in S$, suppose we choose $x, x'$ so $dx = dx' = \textup{res}_{G_v} \overline{\phi_i}$. Then $x - x'$ lies in $H^0(G_v, M_i)$, so
\[
\text{inv}_v\left((x - x') \cup \phi_j \cup \phi_k\right) = 0
\]
by the second part of the symbol condition. So the choice of $x$ does not matter.

At the same step, suppose $\textup{res}_{G_v}\phi_j$ is also trivial, and choose $y \in M_j$ whose coboundary equals $\textup{res}_{G_v}\overline{\phi_j}$. Swapping $i$ and $j$ and replacing $x$ with $y$ in \eqref{eq:fv} defines a new cochain $f'_v$. But we see that
\[
f_v - f'_v \in \pm d(\alpha_1 \cup \alpha_2 \cup \alpha_3) \quad\text{with}\quad \alpha_i = x, \,\,\,\, \alpha_j  = y,\,\text{ and }\,\, \alpha_k  = \textup{res}_{G_v} \overline{\phi_k}.
\]
In short, $f_v$ is determined up to a coboundary, so its choice does not affect the symbol.

By Poitou--Tate duality, we see that the choice of $\epsilon$ does not affect the symbol.

Now suppose we replace $\overline{\phi_1}$ with $\overline{\phi_1} + dx$, with $x \in M_1$. We then may adjust $\epsilon$ to $\epsilon' = \epsilon + x \cup \overline{\phi_2} \cup \overline{\phi_3}$. For each $v \in S$, if the permutation $(i, j, k)$ corresponding to $v$ has $i = 1$, $f_v$ may be adjusted by the restriction of $\epsilon' - \epsilon$. If $i = 2$, $f_v$ may be replaced with
\[
f_v' = f_v - dx \cup y \cup \textup{res}_{G_v} \overline{\phi_3}
\]
for some $y$ satisfying $dy = \textup{res}_{G_v} \overline{\phi_2}$, and $- dx \cup y \cup \textup{res}_{G_v} \overline{\phi_3}$ differs from the restriction of $\epsilon' - \epsilon$ by a coboundary. This also holds for $i = 3$, so we find that the choice of $\overline{\phi_1}$ does not affect the pairing. We may similarly verify that the choices of $\overline{\phi_2}$ and $\overline{\phi_3}$ do not affect the pairing.

This leaves the choice of embeddings $\iota_v: \overline{F} \hookrightarrow \overline{F_v}$ to consider. We consider how the symbol changes if we replace $\iota_v$ with $\iota_v \circ \tau$ for a given $\tau \in G_F$. Since the proof of Proposition \ref{prop:rec} below does not depend on invariance of field embeddings, we may assume $\phi_1$ is trivial at $G_v$.

From the above work, we may therefore assume that the cocycle $\overline{\phi_1}$ is trivial on the set $\iota_v^\ast(G_v) := \{\iota_v^{-1} \circ \sigma \circ \iota_v : \sigma \in G_v\}$. Then
\[ 
\overline{\phi_1}(\tau^{-1} \sigma \tau) = -(\tau^{-1} \sigma \tau - 1) \overline{\phi_1}(\tau^{-1})
\]
for $\sigma \in \iota_v^\ast(G_v)$. Taking $x = -\overline{\phi_1}(\tau^{-1})$ in the definition of $f_v$ with respect to $\iota_v \circ \tau$, we find that the claim reduces to
$$
\textup{inv}_v\left(\textup{res}_{G_v}\left(\tau \epsilon - \epsilon - \tau f_v\right)\right) = 0\quad\text{with}\quad f_v = -\overline{\phi_1}(\tau^{-1}) \cup \overline{\phi_2} \cup \overline{\phi_3},
$$
where the implicit field embedding here is $\iota_v$, so that $\textup{res}_{G_v}$ is to be interpreted as the restriction to the subgroup $\iota_v^\ast(G_v)$ of $G_F$. In the above equation, $\tau \epsilon$ and $\tau f_v$ are shorthand for the map $\text{conj}_\tau^G$ introduced in \cite[p. 9]{MS22b}. Here we implicitly view $f_v$ as an element of $C^2(G_F, \overline{F}^\times)$ so that the conjugation map is defined on $f_v$.

Using the map $h_{\tau}$ appearing in \cite[(3.2)]{MS22b}, we find that $\textup{res}_{G_v}(\tau \epsilon - \epsilon)$ differs from
\[
\res_{G_v} h_{\tau}\left(\overline{\phi_1} \cup \overline{\phi_2} \cup \overline{\phi_3}\right) = \textup{res}_{G_v} \left(\overline{\phi_1}(\tau) \cup \tau \overline{\phi_2} \cup \tau\overline{\phi_3}\right)
\]
by a coboundary, where the equality follows from \cite[(3.3)]{MS22b}. This cancels the contribution from $f_v$, establishing the independence from the choice of $\iota_v$.
\end{proof}

\begin{proposition}[Trilinearity and reciprocity]
\label{prop:rec}
Take $M_1, M_2, M_3$ and the homomorphism \eqref{eq:M1M2M3_pairing} as above. Given $\phi_1, \phi_1' \in H^1(G_F, M_1)$, $\phi_2 \in H^1(G_F, M_2)$, and $\phi_3 \in H^1(G_F, M_3)$, if $\Redei{\phi_1}{\phi_2}{\phi_3}$ is nonzero, then
\[
\Redei{\phi_1 + \phi_1'}{\phi_2}{\phi_3} = \Redei{\phi_1}{\phi_2}{\phi_3} \cdot \Redei{\phi_1'}{\phi_2}{\phi_3}.
\]
Furthermore, we have the identities
\[
\Redei{\phi_2}{\phi_1}{\phi_3} = \Redei{\phi_1}{\phi_3}{\phi_2} = \overline{{\Redei{\phi_1}{\phi_2}{\phi_3}}},
\]
where the homomorphisms used to define $\Redei{\phi_2}{\phi_1}{\phi_3}$ and $\Redei{\phi_1}{\phi_3}{\phi_2}$ are given by swapping inputs in \eqref{eq:M1M2M3_pairing}.
\end{proposition}

\begin{proof}
We start with the proof of linearity. We suppose to start that $\Redei{\phi_1'}{\phi_2}{\phi_3}$ is nonzero.

Take $S$ to be the set of bad primes of $\Redei{\phi_1}{\phi_2}{\phi_3}$ and $S'$ to be the set of bad primes of $\Redei{\phi_1'}{\phi_2}{\phi_3}$. Choose the tuple
\[
\left(\overline{\phi_1}, \,\overline{\phi_2}, \,\overline{\phi_3}, \,\epsilon\right)
\]
as in Definition \ref{def:Redei}, and choose $\overline{\phi_1'}$ and $\epsilon'$ as in this definition for evaluating $\Redei{\phi'_1}{\phi_2}{\phi_3}$.

Choose $v \in S$. We claim the symbol condition holds at $v$ for $\Redei{\phi'_1}{\phi_2}{\phi_3}$. This is clear if $v$ is in $S'$. Otherwise, $\phi_1$ must be ramified at $v$, and $\phi_i$ must be trivial at $v$ for $i$ either $2$ or $3$. We also see that $\phi'_1$, $\phi_2$, and $\phi_3$ are all unramified at $v$, so the second part of the symbol condition holds at $v$ for $\Redei{\phi'_1}{\phi_2}{\phi_3}$. The claim follows. Symmetrically, the symbol conditions hold at $v \in S'$ for $\Redei{\phi_1}{\phi_2}{\phi_3}$. We conclude that the symbol conditions hold at all $v \in S \cup S'$ for $\Redei{\phi_1 + \phi'_1}{\phi_2}{\phi_3}$.

If neither $\phi_2$ nor $\phi_3$ are trivial at $v \in S \cup S'$, then both $\phi_1$ and $\phi'_1$ are trivial at $v$. So we may choose $(f_v)_{v \in S \cup S'}$ as in Definition \ref{def:Redei} for $\Redei{\phi_1}{\phi_2}{\phi_3}$ and $(f'_v)_{v \in S \cup S'}$ for $\Redei{\phi_1'}{\phi_2}{\phi_3}$ such that, for each $v$, $f_v$ and $f_v'$ are defined with respect to the same permutation of $(1, 2, 3)$. With this set, $\Redei{\phi_1 + \phi_1'}{\phi_2}{\phi_3}$ can be calculated from the tuple
\[
\left(\epsilon + \epsilon', (f_v + f'_v)_{v \in S \cup S'}\right).
\]
Therefore, in case $\Redei{\phi_1'}{\phi_2}{\phi_3}$ is nonzero, linearity is a consequence of the observation that
\[
\textup{inv}_v\left(\textup{res}_{G_v}\epsilon - f_v\right)= 0 \quad\text{for }\, v \in S' \backslash S\quad\text{and}\quad\textup{inv}_v\left(\textup{res}_{G_v}\epsilon' - f'_v\right)= 0 \quad\text{for }\, v \in S \backslash S',
\]
which follows from the fact that these cocycle classes come from $H^2(G_v/I_v, \mathcal{O}_{F, S}^{\times}) = 0$ and $H^2(G_v/I_v, \mathcal{O}_{F, S'}^{\times}) = 0$, respectively.

Now suppose that $\Redei{\phi'_1}{\phi_2}{\phi_3}$ is $0$. By what we have just proved, one of $\Redei{\phi'_1 + \phi_1}{\phi_2}{\phi_3}$ and $\Redei{\phi_1}{\phi_2}{\phi_3}$ must then also be $0$. So $\Redei{\phi'_1 + \phi_1}{\phi_2}{\phi_3} = 0$, giving linearity in this case.

We turn to reciprocity. We will show that
\[
\Redei{\phi_2}{\phi_1}{\phi_3} = \overline{\Redei{\phi_1}{\phi_2}{\phi_3}},
\]
with the proof of the other identity being similar. The identity is clear from the definitions if one of the symbols is zero. Otherwise, choose 
\[
\left(\overline{\phi_1}, \overline{\phi_2}, \overline{\phi_3}, \epsilon, (f_v)_{v \in S}\right)
\]
for calculating $\Redei{\phi_1}{\phi_2}{\phi_3}$.

Define
\[
\overline{\phi_1} \cdot \overline{\phi_2} \in C^1(G_F, M_1 \otimes M_2)
\]
by $\overline{\phi_1} \cdot \overline{\phi_2}(\sigma) = \overline{\phi_1}(\sigma) \otimes \overline{\phi_2}(\sigma)$. We may calculate
\[
d(\overline{\phi_1} \cdot \overline{\phi_2}) = -\overline{\phi_1} \cup \overline{\phi_2} - \overline{\phi_2} \cup \overline{\phi_1}.
\]
So we find that we may calculate $\Redei{\phi_2}{\phi_1}{\phi_3}$ from a tuple of the form
\[
\left(\overline{\phi_2},\, \overline{\phi_1}, \,\overline{\phi_3},\, -\epsilon - \left(\overline{\phi_1} \cdot \overline{\phi_2}\right) \cup \overline{\phi_3} ,\, (f'_v)_{v \in S}\right).
\]
We now focus on the $f_v$ and $f'_v$. First, given $v \in S$, suppose $\textup{res}_{G_v}\phi_1$ is trivial, and choose $x \in M_1$ so $dx = \textup{res}_{G_v} \overline{\phi_1}$. An explicit cochain calculation gives
\[
dx \cdot \textup{res}_{G_v} \overline{\phi_2} = \textup{res}_{G_v} \overline{\phi_2} \cup x - x \cup \textup{res}_{G_v} \overline{\phi_2},
\]
so we find that we may select $f_v$ and $f'_v$ such that
\begin{equation}
\label{eq:res_fvfvp}
-\textup{res}_{G_v}\left(\left(\overline{\phi_1} \cdot \overline{\phi_2}\right) \cup \overline{\phi_3} \right)= f_v + f_v'.
\end{equation}
A similar calculation holds if $\phi_2$ is trivial at $v$. Finally, if $\phi_3$ is trivial at $v$, with $\textup{res}_{G_v} \overline{\phi_3} = dx$, we may take $f_v = \textup{res}_{G_v}(\overline{\phi_1} \cup \overline{\phi_2}) \cup x$ and $f'_v = \textup{res}_{G_v}(\overline{\phi_2} \cup \overline{\phi_1}) \cup x$. Noting that
\[
d\left(\textup{res}_{G_v}\left(\overline{\phi_1} \cdot \overline{\phi_2}\right) \cup x\right) = - \textup{res}_{G_v}\left(\left(\overline{\phi_1} \cdot \overline{\phi_2}\right) \cup \overline{\phi_3} \right) - f_v - f'_v,
\]
we find that \eqref{eq:res_fvfvp} still holds modulo coboundaries. So
\[
\Redei{\phi_1}{\phi_2}{\phi_3} \cdot \Redei{\phi_2}{\phi_1}{\phi_3} = 1,
\]
and we are done.
\end{proof}

\subsection{Some examples of symbols}
Our first example justifies calling these symbols generalized R\'edei symbols.

\begin{example}
\label{ex:Redei}
Take $F = \Q$, take $M_1 =  M_2 = M_3 = \mu_2$, and take \eqref{eq:M1M2M3_pairing} to be the unique isomorphism from $\mu_2 \otimes \mu_2  \otimes \mu_2$ to $\mu_2$. The Kummer map defines an isomorphism 
\[ \Q^{\times}/(\Q^{\times})^2 \xrightarrow{\,\,\sim\,\,} H^1(G_{\Q}, \mu_2),\]
so we may think of the symbol defined in Definition \ref{def:Redei} as a map
\begin{equation}
\label{eq:discount_Redei}
\Redei{\,\,\,}{\,\,}{\,\,}: \,\Q^{\times}/\left(\Q^{\times}\right)^2 \,\times \,\Q^{\times}/\left(\Q^{\times}\right)^2 \,\times \,\Q^{\times}/\left(\Q^{\times}\right)^2 \,\xrightarrow{\quad}\, \pm 1 \cup \{0\}.
\end{equation}
This almost recovers the R\'edei symbol as codified by Stevenhagen in \cite[Section 7]{Stevenhagen}, which we write as $\Redei{a}{b}{c}_\text{R}$. Specifically, we have the following:
\end{example}

\begin{proposition}
Choose squarefree integers $a, b, c$. Then, if $\Redei{a}{b}{c}_{\textup{R}}$ is defined, the symbol \eqref{eq:discount_Redei} satisfies
\[
\Redei{a}{b}{c} = 
\begin{cases} 
\Redei{a}{b}{c}_{\textup{R}} &\textup{ if at least one of } a, b, c \textup{ is } 1 \bmod 8 \\
0 &\textup{ otherwise.} 
\end{cases}
\]
If $\Redei{a}{b}{c}_{\textup{R}}$ is not defined, $\Redei{a}{b}{c} = 0$.
\end{proposition}

\begin{proof}
Given a nonzero integer $d$, take $\chi_d$ to be the image of $d$ under the Kummer map to $H^1(G_{\QQ}, \mu_2)$.

By \cite[(1) and (2)]{Stevenhagen}, the symbol $\Redei{a}{b}{c}_\text{R}$ is defined if and only if the associated Hilbert symbols satisfy
\[
(a, b)_v = (a, c)_v = (b, c)_v = 1 \quad \text{for all rational places } v
\]
and if no rational place ramifies in all three of $\Q(\sqrt{a})/\Q$, $\Q(\sqrt{b})/\Q$, and $\Q(\sqrt{c})/\Q$. The condition on Hilbert symbols is equivalent to the second part of the symbol condition holding for all places $v$.

Given $v \ne 2$, if $\Q(\sqrt{a})/\Q$ is ramified at $v$ but $\Q(\sqrt{b})/\Q$ is not, the condition on Hilbert symbols forces the latter extension to split at $v$. Adjusting this observation as necessary for the other symbol arguments, we find that the symbol condition holds at all bad $v$ except potentially $2$ if $\Redei{a}{b}{c}_\text{R}$ is defined.

The first part of the symbol condition at $2$ amounts to requiring that one of $a, b, c$ equals $1$ mod $8$. To summarize, we have
\[
\Redei{a}{b}{c} \ne 0 \iff \Redei{a}{b}{c}_\text{R} \text{ is defined and at least one of } a, b, c \text{ is } 1 \bmod 8. 
\]
Now suppose that $\Redei{a}{b}{c}$ is nonzero. We may assume that none of $a, b, c$ are equal to $1$, as linearity implies that both $\Redei{a}{b}{c}$ and $\Redei{a}{b}{c}_\text{R}$ would otherwise equal $1$.

Using reciprocity for $\Redei{a}{b}{c}$ and $\Redei{a}{b}{c}_\text{R}$ \cite[Theorem 1.1]{Stevenhagen}, we may assume $c$ is $1$ mod $8$. Take $F_{a,b}$ to be one of the minimally ramified quartic extensions of $\Q(\sqrt{ab})/\Q$ defined in \cite[Definition 7.6]{Stevenhagen}. Take $\sigma$ to be a generator of $\Gal(F_{a, b}/\Q(\sqrt{ab}))$, and take $\tau$ to be an element in $\Gal(F_{a, b}/\Q)$ that generates $\Gal(\Q(\sqrt{a}, \sqrt{b})/\Q(\sqrt{a}))$. If this latter Galois group is trivial, we take $\tau =1$. The group $\Gal(F_{a, b}/\Q)$ is either dihedral of  order $8$ or cyclic of order $4$.

We define a map $\gamma: \Gal(F_{a, b}/\Q) \to \pm 1$ by
\[
\gamma(\sigma^k) = \gamma(\tau \sigma^k) = 
\begin{cases} 
+1 &\text{ if } k \equiv 0, 1 \bmod 4 \\ 
-1 &\text{ if } k \equiv 2, 3 \bmod 4
\end{cases}
\]
This map satisfies $\gamma(\sigma^2 \alpha) = - \gamma(\alpha)$ for all $\alpha \in G_{\QQ}$, so $d\gamma$ factors through $\Gal(\Q(\sqrt{a}, \sqrt{b})/\Q)$. It may then be calculated directly that $d \gamma = \chi_a \cup \chi_b$; this calculation also appears in the proof of \cite[Proposition 2.1]{Smi16}. In Definition \ref{def:Redei}, we then may take $\epsilon = \gamma \cup \chi_c$. The definition of the symbol and \cite[Section XIV.1.3]{Serre} give
\[
\Redei{a}{b}{c} = \exp\left(2 \pi i \sum_{v \in S} \inv_v (\res_{G_v}\left(\gamma \cup\chi_c)\right)\right) = \prod_{v \mid c} \gamma(\Frob\, v),
\]
where $S$ is the bad set of primes for the symbol, and where we say $\infty \mid c$ if $c$ is negative. This agrees with the definition of the R\'{e}dei symbol.
\end{proof}

Besides the R\'edei symbol, our definition of a generalized R\'edei symbol was influenced by the definition of the Poitou--Tate pairing and the Cassels--Tate pairing. We cover these relationships next.

\begin{example}
Take $M_1, M_2, M_3$ and $\phi_1, \phi_2, \phi_3$ as in Definition \ref{def:Redei}. We suppose that 
\[
\phi_2 \cup \phi_3 \in \Sha^2(G_F, M_1^{\vee}) \quad \text{and} \quad \phi_1\in\Sha^1(G_F, M_1),
\]
where the cup product is defined from the homomorphism $M_2 \otimes M_3 \to M_1^{\vee}$ coming from \eqref{eq:M1M2M3_pairing}.

The Poitou--Tate pairing is a perfect pairing
\[
\langle\,\,,\,\, \rangle:\,\Sha^1(G_F, M_1) \times \Sha^2(G_F, M_1^{\vee}) \to \Q/\Z
\]
first defined by Poitou \cite{Poitou} and Tate \cite[Theorem 3.1]{Tate}. With Tate's sign convention, the identity
$$
\Redei{\phi_1}{\phi_2}{\phi_3} = \exp\left(-2\pi i \langle \phi_1, \phi_2 \cup \phi_3 \rangle\right)
$$
follows from the definitions.
\end{example}

\begin{example}
Take $M_1, M_2, M_3$ and $\phi_1, \phi_2, \phi_3$ as in Definition \ref{def:Redei}, and choose a pairing \eqref{eq:M1M2M3_pairing}. We suppose the associated symbol $\Redei{\phi_1}{\phi_2}{\phi_3}$ is nonzero.

The pairing corresponds to a $G_F$-homomorphism $M_1 \to \Hom(M_2, M_3^{\vee})$, and we take $\chi$ to be a cocycle representing the image of $\phi_1$ in $Z^1(G_F, \Hom(M_2, M_3^{\vee}))$. We then may construct an exact sequence of $G_F$-modules
\begin{equation}
\label{eq:chi_ext}
0 \to M_3^{\vee} \xrightarrow{\iota} M \xrightarrow{\pi} M_2 \to 0
\end{equation}
such that there is a linear section $s: M_2 \to M$ of $\pi$ satisfying
\[\sigma s(\sigma^{-1}x) - s(x) \,=\, \iota\big(\chi(\sigma)(x)\big)\quad\text{for all }\, x \in M_2 \text{ and }\sigma \in G_F.\]
An equivalent abstract approach to constructing this extension of $G_F$-modules is to take the image of $\chi$ in $\text{Ext}^1_{G_F}(M_2, M_3^{\vee})$ under an edge map corresponding to a certain spectral sequence \cite[Example 0.8]{Milne ADT}.

For each place $v$ of $F$, we define a subspace of local conditions $\mathscr{W}_v \subseteq H^1(G_v, M)$ as follows:
\begin{itemize}
\item If $v$ is outside $S$, take $\mathscr{W}_v$ to be the set of unramified cocylce classes.
\item Otherwise, if $\phi_2$ is trivial at $G_v$, take $\mathscr{W}_v = 0$.
\item Otherwise, if $\phi_3$ is trivial at $G_v$, take $\mathscr{W}_v = H^1(G_v, M)$.
\item Otherwise, $\phi_1$ must be trivial at $v$, so we may choose $h_v$ in the image of $M_1$ inside $\Hom(M_2, M_3^{\vee})$ such that the section 
\[s_v =s + \iota \circ h_v\]
is $G_v$-equivariant. We then take $\mathscr{W}_v = s_v(H^1(G_v, M_2))$.
\end{itemize}
We then assign $M_2$ and $M_3^{\vee}$ local conditions so \eqref{eq:chi_ext} is exact in the category of modules with local conditions of \cite{MS22a}. Taking $\text{CTP}$ to be the associated Cassels--Tate pairing as defined in \cite[Definition 3.2]{MS22a}, we may check that
\[\Redei{\phi_1}{\phi_2}{\phi_3} = \exp\big(-2\pi i\cdot\CTP(\phi_2,\, \phi_3)\big).\]
\end{example}
\section{Trilinear character sums}
\label{sec:trilinear}

\begin{notat}
\label{notat:trilinear_setup}
Choose a number field $F$, choose finite $G_F$ modules $M_1$, $M_2$, and $M_3$, and choose an equivariant homomorphism 
\[
P: M_1 \otimes M_2 \otimes M_3 \to \overline{F}^{\times}
\]
as in \eqref{eq:M1M2M3_pairing}.

Choose subsets $M_i^* \subseteq M_i$ for $i \le 3$ that are closed under $G_F$ and under multiplication by $\widehat{\Z}^{\times}$. We assume that, for any $(m_1, m_2, m_3)$ in $M_1^* \times M_2^* \times M_3^*$, there are $\tau_1, \tau_2, \tau_3 \in G_F$ with
\begin{equation}
\label{eq:rough_ramification}
P(\tau_1 m_1, \tau_2m_2, \tau_3 m_3) \ne 1.
\end{equation}
Choose a finite set of places $\Vplac_0$ of $F$ including all archimedean places, all places dividing some $\# M_i$ and all places ramified in the field of definition of some $M_i$. Given $i \le 3$ and $\phi \in H^1(G_F, M_i)$, we say $\phi$ has \emph{acceptable ramification} if, for every prime $\mfp$ of $F$ outside $\Vplac_0$, and given any prime $\ovp$ of $\overline{F}$ dividing $\mfp$, the ramification-measuring homomorphism \cite[Definition 3.4]{Smi22a}
\[
\mathfrak{R}_{\ovp, M_i}: H^1(G_F, M_i) \to M_i(-1)^{G_{\mfp}}
\] 
corresponding to $\ovp$ maps $\phi$ either to $0$ or to an element in $M_i^*(-1)$, where this last object is the subset of the Tate twist $M_i(-1)$ corresponding to the subset $M_i^*$ of $M_i$.

Given $\phi$ in $H^1(G_F, M_i)$, we take $\mathfrak{f}(\phi)$ to be the product of all primes of $F$ outside $\Vplac_0$ where $\phi$ is ramified, and we take $H(\phi)$ to be the rational norm of this ideal. 

We take $o_i = \# M_i$ for $i \le 3$.
\end{notat}

\begin{theorem}
\label{thm:trilinear}
Take all notation as in Notation \ref{notat:trilinear_setup}. Choose positive numbers $H_1, H_2, H_3 > 3$. For $i \le 3$, take $X_i$ to be the set of cocycle classes $\phi$ in $H^1(G_F,M_i)$ with acceptable ramification that satisfy $H(\phi) \le H_i$.

For $(i, j) \in \{(1, 2), (1, 3), (2, 3)\}$, let $a_{ij}$ be a function
\[
a_{ij}: X_i \times X_j \to \{z \in \C\,:\,\, |z| \le 1\}.
\]
Then
\begin{equation}
\label{eq:trilinear_bound}
\left| \sum_{\phi_1 \in X_1} \sum_{\phi_2 \in X_2} \sum_{\phi_3 \in X_3} a_{12}(\phi_1, \phi_2) a_{13}(\phi_1, \phi_3) a_{23}(\phi_2, \phi_3)\Redei{\phi_1}{\phi_2}{\phi_3}\right|
\end{equation}
\[
\ll H_1H_2H_3 \cdot (\log H_1H_2H_3)^C \cdot \left(H_1^{-c} + H_2^{-c} + H_3^{-c}\right),
\]
with
\[
C = 7[K:\Q]^3 \cdot \max(o_1, o_2, o_3)^{9} \quad\textup{and}\quad c = \frac{1}{64 [K:F]^3 \cdot \max(o_1, o_2, o_3)^7}, 
\]
where the implicit constant depends only on $F$, $\Vplac_0$, and the $M_i$.
\end{theorem}

The proof of this result starts with a sequence of reductions performed over the next few subsections.

\subsection{Applying Cauchy--Schwarz}
Our first goal will be to remove the terms $a(\phi_i, \phi_j)$ appearing in the sum. In a manner familiar to those who have studied the large sieve \cite[Ch. 7]{IwaniecKowalski}, we do this by applying Cauchy--Schwarz to our trilinear sum and expanding the resulting squares of sums as sums of individual terms. This starts with the following lemma, which is necessary because the generalized R\'edei symbol is incompletely trilinear.

\begin{lemma}
\label{lem:trilinear_positivity}
Take all notation as in Notation \ref{notat:trilinear_setup}. Choose $\phi_1 \in H^1(G_F, M_1)$ and $\phi_2 \in H^1(G_F, M_2)$, choose a finite subset $Y$ of $H^1(G_F, M_3)$, and choose a function $a: Y \to \C$. Then

\begin{equation}
\label{eq:trilinear_positivity}
\left|\sum_{\phi \in Y} a(\phi) \cdot \Redei{\phi_1}{\phi_2}{\phi}\right|^2 \le \sum_{\phi, \phi' \in Y}  a(\phi) \cdot \overline{a(\phi')} \cdot \Redei{\phi_1}{\phi_2}{\phi - \phi'}.
\end{equation}

\end{lemma}
\begin{proof}
If $\Redei{\phi_1}{\phi_2}{0}$ is $0$, then both sides of \eqref{eq:trilinear_positivity} are equal to $0$. So we suppose this symbol is nonzero.

Define a relation on $Y$ by taking $\phi \sim \phi'$ if $\Redei{\phi_1}{\phi_2}{\phi - \phi'}$ is nonzero. This is an equivalence relation, with transitivity and symmetry following from Proposition \ref{prop:rec} and reflexivity following from our above assumption. The right hand side of \eqref{eq:trilinear_positivity} then equals
\[
\sum_{V \in  Y/\sim} \,\,\sum_{\phi, \phi' \in V} a(\phi) \cdot \overline{a(\phi')} \cdot \Redei{\phi_1}{\phi_2}{\phi - \phi'} = \sum_{V \in Y/\sim} \left| \sum_{\phi \in V} a(\phi) \Redei{\phi_1}{\phi_2}{\phi - \phi_V}\right|^2  
\]
by linearity of the generalized R\'edei symbol, where $Y/\sim$ denotes the set of equivalence classes of $Y$ considered as a subset of the power set of $Y$ and $\phi_V$ is any representative of $V$ for each $V$. Take $V_0$ to be the set of $\phi \in Y$ such that $\Redei{\phi_1}{\phi_2}{\phi}$ is nonzero. This is either the empty set or an element of $Y/\sim$. The left hand side of \eqref{eq:trilinear_positivity} is
\[
\left|\sum_{\phi \in V_0} a(\phi)\cdot  \Redei{\phi_1}{\phi_2}{\phi}\right|^2,
\]
and the result follows from
\[
\sum_{\substack{V \in Y/\sim\\ V \ne V_0}} \left| \sum_{\phi \in V} a(\phi) \Redei{\phi_1}{\phi_2}{\phi - \phi_V}\right|^2 \ge 0.
\]
\end{proof}

\begin{proposition}
\label{prop:CS}
In the situation of Theorem \ref{thm:trilinear}, we have
\begin{align}
\nonumber
&\left| \sum_{\phi_1 \in X_1} \sum_{\phi_2 \in X_2} \sum_{\phi_3 \in X_3} a_{12}(\phi_1, \phi_2) a_{13}(\phi_1, \phi_3) a_{23}(\phi_2, \phi_3)\Redei{\phi_1}{\phi_2}{\phi_3}\right| \\
&\qquad\le |X_1| ^{3/4} \cdot  | X_2|^{1/2} \cdot |X_3|^{1/2} \cdot S^{1/4},
\label{eq:CS_full}
\end{align}
where we have taken
\[
S = \sum_{\phi_2, \phi'_2 \in X_2}  \sum_{\phi_3, \phi'_3 \in X_3}\left|\sum_{\phi_1 \in X_1} \Redei{\phi_1}{\phi_2 - \phi'_2}{\phi_3 - \phi'_3}  \right|.
\]
\end{proposition}

\begin{proof}
The left hand side of \eqref{eq:CS_full} is at most
\[\sum_{\phi_1 \in X_1} \sum_{\phi_2 \in X_2}  \left| \sum_{\phi_3 \in X_3} a_{13}(\phi_1, \phi_3) a_{23}(\phi_2, \phi_3) \cdot \Redei{\phi_1}{\phi_2}{\phi_3} \right|.\]
By the Cauchy--Schwarz inequality, this is at most $S_1^{1/2}\cdot S_2^{1/2}$, where
\begin{align*}
&S_1 = \sum_{\phi_1 \in X_1} \sum_{\phi_2 \in X_2}  1     
\qquad\text{and}\\
&S_2 = \sum_{\phi_1 \in X_1} \sum_{\phi_2 \in X_2} \left| \sum_{\phi_3 \in X_3} a_{13}(\phi_1, \phi_3) a_{23}(\phi_2, \phi_3) \cdot \Redei{\phi_1}{\phi_2}{\phi_3} \right|^2.
\end{align*}
By Lemma \ref{lem:trilinear_positivity}, we have
\[S_2 \le \sum_{\phi_1 \in X_1} \sum_{\phi_2 \in X_2} \sum_{\phi_3, \phi'_3 \in X_3}  a_{13}(\phi_1, \phi_3)a_{23}(\phi_2, \phi_3) \cdot  \overline{a_{13}(\phi_1, \phi_3')a_{23}(\phi_2, \phi_3')}  \cdot \Redei{\phi_1}{\phi_2}{\phi_3 - \phi'_3}.\]
Again applying Cauchy--Schwarz, we find that $S_2 \le S_3^{1/2} \cdot S_4^{1/2}$ with
\begin{align*}
&S_3 = \sum_{\phi_1 \in X_1}\sum_{\phi_3, \phi'_3 \in X_3} 1 \qquad \text{and}\\
&S_4 = \sum_{\phi_1 \in X_1}\sum_{\phi_3, \phi'_3 \in X_3} \left| \sum_{\phi_2 \in X_2}   a_{23}(\phi_2, \phi_3)\cdot  \overline{a_{23}(\phi_2, \phi_3')} \cdot \Redei{\phi_1}{\phi_2}{\phi_3 - \phi'_3} \right|^2.
\end{align*}
Another application of Lemma \ref{lem:trilinear_positivity} gives that $S_4 \le S$, and the result follows.
\end{proof}

\subsection{Towards an averaged sum of multiplicative functions}
In the context of Theorem \ref{thm:trilinear}, we will take $e_0$ to be exponent of $M_1 \oplus M_2 \oplus M_3$ considered as an abelian group. We will then take $K$ to be the minimal Galois extension of $F$ containing $\mu_{e_0}$ such that each $M_i$ is a $\Gal(K/F)$ module.

\begin{lemma}
\label{lem:basic_assumptions}
Suppose that Theorem \ref{thm:trilinear} holds subject to the additional assumptions that $H_1 \ge H_2 \ge H_3$ and that $(K/F, \Vplac_0, e_0)$ is an unpacked tuple in the sense of  \cite[Definition 3.4]{Smi22a}. Then the theorem holds without these additional assumptions. 
\end{lemma}

\begin{proof}
Take $M_1, M_2, M_3$, $P$, and $(K/F, \Vplac_0, e_0)$ as in Notation \ref{notat:trilinear_setup}, with $K$ and $e_0$ defined as above. We wish to prove that Theorem \ref{thm:trilinear} holds for these objects conditionally on the assumptions of the lemma.

Take $\Vplac$ to be a minimal set of places of $F$ such that $(K/F, \Vplac, e_0)$ is unpacked and such that, if a given finite prime $\mfp$ lies in $\Vplac$, then all finite primes of norm at most $N_F(\mfp)$ lie in $\Vplac$. This condition uniquely determines $\Vplac$.
We take $X'_i$ to be the set of $\phi$ with acceptable ramification that satisfy $H(\phi) \le H_i$ with respect to the set of places $\Vplac_0 \cup \Vplac$. By taking
\[a'_{ij}(\phi_i, \phi_j) = \begin{cases} a_{ij}(\phi_i, \phi_j) &\text{ if }\phi_i \in X_i\text{ and } \phi_j \in X_j \\ 0&\text{ otherwise,}\end{cases}\]
we find that \eqref{eq:trilinear_bound} equals
 \[\left| \sum_{\phi_1 \in X_1'} \sum_{\phi_2 \in X_2'} \sum_{\phi_3 \in X_3'} a'_{12}(\phi_1, \phi_2) a'_{13}(\phi_1, \phi_3) a'_{23}(\phi_2, \phi_3)\Redei{\phi_1}{\phi_2}{\phi_3}\right|.\]
By the assumptions of the lemma, Theorem \ref{thm:trilinear} gives an estimate for this sum under the condition that $H_1 \ge H_2 \ge H_3$. Since $\Vplac$ was uniquely determined, the implicit constant of this estimate depends only on $F$, the $M_i$, and $\Vplac_0$. By permuting the indices $1,2,3$ and applying Proposition \ref{prop:rec} as necessary, we see that this estimate also holds outside the case that $H_1 \ge H_2 \ge H_3$.
\end{proof}

\begin{mydef}
\label{defn:lift_ramfact}
Take all notation as in Theorem \ref{thm:trilinear} and $K$ and $e_0$ as above. We assume $(K/F, \Vplac_0, e_0)$ is unpacked.

Given a prime $\mfp$ of $F$ outside $\Vplac_0$, we call a given homomorphism $\chi_{\mfp}: H^1(G_F, M_1) \to \mathbb{C}^{\times}$ a \emph{ramification factor for $\mfp$} if it factors through the ramification-measuring homomorphism $\mathfrak{R}_{\ovp, M_1}$ defined in \cite[Section 3]{Smi22a} for some prime $\ovp$ of $\overline{F}$ dividing $\mfp$. Given a squarefree ideal $\mfh$ of $F$ indivisible by any prime in $\Vplac_0$, we call $\chi: H^1(G_F, M_1) \to \mathbb{C}^{\times}$ a \emph{ramification factor for $\mfh$} if it is a product of ramification factors for the primes dividing $\mfh$.

Choose $\phi_2 \in H^1(G_F, M_2)$ and $\phi_3 \in H^1(G_F, M_3)$, and choose cocycles $\overline{\phi_2}$ and $\overline{\phi_3}$ representing these classes. Take $\mff_2 = \mff(\phi_2)$ and $\mff_3 = \mff(\phi_3)$. Take $E$ to be the maximal extension of $F$ ramified only at places in $\Vplac_0$ and primes dividing $\mff_2 \cdot \mff_3$. We then call
\[g \in C^1(\Gal(E/F), M_1^{\vee})\]
a \emph{lift} for $(\phi_2, \phi_3)$ if
\[dg = \overline{\phi_2} \cup \overline{\phi_3}.\]
Given $\phi_1 \in H^1(G_F, M_1)$, take $\Vplac(\phi_1)$ to be the set of primes dividing $\mff(\phi_1)$ that do not divide $\mff_2 \cdot \mff_3$. For $v$ in this set, take $\symb{\phi_1}{g}_v$ to be $0$ if neither $\phi_2$ nor $\phi_3$ is trivial at $v$, or if
\[\res_{G_v}(\phi_1 \cup \phi_3) \ne 0 \,\text{ in }\,  H^2(G_v, M_2^{\vee}) \quad\text{or}\quad \res_{G_v}(\phi_1 \cup \phi_2) \ne 0 \,\text{ in } \, H^2(G_v, M_3^{\vee}).\]
Otherwise, we may choose $x \in M_i$ such that $\res_{G_v} \overline{\phi_i} = dx$ either for $i =  2$ or for $i = 3$. Taking $f_v$ to be $x \cup \res_{G_v}(\overline{\phi_3})$ in the former case and $- \res_{G_v}(\overline{\phi_2}) \cup x$ in the latter case, we then set
\[\symb{\phi_1}{g}_v = \exp\left( 2\pi i \cdot \inv_v( \phi_1 \cup \left(f_v - \text{res}_{G_v} g)\right)\right).\]
We then define
\[\symb{\phi_1}{g} = \prod_{v \in \Vplac(\phi_1)} \symb{\phi_1}{g}_v.\]
\end{mydef}
Mimicking the proof of Lemma \ref{lem:symbol_well_defined} shows that $[\phi_1, g]$ does not depend on the implicit choice of local embeddings $\overline{F} \hookrightarrow \overline{F_v}$ for $v \in \Vplac(\phi_1)$.

\begin{lemma}
\label{lem:to_multiplicative}
Suppose we are in the situation of Definition \ref{defn:lift_ramfact}. Then either
\begin{equation}
\label{eq:trivial_line}
\sum_{\phi_1 \in X_1} \Redei{\phi_1}{\phi_2}{\phi_3} = 0
\end{equation}
or there is a lift $g$ of $(\phi_2, \phi_3)$ and a ramification factor $\chi$  for $\mff_2 \cdot \mff_3$ such that
\[
\left|\sum_{\phi_1 \in X_1} \Redei{\phi_1}{\phi_2}{\phi_3} \right| \,\le \,\left|\sum_{\phi_1\in X_1} \chi(\phi_1) \cdot \symb{\phi_1}{g}\right|.
\]
\end{lemma}

\begin{proof}
Take $\Vplac$ to be the set of places in $\Vplac_0$ together with the primes dividing $\mff_2 \cdot \mff_3$.

If $\Redei{0}{\phi_2}{\phi_3} = 0$, then $\Redei{\phi_1}{\phi_2}{\phi_3} = 0$ for all $\phi_1$ and the lemma is vacuous. So we may assume that $\Redei{0}{\phi_2}{\phi_3}$ is nonzero. From the symbol conditions, $\phi_2 \cup \phi_3$ then lies in
\[
\ker\left(H^2(\Gal(E/F), \, M_1^{\vee}) \to \prod_{v \in \Vplac} H^2(G_v, M_1^{\vee})\right).
\]
Under Poitou--Tate duality, this group is dual to
\[
\ker\left(H^1(G_F, M_1) \to \prod_{v \in \Vplac} H^1(G_v, M_1) \times \prod_{v \not \in \Vplac} H^1(I_v, M_1)\right),
\]
which equals $\Sha^1(G_F, M_1)$ since $(K/F, \Vplac_0, e_0)$ is unpacked. So, if $\phi_2 \cup \phi_3$ is nontrivial in $H^2(\Gal(E/F), M_1^{\vee})$, it follows that
\[
\sum_{\phi \in \Sha^1(G_F, M_1)} \Redei{\phi}{\phi_2}{\phi_3} = 0,
\]
and \eqref{eq:trivial_line} follows from Proposition \ref{prop:rec} and the fact that $X_1$ is closed under addition by $\Sha^1(G_F, M_1)$. So we may assume $\phi_2 \cup \phi_3$ is trivial. Equivalently, we may assume that there is some lift $g$ of $(\phi_2, \phi_3)$.

Take $\Vplac_{\max}$ to be the union of the set $\Vplac$ and the set of primes of $F$ of norm at most $H_1$. For $v \in \Vplac_{\max}$, define a subset $\mathscr{L}_v$ of $H^1(G_v, M_1^{\vee})$ as follows:

\begin{itemize}
\item If $\phi_2$ is trivial at $v$, choose $x \in M_2$ so $\res_{G_v} \overline{\phi_2} = dx$, take 
\[h_{0v} = x \cup \res_{G_v} \overline{\phi_3} - \res_{G_v} g,\] and take $\mathscr{L}_v $ to be the coset $h_{0v} + \mathscr{L}_{0v}$ with
\[\mathscr{L}_{0v} = H^0(G_v, M_2) \cup \res_{G_v}\phi_3.\]
\item Otherwise, if $\phi_3$ is trivial at $v$, choose $x \in M_3$ so $\res_{G_v} \overline{\phi_3} = dx$, take
\[h_{0v} = - \res_{G_v} \overline{\phi_2} \cup x - \res_{G_v} g,\]
and take $\mathscr{L}_v = h_{0v} + \mathscr{L}_{0v}$ with 
\[\mathscr{L}_{0v} =\res_{G_v} \phi_2 \cup H^0(G_v, M_3). \]
\item Otherwise, if $v$ is not a bad prime for $\Redei{0}{\phi_2}{\phi_3}$, take $\mathscr{L}_v = H^1(G_v/I_v, M_1^{\vee})$.
\item Otherwise, take $\mathscr{L}_v = H^1(G_v, M_1^{\vee})$.
\end{itemize}
We then take $\mathscr{L} = \prod_{v \in \Vplac_{\max}} \mathscr{L}_v$.
Given $(h_v)_v \in \mathscr{L}$ and $\phi_1 \in X_1$, we then define a symbol
\[
\symb{\phi_1}{(h_v)_v} = 
\begin{cases} 
0 &\text{ if } \symb{\phi_1}{g} = 0 \\ 
\exp\left(2\pi i \cdot \sum_{v \in \Vplac_{\max}} \inv_v(\phi_1 \cup h_v)\right)&\text{ otherwise.}
\end{cases} 
\]
We note that the zeroness of the symbol $\symb{\phi_1}{(h_v)_v}$ does not depend on $g$, as the condition $\symb{\phi_1}{g} = 0$ depends only on $\phi_1, \phi_2, \phi_3$ but not on the choice of $g$. We claim that
\begin{equation}
\label{eq:averaged_Redei}
\Redei{\phi_1}{\phi_2}{\phi_3} = \frac{1}{\# \mathscr{L}} \sum_{(h_v)_v \in \mathscr{L}} \symb{\phi_1}{(h_v)_v}.
\end{equation}
This claim splits into cases depending on whether the left hand side is nonzero or not. If it is nonzero, it is straightforward to show that the choice of $(h_v)_v \in \mathscr{L}$ does not affect $\symb{\phi_1}{(h_v)_v}$ from the symbol conditions, and the equality follows from the definition of the generalized R\'edei symbol.

If it is zero, then the symbol condition fails for some $v \in \Vplac_{\max}$. This means that either
\begin{itemize}
    \item $\phi_1$, $\phi_2$, and $\phi_3$ are all nontrivial at $v$, and either $\phi_1$ is ramified at $v$ or $v$ is a bad prime for $[0, \phi_2, \phi_3]$; or
    \item $\phi_2$ is trivial at $v$ but $\phi_1 \cup \phi_3$ is non-trivial in $H^2(G_v, M_2^\vee)$; or
    \item $\phi_3$ is trivial at $v$ but $\phi_1 \cup \phi_2$ is non-trivial in $H^2(G_v, M_3^\vee)$.
\end{itemize}
In each of these cases, we find that \eqref{eq:averaged_Redei} has right hand side $0$ from the definition of $\mathscr{L}_v$. This gives \eqref{eq:averaged_Redei}.

In particular,
\[
\left|\sum_{\phi_1 \in X_1} \Redei{\phi_1}{\phi_2}{\phi_3} \right|  \le \max_{(h_v)_v \in \mathscr{L}} \left| \sum_{\phi_1 \in X_1} \symb{\phi_1}{(h_v)_v}\right|.
\] 
Choose $(h_v)_v$ attaining the maximum on the right hand side.

Since $(K/F, \Vplac_0, e_0)$ is unpacked, we may choose $\psi \in H^1(G_F, M_1^{\vee})$ such that $h_{v} - \res_{G_v} \psi$ is unramified for $v$ outside $\Vplac_0$.  From the definition of the $h_v$, $\psi$ is unramified outside $\Vplac$.

Suppose that there is no such choice of $\psi$ such that 
\[
h_{v} - \res_{G_v} \psi = 0 \quad\text{ for all }v \in \Vplac_0.
\]
From Poitou--Tate duality, there then is some
\[
\phi \in W := \ker \left(H^1(G_F, M_1) \to \prod_{v \not\in \Vplac_0} H^1(I_v, M_1)\right)
\]
such that
\[
\sum_{v \in \Vplac_0} \inv_v(\phi \cup h_v) \ne 0.
\]
But then
\[
\sum_{\phi \in W} \symb{\phi + \phi_1}{(h_v)_v} = 0\,\text{ for } \,\phi_1 \in X_1, \quad\text{so}\quad \sum_{\phi_1 \in X_1} \symb{ \phi_1}{(h_v)_v} = 0, 
\]
and \eqref{eq:trivial_line} holds. So we may assume that we may choose $\psi$ so $h_v - \res_{G_v} \psi$ is trivial for $v \in \Vplac_0$.

We have $\sum_v \inv_v(\phi_1 \cup \psi) = 0$. As a result, we have
\[
\symb{\phi_1}{(h_v)_v} = \symb{\phi_1}{g - \psi} \cdot \exp\left(2\pi i \cdot \sum_{v \mid \mff_2 \cdot \mff_3} \inv_v(\phi_1 \cup (h_v - \res_{G_v} \psi))\right)
\]
for $\phi_1 \in X_1$.
Since $h_v - \res_{G_v} \psi$ is unramified for $v \mid \mff_2 \cdot \mff_3$, this final term takes the form of a ramification factor $\chi(\phi_1)$ for $\mff_2 \cdot \mff_3$. $g - \psi$ is a lift of $(\phi_2, \phi_3)$, and the result follows.
\end{proof}

\begin{mydef}
\label{defn:Cheb_mult}
In the context of Lemma \ref{lem:to_multiplicative}, choose a lift $g$ of $(\phi_2, \phi_3)$ and a ramification factor $\chi$ of $\mff_2 \cdot \mff_3$.

We define a map $\gamma: G_F \to \C$ as follows. To begin, choose $\sigma \in G_F$. If
\begin{align}
\label{eq:triviality_condition}
\res_{\langle \sigma \rangle }(\phi_2) \ne 0 \,\text{ in }\, H^1(\langle \sigma \rangle, M_2)\quad\text{and}\quad \res_{\langle \sigma \rangle }(\phi_3) \ne 0 \,\text{ in }\, H^1(\langle \sigma \rangle, M_3)
\end{align}
we take $\gamma(\sigma) = 0$.

Otherwise, choose $j$ either $2$ or $3$ and $x \in M_j$ such that $\overline{\phi_j}(\sigma) = (\sigma - 1 )x$. If $j = 2$, take $f(\sigma)$ in $M_1^{\vee}$ to be $x \cdot \overline{\phi_3}(\sigma)$. If $j = 3$, instead take $f(\sigma) = - \sigma x \cdot \overline{\phi_2}(\sigma)$. Here we use $\cdot$ to signify both the map $M_2 \otimes M_3 \rightarrow M_1^\vee$ and $M_3 \otimes M_2 \rightarrow M_1^\vee$. 

Let $\mu_\infty$ be the Galois module consisting of all the roots of unity inside $\overline{F}^\times$. Fix an identification between $\mu_\infty(-1)$ and $\Q/\Z$, so that the pairing $P$ from \eqref{eq:M1M2M3_pairing} becomes
\[
M_1(-1) \otimes M_2 \otimes M_3 \rightarrow \Q/\Z.
\]
Take $M'$ to be the set of $m$ in $M_1^*(-1)^{\langle \sigma \rangle}$ satisfying
\begin{align}
\label{eq:inertia_condition}
m \cdot \overline{\phi_j}(\sigma) \in (\sigma - 1) \cdot \Hom(M_{5 - j}, \QQ/\Z)\quad\text{for } j = 2,3,
\end{align}
and take
\[
\gamma(\sigma) = \sum_{m \in M'} \exp\left( -2\pi i \cdot m \cdot (f(\sigma) - g(\sigma))\right).
\]

Take $X_{10}$ to be the subset of cocycle classes in $H^1(G_F, M_1)$ with acceptable ramification, so $X_1$ is a subset of $X_{10}$. Given any ideal $\mfh$ of $F$, take $X_{10}(\mfh)$ to be the subset of $\phi_1$ in $X_{10}$ such that $\mfh = \mff(\phi_1)$. We then define
\[
G(\mfh) =  \frac{1}{\# X_{10}((1))} \sum_{\phi_1 \in X_{10}(\mfh)} \chi(\phi_1) \cdot \symb{\phi_1}{g}.
\]
\end{mydef}

\noindent Relating $G$ and $\gamma$ will require the following piece of bookkeeping.

\begin{lemma}
\label{lem:explicit_loc_Tate}
Take all notation as in Definition \ref{defn:lift_ramfact}, and choose a prime $\ovp$ of $\overline{F}$ not over a prime in $\Vplac_0$. Define the associated ramification section $\mathfrak{R}_{\ovp, M}$ as in \cite[Section 3]{Smi22a}. Then, given $\phi \in H^1(G_F, M)$ and $\psi \in H^1(G_F, M^{\vee})$ with $\psi$ unramified at $\ovp \cap F$,
we have
\begin{equation}
\label{eq:inv_formula}
\inv_{F \cap \ovp}(\phi \cup \psi) = -\mathfrak{R}_{\ovp, M}(\phi) \cdot \psi(\Frob_F \ovp),
\end{equation}
where the cup product is with respect to the evaluation pairing between $M$ and $M^{\vee}$, and where the product on the right is the evaluation pairing between $M(-1)$ and $M^{\vee}$.
\end{lemma}

\begin{remark}
Let us specify some of the implicit identifications in \eqref{eq:inv_formula}. Recall that $\mu_{e_0}$ are the roots of unity inside $\overline{F}^\times$, and also recall that $M^\vee = \Hom(M, \mu_{e_0}) = \Hom(M, \overline{F}^\times)$ and $M(-1) = \Hom(\hat{\Z}(1), M) = \Hom(\mu_{e_0}, M)$. Then the cup product on the left of \eqref{eq:inv_formula} naturally ends in $\Hom(\mu_{e_0}, \mu_{e_0})$. We identify $\Hom(\mu_{e_0}, \mu_{e_0})$ with $\Z/e_0\Z$, which we view as a subset of $\Q/\Z$ by sending $1$ to $1/e_0$.
\end{remark}

\begin{proof}
Take $F_{\mfp}$ to be the completion of $F$ at the place $\ovp$, and take $G_{\mfp} \subseteq G_F$ to be the associated absolute Galois group.  Take $m = \mathfrak{R}_{\ovp, M}(\phi)$. Then $m$ defines a $G_{\mfp}$-equivariant map from $\mu_{e_0}$ to $M$. If we take $\phi_0 \in H^1(G_{\mfp}, \mu_{e_0})$ to be associated with a uniformizer of $F_{\mfp}$, we find that $m(\phi_0) - \phi$ is unramified. Naturality of the cup product gives
\[
\inv_{F \cap \ovp}(\phi \cup \psi) = \inv_{F \cap \ovp}(\phi_0 \cup m^{\vee}(\psi)).
\]
The compatibility of the cup product with connecting maps \cite[Proposition 1.4.5]{Neuk07} gives that this equals $-\inv_{F \cap \ovp}(\pi \cup \delta m^{\vee}(\psi))$, where $\delta m^{\vee}(\psi) \in H^2(G_{\mfp}, \Z)$ is defined from the connecting map associated to the short exact sequence
\[
0 \to \Z \to \Z \to \Z/e_0 \Z \to 0,
\]
where $\pi$ is a uniformizer for $F_{\mfp}$, and where the cup product is as in \cite[XIV.1]{Serre}. But this is just $-m^{\vee}(\psi)(\Frob_F \ovp)$ \cite[XIV.1.3]{Serre}, and the result follows.
\end{proof}

\begin{lemma}
In the context of Definition \ref{defn:Cheb_mult}, suppose $G(\mfh)$ is nonzero for some ideal $\mfh$. Then $G$ is a multiplicative function on the integral ideals of $F$. 
Furthermore, this multiplicative function satisfies 
\begin{align*}
&G(\mfp) = \gamma( \Frob\, \mfp) \quad\textup{ for all primes } \mfp \textup{ outside } \Vplac_0 \cup \{\mfp \,:\,\, \mfp \mid \mff_2 \cdot \mff_3\},\\
&G(\mfp^k) = 0\quad\textup{for all primes } \mfp\textup{ of } F \textup{ and all } k \ge 2,\textup{ and }\\
&G(\mfp) = 0\quad\textup{for all primes } \mfp \in \Vplac_0.
\end{align*}
\end{lemma}

\begin{proof}
For any ideal $\mfh$ of $F$, we see that $X_{10}(\mfh)$ is closed under addition by elements of $X_{10}((1))$. Since $\phi_1 \mapsto \chi(\phi_1) \cdot \symb{\phi_1}{g}$ is a homomorphism, we find that $G$ is nonzero on some ideal if and only if it equals $1$ on $X_{10}((1))$. So we suppose this is the case.

Given coprime ideals $\mfa$, $\mfb$ of $H$, the map $(\phi_{1a}, \phi_{1b}) \mapsto \phi_{1a} + \phi_{1b}$ defines a surjection
\[
X_{10}(\mfa) \oplus X_{10}(\mfb) \to X_{10}(\mfa \cdot \mfb)
\]
since $(K/F, \Vplac_0, e_0)$ was assumed to be unpacked. The fibers of this map take the form
\[
\{(\phi_{1a} + \phi, \phi_{1b} - \phi)\,:\,\,\phi \in X_{10}((1))\}
\]
for some $(\phi_{1a}, \phi_{1b})$ in the domain, so
\[
G(\mfa \cdot \mfb) = \frac{1}{(\# X_{10}((1)))^2} \sum_{\phi_{1a} \in X_{10}(\mfa)} \sum_{\phi_{1b} \in X_{10}(\mfb)} \chi(\phi_{1a} + \phi_{1b}) \cdot \symb{\phi_{1a} + \phi_{1b}}{g}   = G(\mfa) \cdot G(\mfb). 
\]
That is, $G$ is multiplicative. That it is zero for primes in $\Vplac_0$ and non-squarefree ideals is clear from the definition of $\mff(\phi_1)$.

So it remains to calculate $G(\mfp)$ in the case that $\mfp$ is a prime outside $\Vplac_0$ that does not divide $\mff_2 \cdot \mff_3$. Choose a prime $\ovp$ of $\overline{F}$ dividing $\mfp$, and take $\sigma$ to be a Frobenius element corresponding to $\ovp$. Since the symbol $[\phi, g]$ was not effected by the implicit choice of local embeddings, we may assume that the choice of embedding $\overline{F} \hookrightarrow \overline{F_{\mfp}}$ corresponds to $\ovp$.

The ramification-measuring homomorphism $\mathfrak{R}_{\ovp, M_1}$ at $\ovp$ defines a surjection
\[
X_{10}(\mfp) \to M_1^*(-1)^{\langle \sigma \rangle}
\]
since $(K/F, \Vplac_0, e_0)$ is unpacked (\cite[Definition 3.4]{Smi22a}).

Take $M'$ to be the subset of $M_1^*(-1)^{\langle \sigma \rangle}$ defined in Definition \ref{defn:Cheb_mult}. If $\phi \in X_{10}(\mfp)$ satisfies $\mathfrak{R}_{\ovp, M_1}(\phi) \not \in M'$, we find that $\symb{\phi}{g}_{\mfp} = 0$. Indeed, equation \eqref{eq:inertia_condition} is equivalent to $m \cdot \overline{\phi_j}(\sigma)$ being orthogonal to $H^0(G_\mfp, M_{5 - j})$ (which are the $\langle \sigma \rangle$-fixed points of $M_{5 - j}$). Then local Tate duality and Lemma \ref{lem:explicit_loc_Tate} show that this condition is equivalent to $\phi \cup \phi_j$ being trivial in $H^2(G_\mfp, M_{5 - j}^\vee)$ for $j =2, 3$, giving the claim. 

Furthermore, if equation \eqref{eq:triviality_condition} holds, then we see that $\gamma(\sigma) = 0$ and $\symb{\phi}{g}_{\mfp} = 0$, because neither $\phi_2$ nor $\phi_3$ is locally trivial at $\mfp$.

Suppose now that $\phi$ maps to some $m \in M'$ and that \eqref{eq:triviality_condition} does not hold. We need to calculate $\chi(\phi) \cdot \symb{\phi}{ g}$. The term $\chi(\phi)$ is just $1$, 
so we are left calculating 
\[
\symb{\phi}{g} = \symb{\phi}{g}_{\mfp} = \exp\left(2\pi i\cdot \inv_{\mfp}(\phi \cup (f_{\mfp} - \res_{G_{\mfp}}g))\right) = \exp(-2\pi i \cdot m \cdot(f_{\mfp}(\sigma) - g(\sigma))),
\]
where the final equality is by Lemma \ref{lem:explicit_loc_Tate}.

The choice of $f_{\mfp}$ does not affect this expression since $m$ is invariant under $\sigma$. In particular, we may choose $f_{\mfp}$ so it satisfies $f(\sigma) = f_{\mfp}(\sigma)$. The result follows from taking the sum over $M'$.
\end{proof}

We now collect the properties we need of the map $\gamma$. For a prime $\mfp$ of $F$, we write $I_{\mfp}$ for the inertia group of one fixed prime $\mfq$ of $\overline{\Q}$ above $\mfp$. We shall only use this notation when the choice of $\mfq$ does not matter.

\begin{notat}
With all notation as in Definition \ref{defn:Cheb_mult}, and given distinct primes $\mfp_2, \mfp_3$ of $F$, we take $N(\mfp_2, \mfp_3)$ to be the minimal normal subgroup of $G_F$ containing
\[
\symb{\sigma_2\tau_2\sigma_2^{-1}}{\sigma_3\tau_3\sigma_3^{-1}}\quad\text{for all } \tau_2 \in I_{\mfp_2},\quad\tau_3 \in I_{\mfp_3},\quad\text{and}\quad \sigma_2, \sigma_3 \in G_F.
\]
We call this subgroup \emph{marked} if $\mfp_2 \nmid \mff_3$ and $\mfp_3 \nmid \mff_2$, and if we have
\[
\mathfrak{R}_{\ovp_2, M_2}(\phi_2) \in M_2^*(-1) \quad\text{and}\quad \mathfrak{R}_{\ovp_3, M_3}(\phi_3) \in M_3^*(-1)
\]
for some primes $\ovp_2, \ovp_3$ over $\mfp_2, \mfp_3$.
\end{notat}

\begin{lemma}
\label{lem:gamma_props}
In the context of Definition \ref{defn:Cheb_mult}, consider the minimal Galois extension $L/F$ containing $K$ such that $\phi_2$ and $\phi_3$ are trivial, and take $E/F$ to be the minimal Galois extension containing $L$ such that $\gamma$ is the inflation of a cochain on $\Gal(E/F)$. Then $\Gal(E/L)$ is abelian, and
\[
\#\Gal(E/L) \le \# M_1 \quad \quad \textup{and} \quad \quad \# \Gal(L/K) \le \# M_2 \cdot \# M_3.
\]
Furthermore, $N(\mfp_2, \mfp_3)$ is a subgroup of $G_L$; take $N$ to be its image in $\Gal(E/L)$. If $\mfp_2 \cdot \mfp_3$ does not divide $\mff_2 \cdot \mff_3$, then $N$ is trivial.

If instead $N(\mfp_2, \mfp_3)$ is marked, we have
\begin{equation}
\label{eq:dead_over_marked}
\sum_{\tau \in N} \gamma(\tau\sigma) = 0 \quad \quad \textup{for any } \sigma \in G_F.
\end{equation}
\end{lemma}

\begin{proof}
The cocycle class $\phi_k$ restricts over $G_K$ to a homomorphisms to $M_k$ for $k = 2, 3$ fixed by the natural action of $\Gal(K/F)$ on $H^1(G_K, M_k)$, so $\Gal(L/K)$ injects into $M_2 \times M_3$. The bound on $[L: K]$ follows.

Similarly, the map $g$ used to define $\gamma$ is seen to satisfy
\begin{equation}
\label{eq:g_E_L}
g(\tau \sigma) = g(\tau) + g(\sigma)\quad \text{and}\quad g(\sigma\tau) = \sigma g(\tau) + g(\sigma)  \quad\text{for } \tau \in G_L\quad\text{and }\,  \sigma \in G_F.
\end{equation}
Since $\gamma(\tau \sigma)$ is thus determined in this case by $\gamma(\sigma)$ and $g(\tau)$, we see that $g$ defines an injective homomorphism from  $\Gal(E/L)$ into $M_1^{\vee}$. So $\Gal(E/L)$ is an abelian group of order at most $\# M_1$. These relations also imply that $\Gal(E/L)$ is central in $\Gal(E/K)$.

Now choose primes $\mfp_2, \mfp_3$ outside $\Vplac_0$, and choose $\tau_k \in I_{\mfp_k}$ for $k = 2, 3$. Then $\tau_2$ and $\tau_3$ lie in $G_K$ since $K/F$ is unramified at $\mfp_2$ and $\mfp_3$. Therefore $\sigma_2 \tau_2 \sigma_2^{-1}$ and $\sigma_3 \tau_3 \sigma_3^{-1}$ also lie in $G_K$. Since $L/K$ is abelian, we find that $[\sigma_2 \tau_2 \sigma_2^{-1}, \sigma_3 \tau_3 \sigma_3^{-1}]$ lies in $G_L$.

If $\mfp_k$ does not divide $\mff_2 \cdot \mff_3$ for $k$ either $2$ or $3$, then $\sigma_k \tau_k \sigma_k^{-1}$ lies in $G_L$, and $[\sigma_2 \tau_2 \sigma_2^{-1}, \sigma_3 \tau_3 \sigma_3^{-1}]$ maps into $G_E$ since $\Gal(E/L)$ was central in $\Gal(E/K)$. In this case, we find that $N$ maps to $0$ in $\Gal(E/L)$.

If instead $N(\mfp_2, \mfp_3)$ is marked, choose $\tau_k$ topologically generating the tame quotient of $I_{\mfp_k}$ for $k = 2, 3$. Another cocycle calculation gives 
$$
g( \symb{\tau_2}{\tau_3}) = - \phi_2(\tau_2) \cdot \phi_3(\tau_3).  
$$
By conjugating, given any $\sigma_2$ and $\sigma_3$ in $G_F$, we find that there is some $\tau$ in $N(\mfp_2, \mfp_3)$ such that
\begin{align}
\label{eq:GammaNotZero}
g(\tau) = \sigma_2 \phi_2(\tau_2) \cdot \sigma_3 \phi_3(\tau_3).
\end{align}
Since $N$ is inside $G_L$, we note that $f( \tau\sigma ) = f(\sigma)$ for $\sigma \in G_F$ and $\tau \in N$, and the set $M'$ defined in Definition \ref{defn:Cheb_mult} is the same for $\sigma$ as for $\tau\sigma$. Furthermore, we also see that $\gamma$ is either identically zero on the coset $N \sigma $ or is always nonzero on this coset. Henceforth we assume that $\gamma$ is always nonzero on $ N\sigma$, as the lemma is trivial otherwise. But then equation \eqref{eq:GammaNotZero} and \eqref{eq:rough_ramification} show that, for any $m$ in $M'$, $m \cdot g(\tau)$ is nonzero for some choice of $\tau$ in $N$. The identity \eqref{eq:dead_over_marked} follows from the definition of $\gamma$ and \eqref{eq:g_E_L}.
\end{proof}

\subsection{The proof of Theorem \ref{thm:trilinear}}
To finish the proof of Theorem \ref{thm:trilinear}, we first prove a couple additional lemmas.

\begin{lemma}
\label{lem:monomial_decomp}
Take $G$ to be a finite group, take $N$ to be an abelian normal subgroup of $G$, and take $N_1, \dots, N_k$ to be normal subgroups of $G$ contained in $N$. Take $\phi: G \to \mathbb{\C}$ to be a class function satisfying
\[
\sum_{\tau \in N_i} \phi(\sigma \tau) = 0 \quad\textup{for all } \,\sigma \in G\,\textup{ and } \, i \le k.
\]
Then there is a finite sequence $H_1, \dots, H_m$ of subgroups of $G$ containing $N$ and, for $i \le m$, a choice of linear character $\psi_i$ of $H_i$ and a complex number $a_i$ such that
\[\phi = \sum_{i = 1}^m a_i \cdot \textup{Ind}^G_{H_i} \psi_i,\]
such that
\[[G: H_i] \cdot |a_i| \le \max_{g \in G} |\phi(g)| \quad\textup{for}\quad i \le m\quad\textup{and}\quad\sum_{i \le m} [G: H_i]^2 \le |G|^2 \cdot [G:N],\]
and such that $N_j$ is not a subgroup of $\ker \psi_i$ for any $j \le k$ or $i \le m$.
\end{lemma}

\begin{proof}
Take $\mathscr{H}$ to be set of subgroups $H$ of $G$ containing $N$ such that $H/N$ is cyclic. For $H$ in $\mathscr{H}$, take $H^*$ to be the set of $\sigma$ in $H$ such that $\sigma N$ is a generator for $H/N$, and define a function $\phi^*_H:G \to \C$ by
\[\phi^*_H(\sigma) = \begin{cases} \phi(\sigma) &\text{ if } \sigma \in H^* \\ 0 &\text{ otherwise.}\end{cases}\]
Then
\begin{equation}
\label{eq:H_decomp}
\phi = \sum_{H \in \mathscr{H}} \phi^*_H = \sum_{H \in \mathscr{H}} \frac{1}{[G: H]} \text{Ind}^G_H \phi^*_H.
\end{equation}
For $H$ in $\mathscr{H}$, the irreducible characters of $H$ are all induced from linear characters of some subgroup of $H$ containing $N$ \cite[Theorem 6.22]{Isaacs}. If an irreducible character $\chi$ of $H$ is induced from a proper subgroup of $H$, then $\chi$ is trivial on $H^*$ and the coefficient of $\chi$ in the decomposition of $\phi^*$ is $0$. If $\chi$ is a linear character, this coefficient is also zero if $\ker \chi$ contains any $N_j$. So
\[\phi^*_H = \sum_{i \le r} a_i \psi_i,\]
where the $\psi_i$ are linear characters of $H$ whose kernels contain no $N_j$. The $a_i$ have magnitude at most $\max_{g \in G} |\phi(g)|$, and $r$ is at most $|H|$ and $[G: H]^2 \cdot r$ is at most $|G| \cdot [G: N]$. The result now follows from \eqref{eq:H_decomp}.
\end{proof}

We record one simple estimate for multiplicative functions.

\begin{lemma}
\label{lem:simple_multiplicative_bound}
Choose a number field $F$ and real numbers $C_2 \ge C_1 \ge 1$. Choose a multiplicative function $f$ on the integral ideals of $F$ such that
\[
|f(\mfp)| \le C_1\quad\textup{and}\quad |f(\mfp^k)| \le C_2 \quad\textup{for } k > 1.
\]
Then, for $H > 3$, we have
\[
\left|\sum_{\substack{N_F(\mfh) \le H}} f(\mfh) \right| \ll H (\log H)^{C_1 - 1},
\]
where the implicit constant depends on $C_1$, $C_2$, and $F$.
\end{lemma}

\begin{proof}
This follows from the triangle inequality and Shiu's theorem \cite{Shiu} applied to the function $g$ defined by
\[
g(n) = \sum_{N_F(\mfh) = n} |f|(\mfh),
\]
where $|f|$ is the multiplicative function given by $|f|(\mfh) := |f(\mfh)|$.
\end{proof}

We now turn to proof of Theorem \ref{thm:trilinear}. By Lemma \ref{lem:basic_assumptions}, we may assume $(K/F, \Vplac_0, e_0)$ is unpacked and that $H_1 \ge H_2 \ge H_3$.

Now take $S_0$ to be the left hand side of \eqref{eq:trilinear_bound}. By Proposition \ref{prop:CS}, we have
\[
S_0 \le |X_1|^{3/4} \cdot |X_2|^{1/2} \cdot |X_3|^{1/2} \cdot S_1^{1/4}
\]
with
\[
S_1 =  \sum_{\phi_2, \phi'_2 \in X_2}  \sum_{\phi_3, \phi'_3 \in X_3}\left|\sum_{\phi_1 \in X_1} \Redei{\phi_1}{\phi_2 - \phi'_2}{\phi_3 - \phi'_3}  \right|.
\]
Take $Y_0$ to be the set of $y = (\phi_2, \phi'_2, \phi_3, \phi'_3)$ in $X_2 \times X_2 \times X_3 \times X_3$ such that
\[
\sum_{\phi_1 \in X_1} \Redei{\phi_1}{\phi_2 - \phi'_2}{\phi_3 - \phi'_3} \ne 0.
\]
If $y$ is in $Y_0$, we may use Lemma \ref{lem:to_multiplicative} to choose a ramification factor $\chi_y$ for $\mff(\phi_2  - \phi'_2) \cdot \mff(\phi_3 - \phi'_3)$ and a lift $g_y$ of $(\phi_2 - \phi'_2, \phi_3 - \phi'_3)$ such that
\[
\left|\sum_{\phi_1 \in X_1} \Redei{\phi_1}{\phi_2 - \phi'_2}{\phi_3 - \phi'_3} \right| \le \left|\sum_{\phi_1 \in X_1} \chi_y(\phi_1) \cdot \symb{\phi_1}{g_y} \right|.
\]
We take $G_y$ and $\gamma_y: \Gal(E_y/F) \to \C$ to be the functions defined from $g_y$ and $\chi_y$ as in Definition \ref{defn:Cheb_mult}, with $E_y$ defined as in Lemma \ref{lem:gamma_props}. We take $\Vplac_y$ to be the set of places in $\Vplac_0$ together with the primes dividing $\mff(\phi_2 - \phi'_2) \cdot \mff(\phi_3 - \phi'_3)$.

Define $Y$ to be the subset of $Y_0$ where $G_y$ is a multiplicative function. Then we have $S_1 \le \# X_{10}((1)) \cdot S_2$, where we have taken
\[
S_2 = \sum_{y \in Y}  \left| \sum_{N_F(\mfh) \le H_1} G_y(\mfh)\right|.
\]
For a given $y \in Y$, take $N_y$ to be the subgroup of $\sigma \in \Gal(E_y/K)$ such that $\phi_2(\sigma) = 0$. One may check that $\gamma_y(\sigma \tau) = \gamma_y(\tau \sigma)$ for $\sigma, \tau$ in this group, so we find it is abelian.

Given $y, z \in Y$, write $y \sim z$ if there is an isomorphism of groups $\Gal(E_{y}/F) \isoarrow \Gal(E_{z}/F)$ that identifies the maps $\gamma_{y}$ and $\gamma_z$, which identifies the collection of images of marked subgroups of $\Gal(E_y/F)$ with the collection of images of marked subgroups of $\Gal(E_z/F)$, and which identifies $N_y$ with $N_z$. Under this relationship, $Y$ splits into $\ll 1$ equivalence classes, where the implicit constant depends just on $F$, $M_1$, $M_2$, and $M_3$. Taking
\[
S_3 = \max_Z \sum_{y \in Z} \left| \sum_{N_F(\mfh) \le H_1} G_y(\mfh) \right| 
\]
where $Z$ varies over the equivalence classes in $Y$, we then have $S_1 \ll S_3$. Choose $Z$ for which this maximum is attained.

For $y \in Z$, we have
\[
\# \Gal(E_y/F) \le [K:F] \cdot o_1 \cdot o_2 \cdot o_3 \quad\text{and}\quad \# \Gal(E_y/F)/ N_y \le [K:F] \cdot o_2.
\]
By Lemma \ref{lem:monomial_decomp} and the definition of the equivalence class $Z$, there is a positive integer $k$, complex coefficients $a_1, \dots, a_k$, and a choice of a linear character $\psi_{iy}$ on a subgroup $H_{iy}$ of $\Gal(E_y/F)$ containing $N_y$ for each $i \le k$ and $y \in Z$ such that
\[
\gamma_y = \sum_{i \le k} a_i \cdot \text{Ind}^{\Gal(E_y/F)}_{H_{iy}} \psi_{iy},
\]
such that
\begin{equation}
\label{eq:sumsquaresbears}
\sum_{i \le k} [\Gal(E_y/F): H_{iy}]^2 \le [K:F]^3 \cdot  o_1^2 \cdot o_2^3 \cdot o_3^2, 
\end{equation}
such that $|a_i| \cdot [\Gal(E_y/F) : H_{iy}]$ is at most $o_1$ for all $i \le k$ and $y \in Z$, and such that the kernel of $\psi_{iy}$ contains no marked subgroup of $\Gal(E_y/F)$ for any $i \le k$ and $y \in Z$. 
 
For $i \le k$ and $y \in z$, we take the notation
\[
\mu_{iy} = \text{Ind}_{H_{iy}}^{\Gal(E_y/F)} \psi_{iy},
\]
and we define a multiplicative function $U_{iy}$ supported on the squarefree ideals of $F$ by
\[
U_{iy}(\mfp) = 
\begin{cases}
\mu_{iy}(\Frob\, \mfp) &\text{for }\, \mfp \not \in \Vplac_y \\ 
\gamma_y(1)^{-1} \cdot G_y(\mfp) \cdot \mu_{iy}(1) &\text{for }\, \mfp \in \Vplac_y.
\end{cases}
\]
We also take $A_i$ to be the multiplicative function supported on squarefree ideals that equals $a_i$ on all prime ideals.

Then, for any squarefree $\mfh$, we have
\[
G_y(\mfh) = \sum_{\mfa_1 \dots \mfa_k = \mfh} \prod_{i \le k} A_i(\mfa_i) U_{iy}(\mfa_i).
\]
We can then bound $S_3$ by an expression of the form
\begin{align*}
\sum_{\substack{N_F(\mfh) \leq H_1}} \left| \sum_{y \in Z} a(y) G_y(\mfh)\right|& \le \sum_{\substack{\mfa_1, \dots, \mfa_k  \\ N_F(\mfa_1 \dots \mfa_k) \le H_1}}\left| \sum_{y \in Z} a(y) \cdot \prod_{i \le k} A_i(\mfa_i) \cdot U_{iy}(\mfa_i)\right| \\
&\le \sum_{\substack{i_1, \dots, i_k \in \Z_{\ge 0} \\ i_1 + \dots + i_k \le \log H_1 + k}} \sum_{\substack{\mfa_1, \dots, \mfa_k \\ N_F(\mfa_j) \le e^{i_j}}} \left| \sum_{y \in Z} a(y) \cdot \prod_{i \le k} A_i(\mfa_i) \cdot U_{iy}(\mfa_i)\right|,
\end{align*}
where $a(y)$ is a complex number of magnitude $1$ for $y \in Z$. The final sum here is over at most $(\log H_1 +2k)^{k}$ choices of $i_1, \dots, i_k$. Choose $i_1, \dots, i_k$ where the associated summand is maximized and take $H_{1j} = e^{i_j}$.

There is some $j \le k$ where
\[
\log H_{1j}/ (\deg \mu_{jy})^2
\]
is maximized, where $y$ is arbitrary in $Z$. Permuting if necessary, we may assume $j = 1$, and we take $d = \deg \mu_{1y}$. Then
\[
\log H_{11} + \dots + \log H_{1k} \le (\deg \mu_{1y}^2 + \dots + \deg \mu_{ky}^2) \frac{\log H_{11}}{d^2}.
\]
Taking $H_{10} = e^k H_{1}/(H_{12} \dots H_{1k}) $, we then have $H_{10} \ge H_{11}$ and
$$
\log H_{10} \ge \frac{d^2 \cdot \log H_1}{[K:F]^3 \cdot o_1^2 \cdot o_2^3 \cdot o_3^2},
$$
where the denominator comes from \eqref{eq:sumsquaresbears}.



By Lemma \ref{lem:simple_multiplicative_bound} and the bounds on the $a_i$, we have
\[
\sum_{\substack{\mfa_2, \dots, \mfa_k \\ N_F(\mfa_j) \le H_{1j}\text{ for } j \ge 2}} \left|\prod_{i \le k} A_i(\mfa_i)  \cdot U_{iy}(\mfa_i)\right| \ll H_{12}\cdot \,\dots\, \cdot H_{1k} \cdot (\log H_1)^{k \cdot o_1}.
\]
The exponent of $\log H_1$ in this inequality is at most $[K : F]^2 \cdot o_1^2 \cdot o_2^2 \cdot o_3$. We find
\[
S_3 \ll H_1/H_{10} \cdot (\log H_1)^{2[K:F]^2 \cdot o_1^2 \cdot o_2^2 \cdot o_3}\cdot  \max_a \sum_{\substack{\mfa\\ N_F(\mfa) \le H_{10}}} \left| \sum_{y \in Z} a(y) \cdot A_1(\mfa) \cdot U_{1y}(\mfa)\right|,
\]
where the maximum is taken over all functions $a: Z \to \C$ of magnitude at most $1$. Choose $a$ maximizing this expression, and take
\[
S_4 = \sum_{\substack{\mfa\\ N_F(\mfa) \le H_{10}}} \left| \sum_{y \in Z} a(y) \cdot A_1(\mfa) \cdot U_{1y}(\mfa)\right|.
\]

We would like to apply \cite[Theorem 5.6]{LOSmith} to bound $S_4$. To do so, we need to estimate the size of the set $D \subseteq Z \times Z$ consisting of 
\[
(y, z) = \big((\phi_{20}, \phi_{20}', \phi_{30}, \phi_{30}'),\, (\phi_{21}, \phi_{21}', \phi_{31}, \phi_{31}')\big)
\]  
such that $\mu_{1y} \cdot \overline{\mu_{1z}}$  has nonzero average over $\Gal(E_{y}E_{z}/F)$. If $(y, z)$ lies in $D$, $\Gal(E_y/F)$ can contain a marked subgroup $H$ only if the image of $H$ in $\Gal(E_z/F)$ is nontrivial by Lemma \ref{lem:gamma_props}. In particular, if we take
\[
\mfh = \mff(\phi_{20}) \cdot \mff(\phi_{20}') \cdot \mff(\phi_{30}) \cdot \mff(\phi_{30}')\cdot \mff(\phi_{21}) \cdot \mff(\phi_{21}') \cdot \mff(\phi_{31}) \cdot \mff(\phi_{31}'),
\]
there either is no prime $\mfp_2$ such that
\[
\mfp_2 \mid \mff(\phi_{20}) \quad\text{and}\quad \mfp_2 \nmid \mfh\cdot \mff(\phi_{20})^{-1}
\]
or there is no prime $\mfp_3$ such that
\[
\mfp_3 \mid \mff(\phi_{30}) \quad\text{and}\quad \mfp_3 \nmid \mfh\cdot \mff(\phi_{30})^{-1}.
\]
Taking $\mfh_0$ to be the squarefree part of $\mfh$, we thus have
\[
N_F(\mfh_0) \le H_2^4 H_3^3.
\]
If we call this ideal the radical of $(y, z)$, then the number of elements of $D$ with this same radical is bounded by
\[
\ll (\# M_2 \cdot \# M_3)^{4 \omega(\mfh_0)}.
\]
So Lemma \ref{lem:simple_multiplicative_bound} gives
\[
\#D \ll H_2^4 H_3^3 \cdot (\log H_2)^{o_2^4 o_3^4}.
\]
A simpler application of Lemma \ref{lem:simple_multiplicative_bound} gives
\[
\# Z \ll H_2^2 H_3^2 \cdot (\log H_2)^{2o_2 + 2o_3}.
\]
To bound $S_4$, we will apply \cite[Theorem 5.6]{LOSmith} with 
\[
Q = C_1 H_2^{2d} H_3^{2d},\quad H = C_2 H_{10}, \quad t= 8 [K:F]^3 \cdot o_1^2 \cdot o_2^3 \cdot o_3^2,\quad\text{and}\quad b_{\mfa} = A_1(\mfa),
\]
where $C_1 \gg 1$ is a sufficiently large constant given $K/F$, $\Vplac_0$, and the $M_i$, and where $C_2 \gg 1$ is sufficiently large given this data and the choice of $C_1$. If $C_1$ is sufficiently large, we note that $\Delta_F^d \cdot N_{F/\QQ} \mff(\mu_{1y})$ is at most $Q$ for $y \in Z$, so the cited theorem may be applied to $G_y$.

With this choice of $t$ and $A_1$, we see that the function $G(\mfr, \mfb)$ defined in \cite[Theorem 5.6]{LOSmith} may be expressed in the form $H(\mfr \mfb)$ for a multiplicative function $H$ satisfying 
\[
|H(\mfp)| \le d t |a_1| \quad\text{and}\quad |H(\mfp^k)| \le C
\]
for all primes $\mfp$ of $F$ and all $k \ge 1$, where $C> 0$ does not depend on the $H_i$. Lemma \ref{lem:simple_multiplicative_bound} then gives that the $A_{0t}$ defined in \cite[Theorem 5.6]{LOSmith} satisfies
\[
A_{0t} \ll (\log H_1)^{d^2t^2 |a_1|^2}.
\]
Next,
\[
H_{10}^{-1/4} Q^{d/4t}\le \exp\left(\frac{-2d^2 \log H_1}{t} + \frac{ 2d^2 (\log H_2 + \log H_3)}{4t}  \right) \le H_1^{-1/t}.
\]
So the diagonal terms dominate in \cite[Theorem 5.6]{LOSmith}, which finally gives
\[
S_4 \ll (\log H_1)^{\kappa} H_{10}H_2^2 H_3^2 \cdot H_3^{-1/2t}\quad\text{with}\quad\kappa =  27d^2t^{-1} [F:\Q] + d^2t|a_1|^2 + 2o_2 + 2o_3 +\tfrac{1}{2}t^{-1}o_2^4o_3^4.
\]
We have $\kappa \le 16 [K:\Q]^3 \cdot o_1^4 o_2^3 o_3^2$, so
\[
S_3 \ll (\log H_1)^{18 [K:\Q]^3 o_1^4 o_2^3 o_3^2}\cdot  H_1 H_2^2 H_3^2 \cdot H_3^{-1/2t}.
\]
After bounding the $|X_i|$ using Lemma \ref{lem:simple_multiplicative_bound}, we then get
\[
S_0 \ll  (\log H_1)^{7[K:\Q]^3 o_1^4 o_2^3 o_3^2}\cdot H_1H_2 H_3\cdot H_3^{-1/8t}, 
\]
giving the theorem.
\qed

\begin{remark}
\label{rmk:supersolvable}
In the case that the groups $\Gal(E_y/F)$ are supersolvable, the characters $\mu_{iy}$ defined above may be restricted to be irreducible characters \cite[Theorem 6.22]{Isaacs} if we allow the complex coefficients $a_i$ to have magnitude at most $o_1$. In this case, $\sum_i \deg \mu_{iy}^2$ is at most $[K:F] \cdot o_1o_2o_3$, allowing us to take
\[
c = \frac{1}{64[K:F] \cdot o_1 o_2o_3 }\quad\textup{and}\quad C = 7 [K:\Q]^3 \cdot o_1^2 o_2^3 o_3^3.
\]
For example, this condition is satisfied if the $M_i$ are cyclic modules and $K/F$ is chosen to be minimal so the $M_i$ and $\mu_{o_i}$ are $\Gal(K/F)$-modules.
\end{remark}

\section{Cubic twist families}
\label{sCubicTwist}
Throughout this section, we will fix a nonnegative integer $d$, so our goal is to study the $3$-Selmer groups among the cubic twists $E_{d, n}$ and $E_{-27d, n}$ with $n$ varying.

\subsection{\texorpdfstring{Galois structure of $E_{d, n}$}{Galois structure of E}}
It will be convenient to fix a primitive third root of unity $\zeta$ and a square root $\sqrt{d}$ in $\ovQQ$. For every nonzero integer $n$, we define the action of $\mu_3$ on $E_{d, n}$ by
\[
\zeta (x, y) = (\zeta x, y).
\]
For each nonzero integer $n$, we may choose a cube root $\sqrt[3]{dn^2/16}$ to define an isomorphism of elliptic curves
\[
\beta_{d, n}: E_{d, n} \to E_{1, 4}
\]
over $\QQ\left(\mu_3, \sqrt[3]{2d^2n}, \sqrt{d}\right)$ by
\[
(x, y) \mapsto \left(\frac{x}{\sqrt[3]{dn^2/16}},\, \,\frac{y}{\tfrac{1}{4}n \cdot \sqrt{d}}\right).
\]
Changing the cube root $\sqrt[3]{dn^2/16}$ varies this isomorphism by an automorphism from $\mu_3$.

Define $\chi_{2, d}$ to be the quadratic character associated to $\QQ(\sqrt{d})/\QQ$ and take $\chi_{3, dn^2/16}: G_{\QQ} \to \mu_3$ to be the cocycle defined by
\[\chi_{3, dn^2/16} (\sigma) = \frac{\sigma \left(\sqrt[3]{dn^2/16}\right)}{\sqrt[3]{dn^2/16}}.\]
Modulo coboundaries in $B^1(G_{\QQ}, \mu_3)$, the map $\chi_{3, dn^2/16}$ is the image of $dn^2/16$ under the Kummer map to $H^1(G_{\QQ}, \mu_3)$. In any case, we may calculate
\begin{equation}
\label{eq:twist}
\beta_{d,n}(\sigma z) = \chi_{3, dn^2/16}(\sigma)\chi_{2,d}(\sigma) \sigma \beta_{d, n}(z)\quad\text{for all}\quad z \in E_{d, n}\left(\ovQQ\right) \,\,\text{ and }\,\, \sigma \in G_{\Q}.
\end{equation}
The curve $E_{1, 4}$ is identifiable as the curve 27.a4 in the LMFDB database, which has two distinct rational $3$-isogenies \cite{LMFDB, LMFDB2}, so we have an isomorphism 
\[E_{1, 4}[3]  \cong \FFF_3 \oplus \mu_3\]
of Galois modules. The point $(0, 4)$ lies in $E_{1, 4}\left[\sqrt{-3}\right]$ and is $G_{\QQ}$ stable, so it generates the submodule $\FFF_3 \oplus 0$ of this direct sum.

Fix a generator $e_1$ for $\mu_3$ in this direct sum. Viewing $\zeta$ as automorphism on $E_{1, 4}$, we see that $e_2 = (\zeta - 1) e_1$ is a generator for the submodule $\FFF_3$ of $E_{1, 4}[3]$.

Then $\beta_{d, n}^{-1}(e_1)$ and $\beta_{d, n}^{-1}(e_2)$ are a basis for the $\FFF_3$ vector space $E_{d, n}[3]$. Taking $z = \beta_{d, n}^{-1}(e_1)$, we have
\[\sigma z = \chi_{2,\, -3d}(\sigma) z + \chi_{2,\, -3d}(\sigma) \cdot (\chi_{3,dn^2/16}(\sigma) - 1) z\quad\text{and}\quad
\sigma \big((\zeta - 1)z\big) = \chi_{2, d}(\sigma) \cdot (\zeta - 1) z\]
for all $\sigma$ in $G_{\QQ}$.

Given any nonzero integer $m$, we define a Galois module $M_{m}$ over $\QQ$ to be the quadratic twist of the trivial module $\FFF_3$ by the faithful character of $\Gal(\QQ(\sqrt{m})/\QQ)$, and we take $\alpha_m$ to be the image of $1$ in this twist. The above calculations show that we have an exact sequence
\begin{equation}
\label{eq:Edn3}
\begin{tikzcd}
0 \arrow{r} & M_{d} \arrow{r}{\iota} & E_{d, n}[3] \arrow{r} & \arrow[bend right=20, swap]{l}{\,s} M_{-3d} \arrow{r} & 0,\end{tikzcd}
\end{equation}
where the embedding $\iota:M_d \to E_{d, n}[3]$ takes $\alpha_d$ to $(\zeta - 1) z$, where the map to $M_{-3d}$ takes $z$ to $\alpha_{-3d}$, and where $s$ is the section of abelian groups  taking $\alpha_{-3d}$ to $z$. Given $\sigma \in G_{\QQ}$, we then have
\begin{equation}
\label{eq:section_behavior}
\sigma s(\sigma^{-1} \alpha_{-3d})  - s(\alpha_{-3d}) =  \frac{\chi_{3, dn^2/16}(\sigma) - 1}{\zeta - 1} \cdot \iota(\alpha_d).
\end{equation}
Choosing another nonzero integer $n'$, we may use this same construction for $E_{d, n'}[3]$. We then have a commutative diagram of abelian groups
\begin{equation}
\label{eq:beta_relation}
\begin{tikzcd}
    0 \arrow{r} & M_d \arrow[d, equal] \arrow{r}  & E_{d, n}[3] \arrow{d}{\beta_{d, n'}^{-1} \circ \beta_{d, n}} \arrow{r} & M_{-3d} \arrow{r} \arrow[d, equal]& 0 \\
    0 \arrow{r} & M_d \arrow{r} & E_{d, n'}[3] \arrow{r} & M_{-3d} \arrow{r} & 0.
\end{tikzcd}
\end{equation}

\subsection{Local conditions}
\label{ssec:loc_cond}
From \eqref{eq:Edn3}, we may view $\Sel^{\sqrt{-3}} E_{d, n}$ as a subgroup of the form 
\[\ker\left(H^1(G_{\QQ}, M_{d}) \to \prod_{v\text{ of }\Q} H^1(G_v, M_{d})/\mathscr{W}_v(E_{d, n})\right),\]
where
\[\mathscr{W}_v(E_{d, n}) = \ker\left(H^1(G_v, M_{d}) \to H^1(G_v, E_{d, n})\right).\]
Recall that we still have $d$ fixed. For all rational places $v$, we note that $\mathscr{W}_v$ depends just on the class of $n$ in $\QQ_v^{\times}/ (\QQ_v^{\times})^3$.

Consider $p \ne 2, 3$, and take $k$ to be the valuation of $dn^2$ at $p$, so $p^k || dn^2$.
If $k$ is divisible by $6$, then $\mathscr{W}_p$ is the set of unramified cocycle classes. Otherwise, $\mathscr{W}_p$ is the image of $H^0\left(G_p, E_{-27d,n}[3^{\infty}]\right)$ under the connecting map in the long exact sequence
\[H^0\left(G_p, E_{d, n}[3^{\infty}]\right) \xrightarrow{\,\,\sqrt{-3}\,\,} H^0\left(G_p, E_{-27d, n}[3^{\infty}]\right) \to H^1\left(G_p, M_{d}\right).\]
By considering how \eqref{eq:twist} behaves at the inertia subgroup for $p$, this reduces to the image of the connecting map
\[H^0(G_p, M_{-3d}) \to H^1(G_p, M_{d})\]
corresponding to $E_{d, n}$. From \eqref{eq:section_behavior}, this connecting map is given by the cup product with the image of $\chi_{3, \,dn^2/16}$ in $H^1(G_p, \mu_3)$, where the associated bilinear map $M_{-3d} \otimes \mu_3 \to M_{d}$ is defined to take $\alpha_{-3d} \otimes \zeta$ to $\alpha_{d}$. As a result, for $p$ a rational prime other than $2$, $3$, or $\infty$, we have
\begin{equation}
\label{eq:loc_Tama}
\left|\mathscr{W}_p(E_{d, n})\right| = 
\begin{cases}
\left|H^0(G_p, M_{-3d})\right|  &\text{ if  $E_{d, n}$ has bad reduction at $p$ } \\ 
\left|H^0(G_p, M_{d})\right| &\text{ otherwise.} 
\end{cases}
\end{equation}
We define the \emph{Tamagawa ratio} for $E_{d, n}$ by
\begin{align}
\label{eTamagawa}
\mathcal{T}(E_{d, n}) = \frac{\left| H^0(G_{\Q}, M_{d})\right|}{\left| H^0(G_{\Q}, M_{-3d})\right|} \cdot \prod_{v} \frac{\left|\mathscr{W}_v(E_{d, n})\right|}{\left|H^0(G_v, M_{d})\right|},
\end{align}
where the product is over all rational places of $\Q$. This is defined so that the Greenberg--Wiles formula, which we take from \cite[Theorem 8.7.9]{Neuk07}, takes the form
\begin{equation}
\label{eq:Tama_relation}
\left|\Sel^{\sqrt{-3}} E_{d, n}\right| = \left|\Sel^{\sqrt{-3}} E_{-27d, n}\right| \cdot \mathcal{T}(E_{d, n}).
\end{equation}

\subsection{Relationships among Cassels--Tate pairings}
Now choose two nonzero integers $n, n'$. By viewing both $\Sel^{\sqrt{-3}} E_{-27d, n}$ and $\Sel^{\sqrt{-3}} E_{-27d, n'}$ as subgroups of $H^1(G_{\Q}, M_{-3d})$ using \eqref{eq:beta_relation} at $-27d$, we may ask about
\[
\CTP_{d, n}(\phi, \psi) - \CTP_{d, n'}(\phi, \psi) \quad\text{ for } \phi, \psi \in \Sel^{\sqrt{-3}} E_{-27d, n} \cap \Sel^{\sqrt{-3}} E_{-27d, n'}.
\]
A typical idea is that this difference should be controllable if $n$ and $n'$ vary at just one prime.

\begin{proposition}
\label{pCT}
Given a nonzero integer $d$, there is an equivariant isomomorphism
\begin{equation}
\label{eq:sign?}
\mu_3 \otimes M_{-3d} \otimes M_{-3d} \to \mu_3
\end{equation}
such that, for any  nonzero integer $m$, any $k \in \{1, 2\}$ and for any pair of primes $p, q$ such that 
\[
p/q \in (\Q_r^{\times})^3 \quad\textup{ for all primes } r \mid 6dm,
\]
we have
\[
\exp\big(2\pi i\left(\CTP_{d, mp^k}(\phi, \psi) - \CTP_{d, mq^k}(\phi, \psi)\right)\big) = \Redei{\chi_{3,\, p/q}}{\phi}{\psi}^k,
\]
for all $\phi, \psi$ in the intersection of $\Sel^{\sqrt{-3}} E_{-27d, \, mp^k}$ and $\Sel^{\sqrt{-3}} E_{-27d, \, mq^k}$, and where the generalized R\'{e}dei symbol is defined from the homomorphism \eqref{eq:sign?}.
\end{proposition}

\begin{proof}
Given a nonzero integer $n$, we have an identification
\[
\iota: M_{-3d} \to \text{Hom}(M_d, \mu_3)
\]
taking $x$ to $y \mapsto \langle y,\overline{x} \rangle$, where the final pairing is the Weil pairing in $E_{d, n}[3]$ and where $\overline{x}$ is a lift of $x$ to $E_{d, n}[3]$ using the sequence \eqref{eq:Edn3}. From \eqref{eq:beta_relation}, this does not depend on the choice of $n$. Here we used that $\beta_{d, n}$ is an isomorphism of elliptic curves over $\ovQQ$; this implies that the Weil pairing is unchanged when applying $\beta_{d, n}$ to both arguments.
Having made this identification, we take \eqref{eq:sign?} to be the homomorphism taking $\zeta \otimes \alpha_{-3d} \otimes \alpha_{-3d}$ to $\iota(\alpha_{-3d})(\alpha_{d})$.

We first claim that the symbol $\Redei{\chi_{3, p/q}}{\phi}{\psi}$ is nonzero. Note that this is a consequence of three observations:
\begin{itemize}
\item We have $\text{inv}_v\left(x \cup \text{res}_{G_v}\left(\phi \cup \psi\right)\right) = 0$ at all rational places $v$ for all $x \in H^0(G_{v}, \mu_3)$.
\item $\chi_{3, p/q}$ has trivial restriction to $G_v$ for $v$ dividing $6dm$.
\item Either $p = q$ or $\phi$ and $\psi$ have trivial restriction to $G_p$ and $G_q$.
\end{itemize}
The first observation is clear if $G_v$ does not fix $\mu_3$. If $G_v$ instead does fix $\mu_3$, then we have that $\text{ker}(H^1(G_v, E_{d,n}[3]) \to  H^1(G_v, E_{d, n}))$ is stable under the action of $\mu_3$ for $n$ either $mp$ or $mq$, so cup product with an element in $H^0(G_v, \mu_3)$ defines a homomorphism from the local conditions for $E_{-27d, n}[\sqrt{-3}]$ to the local conditions for $E_{d, n}[\sqrt{-3}]$ at $v$. These local conditions are orthogonal to the local conditions for $E_{-27d, n}[\sqrt{-3}]$ by the theory of isogeny Selmer groups \cite[Proposition 6.1]{MS22a}.

The second observation is clear from the assumptions. Finally, if $p \ne q$, neither $\phi$ nor $\psi$ can be ramified at either $p$ or $q$. Since nontrivial multiples of $\chi_{3, dm^2p^{2k}/16}$ and $\chi_{3, dm^2q^{2k}/16}$ 
are ramified at $p$ and $q$, the results in Section  \ref{ssec:loc_cond} show that $\phi$ and $\psi$ must have trivial restriction to both $G_p$ and $G_q$. This gives the third observation.

We now calculate the Cassels--Tate pairings. The result is obvious if $p = q$, so we presume they are distinct.  Fix inclusions $\overline{\Q} \to \overline{\Q_p}$ and $\overline{\Q} \to \overline{\Q_q}$; these correspond to inclusions of $G_p$ and $G_q$ in $G_{\Q}$. Choose cocycles $\overline{\phi}$ and $\overline{\psi}$ representing $\phi$ and $\psi$. We choose these cocycles so they are trivial on $G_p$ and $G_q$, respectively.

Take $\beta: E_{d, \,mp^k} \to E_{d, \, mq^k}$ to be the isomorphism $\beta_{d, \,mq^k}^{-1} \circ \beta_{d,\, mp^k}$. From \eqref{eq:twist}, this corresponds to a choice of $\sqrt[3]{p/q}$ in $\ovQQ$. 


Take $\Vplac$ to be the minimal set of places containing $\infty$ and all primes dividing $6dm$. For $v$ in $\Vplac$ we choose $\ovQQ \hookrightarrow \overline{\QQ_v}$ so $\QQ_v$ contains the image of this choice of $\sqrt[3]{p/q}$. Under the corresponding inclusion of $G_v$ in $G_{\QQ}$, we see that $\beta$ is equivariant over $G_v$ and that it carries the local conditions for $E_{d, \,mp^k}[3]$ to the local conditions for $E_{d,\, mq^k}[3]$ at $v$. 

Take $s: M_{-3d} \to E_{d, mp^k}[3]$ to be the section of \eqref{eq:Edn3}. For each place $v$ in $\Vplac$, choose a cocycle $\Phi_v$ representing a class in the local conditions for $E_{d, mp^k}[3]$ at $v$ that projects to $\text{res}_{G_v}\overline{\phi}$, and choose 
\[
\epsilon  \in C^2\left(G_{\QQ, S}, \mathcal{O}_{\QQ, S}^{\times}\right)
\]
so $d\epsilon = d (s\circ \overline{\phi}) \cup \iota(\overline{\psi})$, where the cup product is defined from the Weil pairing and where $S = \Vplac \cup \{p, q\}$.

Then the Weil pairing definition of the Cassels--Tate pairing gives
\[
\CTP_{d, mp^k}(\phi, \psi) = \sum_{v  \in \Vplac} \text{inv}_v\left(\left(s \circ \textup{res}_{G_v}\overline{\phi} - \overline{\Phi}_v\right) \cup \text{res}_{G_v} \iota\circ \overline{\psi} - \text{res}_{G_v}\epsilon\right) + \sum_{v \in \{p, q\}} \text{inv}_v\left( -\text{res}_{G_v} \epsilon\right).
\]
If we choose $\epsilon'$ in the same space of $2$-cochains as above so $d\epsilon' = d(\beta (s \circ \overline{\phi})) \cup \iota(\overline{\psi})$, we also have
\begin{align*}
&\CTP_{d, mq^k}(\phi, \psi)  \\
&=\,\sum_{v  \in \Vplac} \text{inv}_v\left(\left(\beta s \circ \textup{res}_{G_v}\overline{\phi} - \beta(\overline{\Phi}_v)\right) \cup \text{res}_{G_v} \iota \circ \overline{\psi} - \text{res}_{G_v}\epsilon'\right) + \sum_{v \in \{p, q\}} \text{inv}_v\left( -\text{res}_{G_v} \epsilon'\right).  
\end{align*}
Using that the Weil pairing is unchanged upon applying $\beta$ to both arguments and that $\iota \circ \overline{\psi}$ is fixed by $\beta$, we may take the difference of the above two equations to obtain
\[
\CTP_{d, mp^k}(\phi, \psi) - \CTP_{d, mq^k}(\phi, \psi) = \sum_{v \in \Vplac \cup \{p, q\}} \text{inv}_v\circ \textup{res}_{G_v}(\epsilon' - \epsilon).
\]
From \eqref{eq:section_behavior}, we may calculate that
\[
d(\epsilon' - \epsilon) = \chi_{3, q^{2k}/p^{2k}} \cup \overline{\phi} \cup \iota(\overline{\psi}),
\] 
where the first cup product comes from the bilinear pairing $\mu_3 \otimes M_{-3d} \to M_d$ taking $\zeta \otimes \alpha_{-3d}$ to $\alpha_d$. The result follows from the definition of the symbol.
\end{proof}

\section{Generalized Legendre symbols and grids}
\label{sec:Leg}
As in \cite{Smi22a, Smi22b}, we will study  $\sqrt{-3}$-Selmer groups using the framework to generalized Legendre symbols introduced in \cite[Section 3]{Smi22a}. We recall what we need of this now.

\subsection{Setting up our generalized Legendre symbols}
\begin{setup}
Fix a nonzero integer $d$, and take $K = \QQ(\mu_3, \sqrt{d})$. Fix a finite set of rational places $\Vplac_0$ such that $(K/\Q, \Vplac_0, 3)$ is an unpacked starting tuple, in the sense of \cite[Definition 3.4]{Smi22a}. Finally, fix an \emph{assignment of approximate generators}, as in \cite[Notation 3.8]{Smi22a}. We may choose this assignment so, for any prime $\mfp$ of $K$ not over a prime in $\Vplac_0$ and any primes $\ovp$ and $\ovp'$ of $\ovQQ$ over $\mfp$, the approximate generators for $\ovp$ and $\ovp'$ are the same. We take $\Frob\, \mfp$ to be the image of a Frobenius element in $G_{\mfp \cap \Q}$ in $\Gal(K_{\text{ab}, \mfp}/\Q)$ under some embedding corresponding to $\mfp$, where $K_{\text{ab}, \mfp}$ is the maximal abelian extension of $K$ not ramified over $\mfp \cap \Q$. 

From this, given a finite $\FFF_3[\Gal(K/\QQ)]$-module $M$, we can construct a homomorphism
\[
\mfB_{\mfp, M}: M(-1)^{G_{\mfp}} \to H^1(G_\Q, M)
\]
that serves as a section of the natural ramification-measuring homomorphism \cite[Definition 3.8 and (3.1)]{Smi22a}, where $M(-1)$ is the negative Tate twist $\text{Hom}(\mu_3, M)$. Note that the compatibility assumption we have made for primes over $\mfp$ allows us to index these maps with primes of $K$, rather than primes of $\ovQQ$.

Also from the approximate generators, we may define \emph{classes, spins, and symbols}. Classes are normally defined as collections of primes over $\ovQQ$ \cite[Definition 3.6]{Smi22a}, but the definition is such that all primes over the same prime $\mfp$ of $K$ lie in the same class, so we instead view them as collections of primes over $K$. We write the class of $\mfp$ as $\class{\mfp}$.

Similarly, given primes $\mfp, \mfq$ of $K$ not over $\Vplac_0$ such that $\Frob\, \mfp$ and $\Frob\, \mfq$ either map to different elements of $\Gal(K/\QQ)$ or map to the same element of $\Gal(K/\QQ(\mu_3))$, we define the \emph{generalized Legendre symbol} $\symb{\mfp}{\mfq} \in \mu_3$ as the alternative symbol $\symb{\ovp}{\overline{\mfq}}'$ \cite[Definition 3.21]{Smi22a}, where $\ovp$ and $\overline{\mfq}$ are primes of $\ovQQ$ over $\mfp$ and $\mfq$. 

Take $\tau_{\mfp}$ and $\tau_{\mfq}$ to be shorthand for the images of $\Frob\, \mfp$ and $\Frob\, \mfq$ in $\Gal(K/\Q)$. If neither $d$ nor $-3d$ are squares, we define $\tau_{-3}, \tau_d, \tau_{-3d} \in \Gal(K/\QQ)$ by the relationships
\begin{alignat*}{2}
&\tau_{-3}\left(\sqrt{-3}\right) = \sqrt{-3}\quad&&\text{and}\quad \tau_{-3}\big(\sqrt{d}\big) = -\sqrt{d}, \\
&\tau_d \left(\sqrt{-3}\right) = -\sqrt{-3}\quad&&\text{and}\quad \tau_d \big(\sqrt{d}\big) = \sqrt{d}, \\
&\tau_{-3d} \left(\sqrt{-3}\right) = -\sqrt{-3}\quad&&\text{and}\quad \tau_{-3d} \big(\sqrt{d}\big) = -\sqrt{d}.
\end{alignat*}
If $d$ or $-3d$ is a square, then $\Gal(K/\QQ)$ has order $2$ and we take $\tau_1 \in \Gal(K/\QQ)$ to be the unique non-trivial element.

In the case that $\Gal(K/\QQ)$ has order $4$, we define a subset $B(\mfp, \mfq)$ of $\Gal(K/\QQ)$ by
\begin{align}
\label{eq:Bcase}
B(\mfp, \mfq) = 
\begin{cases} 
\{1\} &\text{ if } \langle \tau_{\mfp}, \tau_{\mfq} \rangle = \Gal(K/\QQ) \\ 
\Gal(K/\QQ) &\text{ if } \tau_{\mfp} = \tau_{\mfq} = 1 \\
\{1, \tau_{-3}\} &\text{ if } \tau_{\mfp} \ne \tau_{\mfq} \text{ and }\langle \tau_{\mfp}, \tau_{\mfq}\rangle = \langle \tau_d\rangle \text{ or } \langle \tau_{-3d}\rangle \\
\{1, \tau_d\} &\text{ if } \langle \tau_{\mfp}, \tau_{\mfq} \rangle = \langle \tau_{-3}\rangle \\
\emptyset &\text{ if } \tau_{\mfp} = \tau_{\mfq} \in \{ \tau_d, \tau_{-3d}\}.
\end{cases}
\end{align}
If $\Gal(K/\QQ)$ has order $2$, we instead define
\begin{equation}
\label{eq:Bcase2}
B(\mfp, \mfq) = \begin{cases} 
\{1\} &\text{ if } \tau_{\mfp} \ne \tau_{\mfq} \\
\Gal(K/\QQ) &\text{ if } \tau_{\mfp} = \tau_{\mfq} = 1\\
\emptyset &\text{ if } \tau_{\mfp} = \tau_{\mfq} = \tau_1.
\end{cases}
\end{equation}
With these set, \cite[Proposition 3.22]{Smi22a} gives
\begin{equation}
\label{eq:double_coset}
\mfB_{\mfq, M}(m)(\Frob\, \mfp) \equiv \sum_{\tau \in B(\mfp, \mfq)} \tau \left(m\left(\symb{\tau\mfp}{\mfq}\right)\right) \quad\text{ mod }  (\tau_{\mfp} - 1)M
\end{equation}
for any $m$ in $M(-1)^{\tau_{\mfq}}$.  
\end{setup}

\subsection{Algebraic theory for grids of twists}
\label{ssec:alg_grids}
As in \cite{Smi22a, Smi22b}, our equidistribution results for Selmer groups will initially be proved for product spaces of twists that we call grids. Our development largely follows \cite[Section 4]{Smi22a}; we may not simply invoke this material because  we are considering twists coming from $H^1(G_{\Q}, \mu_3)$, rather than $H^1(G_{\Q}, \FFF_3)$.

\begin{mydef}
In addition to $d$, fix a nonzero integer $b$ and a positive integer $r$, and take $X_1, \dots, X_r$ to be disjoint sets of rational primes not containing any prime dividing $6db$. For $i \le r$, we fix a bijection $\lambda_i: X_i \isoarrow Y_i$, where $Y_i$ is a collection of primes of $K$ and $\lambda_i(p)$ is a prime over $p$ for each $p$ in $X_i$.  We will assume that, for each $i \le r$ and all $p_i, q_i$ in $X_i$, we have 
\begin{equation}
\label{eq:classy_assumption}
\class{\lambda_i(p_i)} = \class{\lambda_i(q_i)} \quad\text{and}\quad \symb{\lambda_i(p_i)}{\mfp} = \symb{\lambda_i(q_i)}{\mfp}\quad\text{for all } \mfp \mid bd \text{ satisfying } \mfp \cap \Q \not \in \Vplac_0.
\end{equation}
We then call $X = \prod_{i \le r} X_i$ a \emph{grid of twists}.

Given $x = (p_1, \dots, p_r)$ in $X$, we take $n(x) = b\cdot p_1 \cdot \dots \cdot p_r$; the associated cubic twist is then $E_{d,  n(x)}$.

We define the \emph{grid class} $\class{x}$ of $x$ to be the set of $y = (q_1, \dots, q_r) \in X$ for which 
\begin{enumerate}
\item For $\tau \in B(\lambda_i(p_i), \lambda_i(p_i)) \backslash \{1\}$ and $i \le r$, we have 
\[\symb{\tau \lambda_i(p_i)}{\lambda_i(p_i)} =\symb{\tau \lambda_i(q_i)}{\lambda_i(q_i)}\]
\item For $i \ne j$ and $\tau \in B(\lambda_i(p_i), \lambda_j(p_j))$, we have
\[ \symb{\tau \lambda_i(p_i)}{\lambda_j(p_j)} =\symb{\tau \lambda_i(q_i)}{\lambda_j(q_j)}.\]
\end{enumerate}
Take $\Vplac$ to be the set of places in $\Vplac_0$ together with the primes dividing $bd$. For $M$ either $M_d$ or $M_{-3d}$, we define
\[H^1_{\Vplac}(G_{\Q}, M) = \ker\left(H^1(G_{\QQ}, M) \to \prod_{v \not \in \Vplac} H^1(I_v, M)\right),\]
and we define a map
$$
\mfB_{x, M}: H^1_{\Vplac}(G_{\Q}, M) \oplus \bigoplus_{i=1}^r M(-1)^{G_{p_i}} \hookrightarrow H^1(G_{\Q}, M)
$$
by
\[
\mfB_{x, M}(\phi_0, (m_1, \dots, m_r)) = \phi_0 + \sum_{i = 1}^r \mfB_{\lambda_i(p_i), M}(m_i).
\]
We write the domain of $\mfB_{x, M}$ as $\mathcal{M}(M)$. Given $i \le r$ and $p_i, q_i \in X_i$, the elements $\Frob\, \lambda_i(p_i)$ and $\Frob\, \lambda_i(q_i)$ have the same image in $\Gal(K/\Q)$ by \eqref{eq:classy_assumption}; this suffices to show that $\mathcal{M}(M)$ does not depend on $x$. We also note that any element $\phi$ in $\Sel^{\sqrt{-3}} E_{-27d, n(x)}$ is in the image of this map for $M = M_{-3d}$.
\end{mydef}

Following the proof of \cite[Proposition 4.8]{Smi22a} for the theory in Section \ref{ssec:loc_cond} gives the following proposition.

\begin{proposition}
\label{prop:48}
Given $X$ as above and given $x$ and $y$ in the same grid class of $X$, we have
\[
\mfB_{x, M_{-3d}}^{-1} \left(\Sel^{\sqrt{-3}} E_{-27d, n(x)}\right) = \mfB_{y, M_{-3d}}^{-1}\left(\Sel^{\sqrt{-3}} E_{-27d, n(y)}\right)\quad\textup{ in } \mathcal{M}(M_{-3d}).
\]
\end{proposition}

\begin{proof}
Choose $x = (p_1, \dots, p_r)$ and $y = (q_1, \dots, q_r)$, and choose $m = (\phi_0, (m_1, \dots, m_r))$ in $\mathcal{M}(M_{-3d})$. We take $\mfp_i = \lambda_i(p_i)$ and $\mfq_i = \lambda_i(q_i)$ for $i \le r$.

For $v$ a rational place not in $\Vplac \cup \{p_1, \dots, p_r\}$, we see that $\mfB_{x, M_{-3d}}(m)$ obeys the local conditions at $v$ by virtue of being unramified. The same holds outside $\Vplac \cup \{q_1, \dots,q_r\}$ for $\mfB_{y, M_{-3d}}(m)$.

For $v \in \Vplac$, \eqref{eq:classy_assumption} forces
\[
\res_{G_v} \chi_{3, p_i} = \res_{G_v} \chi_{3, q_i}\quad\text{and}\quad \res_{G_v} \mfB_{\mfp_i, M_{-3d}}(m_i) = \res_{G_v} \mfB_{\mfq_i, M_{-3d}}(m_i)\quad\text{for } i \le r
\]
by \cite[Definition 3.10 (3) and Proposition 3.23]{Smi22a}. 
The first equality means $M_{-3d}$ has the same local conditions at $v$ when looking at the twists associated either to $x$ or $y$, and the second equality gives that $\mfB_{x, M_{-3d}}(m)$ satisfies these local conditions if and only if $\mfB_{y, M_{-3d}}(m)$ does.

We now wish to prove that $\mfB_{x, M_{-3d}}(m)$ satisfies the local conditions at $p_i$ if and only if $\mfB_{y, M_{-3d}}(m)$ satisfies the local conditions at $q_i$. We focus on the former. From Section \ref{ssec:loc_cond}, the local condition at $v = p_i$ is satisfied if and only if
\[\res_{G_v} \mfB_{x, M_{-3d}}(m) - r \cup \res_{G_v} \chi_{3, dn^2/16} = 0\]
for some $r$ in $H^0(G_v, M_{-3d}(-1)) \cong H^0(G_v, M_{d})$, where the cup product is associated to the evaluation pairing. 
This is only unramified if $r = -m_i$, leaving the condition
\begin{equation}
\label{eq:loc_cond}
\mfB_{x, M_{-3d}}(m)(\Frob\, \mfp_i) + m_i \cdot \chi_{3, dn^2/16}(\Frob\,\mfp_i) \equiv 0\,\text{ mod }\, (\Frob\, \mfp_i - 1) M_{-3d}.
\end{equation}

By \eqref{eq:classy_assumption} and \eqref{eq:double_coset}, there is some $a \in \mu_3$ independent of the choice of $x$ such that the left side of \eqref{eq:loc_cond} equals
\begin{equation}
\label{eq:explicit_local_conds}
m_i(a) + \sum_{j \le r} \sum_{\tau \in B(\mfp_i, \mfp_j)} \tau\left( m_j\left(\symb{\tau \mfp_i}{\mfp_j}\right)\right) - m_i\left(\tau \left(\symb{\tau \mfp_i}{\mfp_j}\right)\right) \quad\text{in } M_{-3d}/(\Frob\, \mfp_i -1)M_{-3d}.
\end{equation}
This can be calculated from from the symbols used to define a grid class. Since $y$ shares this grid class, we find that $ \mfB_{x, M_{-3d}}(m)$ obeys the local condition at $p_i$ if and only if $ \mfB_{y, M_{-3d}}(m)$ obeys the local condition at $q_i$. This gives the result.
\end{proof}

Given some $x \in X$, take $\mathcal{M}_0(M_{-3d})$ to be the subgroup of $m \in \mathcal{M}(M_{-3d})$ such that $\mfB_{x, M_{-3d}}(m)$ maps into $\mathscr{W}_v(E_{-27d, n(x)})$ for all $v \in \Vplac$. As shown in the proof of Proposition \ref{prop:48}, this does not depend on $x$. The typical cardinality of this set is also not hard to calculate.

\begin{lemma}
\label{lem:nosmallramonly}
With all notation as above, take $K(\Vplac)$ to be the maximal abelian extension of $K$ of exponent $3$ ramified only at places in $\Vplac$. Choose some $x = (p_1, \dots, p_r)$ in $X$, and suppose there is a subset $S$ of $\{1, \dots, r\}$ such that $p_i$ splits in $K/\QQ$ for each $i \in S$ and the normal subgroup of $\Gal(K(\Vplac)/\QQ)$ generated by
$$
\{ \Frob\, p_i\,:\,\, i \in S\}
$$
is $\Gal(K(\Vplac)/K)$. We note that this condition does not depend on the choice of $x$ by \eqref{eq:classy_assumption}.

Then
\begin{equation}
\label{eq:nosmallramonly}
H^1_{\Vplac}(G_{\Q}, M_d) \,\cap\, \Sel^{\sqrt{-3}} E_{d, n(x)}\, =\, 0.
\end{equation}
Equivalently, we have
\begin{equation}
\label{eq:M0_count}
\left| \mathcal{M}_0(M_{-3d})\right| = \mathcal{T}(E_{-27d, n(x)}) \cdot \prod_{i \le r} \left|H^0(G_{p_i}, M_{-3d})\right|.
\end{equation}
\end{lemma}

\begin{proof}
Given $\phi$ in $H^1_{\Vplac}(G_{\QQ}, M_{d})$, we note that the minimal field containing $K$ over which $\phi$ becomes trivial is an abelian extension of $K$ of exponent dividing $3$ ramified only at places in $\Vplac$. So the inflation map defines an isomorphism
\[
H^1_{\Vplac}(G_{\QQ}, M_d) \cong H^1(\Gal(K(\Vplac)/\Q), M_d).
\]
Note that the restriction map
\[
H^1(\Gal(K(\Vplac)/\Q), M_d)\xrightarrow{\,\,\text{res}\,\,}H^1(\Gal(K(\Vplac)/K), M_d) \cong \Hom(\Gal(K(\Vplac)/K), M_d) 
\]
is injective since $|\Gal(K/\Q)|$ and $|M_d|$ are coprime.

Take $\psi$ to be the image of $\phi$ in $\Hom(\Gal(K(\Vplac)/K), M_d)$. If $\phi$ was nonzero, $\psi$  is nonzero, so we have
\[
\psi(\sigma \cdot \Frob\, \lambda_i(p_i) \cdot \sigma^{-1}) \ne 0
\]
for some $p_i$ and some $\sigma \in G_{\Q}$. Since $\psi$ is the restriction of a cocycle, we have 
\[
\psi(\sigma \cdot \Frob\, \lambda_i(p_i) \cdot \sigma^{-1}) = \phi(\sigma \cdot \Frob\, \lambda_i(p_i) \cdot \sigma^{-1}) = \sigma \cdot \phi(\Frob\, \lambda_i(p_i)) = \sigma \cdot \psi(\Frob\, \lambda_i(p_i)) \ne 0.
\]
Because $\psi$ is a homomorphism, we conclude that $\psi$ has nontrivial restriction to $G_{p_i}$. It is also unramified at $p_i$, so $\phi$ cannot satisfy the local conditions at $p_i$.

This gives \eqref{eq:nosmallramonly}. Using $\mfB_{x, M_{-3d}}$, we may identify $\mathcal{M}_0(M_{-3d})$ with the Selmer group on $M_{-3d}$ whose local conditions at $v = p_1, \dots, p_r$ are $H^1(G_v, M_{-3d})$ and whose local conditions equal $\mathscr{W}_v(E_{-27d, n(x)})$ for other $v$. The dual Selmer group to this is the left hand side of \eqref{eq:nosmallramonly}, so the Greenberg--Wiles formula \cite[Theorem 8.7.9]{Neuk07} shows the equivalence of \eqref{eq:nosmallramonly} and \eqref{eq:M0_count}.
\end{proof}

If we replace $d$ with $-27d$, the hypothesis of Lemma \ref{lem:nosmallramonly} does not shift. So we note that this same hypothesis is enough to also conclude
\begin{equation}
\label{eq:nosmallramonly2}
H^1_{\Vplac}(G_{\Q}, M_{-3d}) \,\cap\, \Sel^{\sqrt{-3}} E_{-27d, n(x)}\, =\, 0.
\end{equation}

\begin{mydef}
Still working with a grid $X \ni x$ as above, take $\mathcal{M}_{0, \Vplac}(M_{-3d})$ to be the subset of $\mathcal{M}_0(M_{-3d})$ of elements of the form $(\phi_0, 0, \dots, 0) \in \mathcal{M}_0(M_{-3d})$ with $\phi_0$ nonzero. In addition, take $\mathcal{D}(M_{-3d})$ to be the subgroup of $\mathcal{M}_0(M_{-3d})$ identified with 
the image of the connecting map
\[
H^0(G_{\Q}, M_{d}) \to \Sel^{\sqrt{-3}} E_{-27d, n(x)} 
\]
under $\mfB_{x, M_{-3d}}$. Note that this subgroup does not depend on the choice of $x$, and that it is $0$ unless $d$ is square.

We then define
\[\mathcal{M}_1(M_{-3d}) = \mathcal{M}_0(M_{-3d}) \big\backslash (\mathcal{M}_{0, \Vplac}(M_{-3d}) + \mathcal{D}(M_{-3d})).\]
Note that, if $x$ satisfies the hypothesis of Lemma \ref{lem:nosmallramonly}, the preimage of the Selmer group $\Sel^{\sqrt{-3}} E_{-27d, n(x)}$ automatically lies in $\mathcal{M}_1(M_{-3d})$ by \eqref{eq:nosmallramonly2}.
\end{mydef}

\subsection{\texorpdfstring{$\sqrt{-3}$-Selmer moments over a grid}{Sqrt(-3)-Selmer moments over a grid}}
Our next steps follow the basic track of \cite[Sections 5 and 6]{Smi22b}.

Take $S$ to be the subset of $i \le r$ such that $H^0(G_{p_i}, M_{-3d})$ is nonzero for any/every $p_i$ in $X_i$, and take
\[
\mathcal{R} = \bigoplus_{i \in S} \Hom\left(M_{-3d}, \,\tfrac{1}{3}\Z/\Z\right).
\]
As a Galois module, this is isomorphic to the direct sum of $|S|$ copies of $M_{-3d}$.

Given $x = (p_1, \dots, p_r)$ in $X$, $\mu = (\phi_0, (m_1, \dots, m_r))$ in $\mathcal{M}_0(M_{-3d})$ and $i \in S$, take
$L_{i, x}(\mu)$ to be the expression \eqref{eq:explicit_local_conds}. Given $\gamma = (\gamma_i)_{i \in S}$ in $\mathcal{R}$, we then define
\[
(\gamma \cdot \mu)_x = \sum_{i \in S} \gamma_i \cdot L_{i, x}(\mu).
\]
Then Proposition \ref{prop:48} shows that $\mfB_{x, M_{-3d}}(\mu)$ lies in the $\sqrt{-3}$-Selmer group of $E_{d, n(x)}$ if and only if $(\gamma \cdot \mu)_x = 0$ for all $\gamma$ in $\mathcal{R}$.

As a consequence, taking $e(t)$ as notation for $\exp(2\pi i \cdot t)$, we have the key identity
\[\left|\Sel^{\sqrt{-3}} E_{-27d, n(x)}\right| = \frac{1}{\left| \mathcal{R}\right|} \sum_{\mu \in \mathcal{M}_0(M_{-3d})} \sum_{\gamma \in \mathcal{R}} e\left( (\gamma \cdot \mu)_x\right)\quad\text{for all } x \in X,\]
which shows that the average Selmer size over $X$ equals
\begin{equation}
\label{eq:average}
\frac{1}{|X|} \sum_{x \in X} \left|\Sel^{\sqrt{-3}} E_{-27d, n(x)}\right| = \frac{1}{\left| \mathcal{R}\right| \cdot |X|} \sum_{\mu \in \mathcal{M}_0(M_{-3d})} \,\sum_{\gamma \in \mathcal{R}} \, \sum_{x \in X}  e\left((\gamma \cdot \mu)_x \right).
\end{equation}
Our approach is to consider the inner summand for a fixed choice of $(\mu, \gamma)$. The following result helps us keep track of how this inner summand changes as $x$ changes.

\begin{lemma}
\label{lem:bilinear_vary}
In the above situation, choose $\mu = (\phi_0, (m_1, \dots, m_r))$ in $\mathcal{M}_0(M_{-3d})$, choose $\gamma = (\gamma_i)_{i \in S}$ in $\mathcal{R}$, and choose $x = (p_1, \dots, p_r) \in X$. Also choose distinct integers $j, k \le r$, and take $Z$ to be the subset of $y = (q_1, \dots, q_r) \in X$ such that $q_i = p_i$ for $i \ne j, k$. Take $\gamma_i = 0$ for $i$ outside $S$, and take $\mfp_i = \lambda_i(p_i)$ for all $i \le r$.

Then there exist functions $a_j: X_j \to \tfrac{1}{3}\Z/\Z$ and $a_k: X_k \to \tfrac{1}{3}\Z/\Z$ such that
\[
(\gamma \cdot \mu)_y = a_j(q_j) + a_k(q_k) + \sum_{\tau \in B(\mfp_j, \mfp_k)} b_{\tau}\left(\symb{\tau\lambda_j(q_j)}{\lambda_k(q_k)}\right)
\quad\textup{for } y = (q_1, \dots, q_r) \in Z,
\]
where $b_{\tau} \in \Hom\left(\mu_3, \tfrac{1}{3}\Z/\Z\right)$ is given by
\[
b_{\tau} \,=\, - (\tau \gamma_j - \gamma_k) \cdot (\tau m_j - m_k) \,=\, \chi_{2, -3d}(\tau)\gamma_j \cdot  m_k  - \chi_{2, -3}(\tau) \gamma_j \cdot m_j + \chi_{2, d}(\tau) \gamma_k  \cdot m_j - \gamma_k \cdot m_k.
\]
\end{lemma}

\begin{proof}
Choose $y = (q_1, \dots, q_r) \in Z$, and take $\mfq_i$ to be $\lambda_i(q_i)$ for $i \le r$. We see that 
\[L_{j, y}(\mu) - \sum_{\tau \in B(\mfq_j, \mfq_k)} \tau\left( m_k\left(\symb{\tau \mfq_j}{\mfq_k}\right)\right) - m_j\left(\tau \left(\symb{\tau \mfq_j}{\mfq_k}\right)\right)\]
depends only on the choice of $q_j$, and not on $q_k$. Similarly,
\[L_{k, y}(\mu) - \sum_{\tau \in B(\mfq_k, \mfq_j)} \tau\left( m_j\left(\symb{\tau \mfq_k}{\mfq_j}\right)\right) - m_k\left(\tau \left(\symb{\tau \mfq_k}{\mfq_j}\right)\right)\]
depends only on the choice of $q_k$, and not on $q_j$.
For $i \ne j, k$, 
\[ \sum_{\tau \in B(\mfq_i, \mfq_j)} \tau\left( m_j\left(\symb{\tau \mfq_i}{\mfq_j}\right)\right) + m_i\left(\tau \left(\symb{\tau \mfq_i}{\mfq_j}\right)\right)\]
depends only on $q_j$ and not $q_k$; meanwhile, the difference between this expression and $L_{i, y}(\mu)$ depends only on $q_k$ and not on $q_j$.

So we find there are functions $a_{0j}: X_j \to \tfrac{1}{3}\Z/\Z$ and $a_{0k}: X_k \to \tfrac{1}{3}\Z/\Z$ such that
\begin{align*}
(\gamma \cdot \mu)_y = a_{0j}(q_j) + a_{0k}(q_k) \,+ &\sum_{\tau \in B(\mfq_j, \mfq_k)} \gamma_j \cdot \tau\left( m_k\left(\symb{\tau \mfq_j}{\mfq_k}\right)\right) - \gamma_j \cdot m_j\left(\tau \left(\symb{\tau \mfq_j}{\mfq_k}\right)\right) \\
+& \sum_{\tau \in B(\mfq_k, \mfq_j)} \gamma_k \cdot \tau\left( m_j\left(\symb{\tau \mfq_k}{\mfq_j}\right)\right) - \gamma_k \cdot m_k\left(\tau \left(\symb{\tau \mfq_k}{\mfq_j}\right)\right),
\end{align*}
where the products are the evaluation pairing. The sets $B(\mfq_j, \mfq_k)$ and $B(\mfq_k, \mfq_j)$ are equal. Given $\tau$ in $G_{\Q}$, we have $\tau^{-1}\mfq_j = \tau \mfq_j$, so reciprocity \cite[(5.1)]{Smi22b} and \eqref{eq:classy_assumption} give that
\[\symb{\tau \mfq_k}{\mfq_j} - \tau\left(  \symb{\tau \mfq_j}{\mfq_k}\right)\]
depends only on $\tau$ and not on $q_j$ or $q_k$. From this, we find
\begin{align*}
(\gamma \cdot \mu)_y &= a_{0j}(q_j) + a_{0k}(q_k) + a + \sum_{\tau \in B(\mfq_j, \mfq_k)} b_{\tau}\left(\symb{\tau \mfq_j}{\mfq_k}\right),
\end{align*}
for some fixed $a$ in $\tfrac{1}{3}\Z/\Z$, and the result follows.
\end{proof}

Together with Lemma \ref{lem:bilinear_vary} and some results from analytic number theory, the following lemma gives a partial characterization of the $(\mu, \gamma)$ for which the sum over $x \in X$ is significant in \eqref{eq:average}.

\begin{lemma}
\label{lem:bilin_ign}
Choose $\mu = (\phi_0, (m_1, \dots, m_r))$ from $\mathcal{M}_0(M_{-3d})$ and $\gamma = (\gamma_{i})_{i \in S}$ from $\mathcal{R}$. Take $\gamma_i = 0$ for $i$ outside $S$. Choose $(p_1, \dots, p_r)$ in $X$.

Then one of the following holds:
\begin{enumerate}
\item We have $m_i = m$ for all $i \le r$ for some $m$ in $M_{-3d}(-1)$, and $m = 0$ unless $d$ is a square.
\item We have $\gamma_i = r$ for all $i \in S$ for some $r$ in $\Hom\left(M_{-3d}, \tfrac{1}{3}\Z/\Z\right)$, and $r  = 0$ unless $-3d$ is a square.
\item There is some $j, k \le r$ and $\tau \in B(\lambda_j(p_j), \lambda_k(p_k))$ such that
\begin{equation}
\label{eq:ignorable}
\left(\tau \gamma_j - \gamma_k\right) \cdot \left(\tau m_j - m_k\right) \ne 0.
\end{equation}
\end{enumerate}
\end{lemma}

\begin{proof}
Suppose neither of the first two conditions holds. We may then choose integers $j, k$ as follows:
\begin{itemize}
\item If neither $d$ nor $-3d$ are squares, choose $j \in S$ such that $\gamma_j$ is nonzero, and choose $k \le r$ such that $m_k$ is nonzero.
\item If $d$ is a square, choose $j \in S$ such that $\gamma_j$ is nonzero, and choose $k \le r$ such that $m_k \ne m_j$.
\item If $-3d$ is a square, so $S = \{1, \dots, r\}$, choose $k \le r$ such that $m_k$ is nonzero, and choose $j \le r$ such that $\gamma_j \ne \gamma_k$.
\end{itemize}
Take $\tau_j$ and $\tau_k$ to be the respective images of $\Frob\, p_j$ and $\Frob\, p_k$ in $\Gal(K/\Q)$, and take $B = B(\lambda_j(p_j), \lambda_k(p_k))$. To start, we will treat the case that $\Gal(K/\Q)$ has order $4$. Since $\gamma_j \neq 0$ and $m_k \neq 0$, we get in all cases
\[
\tau_j \in \langle \tau_{-3d} \rangle \quad\text{and}\quad \tau_k \in \langle \tau_{d} \rangle. 
\]
We will now distinguish four cases according to the values of $\tau_j$ and $\tau_k$, and one final case when $\Gal(K/\Q)$ has order $2$.

\paragraph{Case I: $\tau_j = \tau_{-3d}, \tau_k = \tau_d$.} This forces $\gamma_k = 0$ and $m_j = 0$. By checking \eqref{eq:Bcase}, we find that $B = \{1\}$, and \eqref{eq:ignorable} holds with $\tau = 1$.

\paragraph{Case II: $\tau_j = \tau_{-3d}, \tau_k = 1$.} This forces $m_j = 0$. By checking \eqref{eq:Bcase} again, we find that $B = \{1, \tau_{-3}\}$. If $\gamma_j \neq \gamma_k$, then \eqref{eq:ignorable} holds with $\tau = 1$, and otherwise it holds with $\tau = \tau_{-3}$.

\paragraph{Case III: $\tau_j = 1, \tau_k = \tau_d$.} In this case we follow exactly the same logic as case II with the roles of $m_j$ and $\gamma_k$ reversed.

\paragraph{Case IV: $\tau_j = \tau_k = 1.$} By checking \eqref{eq:Bcase}, we discover that $B = \Gal(K/\Q)$. Then we may select $\tau \in \Gal(K/\Q)$ such that $\tau \gamma_j \neq \gamma_k$ and $\tau m_j \neq m_k$, and \eqref{eq:ignorable} holds for this choice of $\tau$.

\paragraph{Case V.} It remains to treat the case where $\Gal(K/\Q)$ has order $2$. If $d$ is a square, then $\gamma_j \neq 0$ and $m_j \neq m_k$. This gives $\tau_j = 1$. If $\tau_k \neq 1$, then $\gamma_k = 0$ and \eqref{eq:Bcase2} gives $B = \{1\}$. Selecting $\tau = 1$ verifies \eqref{eq:ignorable}. If instead $\tau_k = 1$, then \eqref{eq:Bcase2} gives $B = \Gal(K/\Q)$. Choosing $\tau$ such that $\tau \gamma_j \neq \gamma_k$ shows that \eqref{eq:ignorable} holds. Finally, if $-3d$ is a square instead, we argue similarly as above.
\end{proof}

\begin{lemma}
\label{lMainTerm}
With all notation as in Lemma \ref{lem:bilin_ign}, if $(\mu, \gamma)$ satisfies either of the first two enumerated conditions, and if $\mu$ lies in $\mathcal{M}_1(M_{-3d})$, we have
\[
(\gamma \cdot \mu)_y = 0 \quad\textup{for all } y \in X.
\] 
\end{lemma}

\begin{proof}
We see $\mu$ satisfies the first condition of Lemma \ref{lem:bilin_ign} only if it either lies in $\mathcal{D}(M_{-3d})$ or in $\mathcal{M}_{0, \Vplac}(M_{-3d}) + \mathcal{D}(M_{-3d})$. If it lies in the former, $\mu$ parameterizes a Selmer element and the result follows. And it cannot lie in the latter by the assumption that $\mu$ lies in $\mathcal{M}_1(M_{-3d})$. So, if the first condition is satisfied, the result follows.

Now suppose $\gamma$ satisfies the second condition. In this case, the result is clear unless $-3d$ is a square. In this case, we may use the argument of \cite[Lemma 7.5]{Smi22b}, which invokes Poitou--Tate duality, to show that $(\gamma \cdot \mu)_y = 0$.
\end{proof}

\section{\texorpdfstring{The distribution of $\sqrt{-3}$-Selmer groups}{The distribution of sqrt(-3)-Selmer groups}}
\label{sFinal}
After the algebraic setup in the last section, we are ready to do the analysis necessary to find the distribution of $\sqrt{-3}$-Selmer groups. This starts by constructing the grids of Section \ref{ssec:alg_grids} more explicitly. The $\sqrt{-3}$-Selmer group was previously studied from a statistical perspective by Chan \cite{Chan} and Fouvry \cite{Fouvry}. 

\subsection{Analytic formulation of grids}
Our setup is similar to \cite[Section 8]{Smi22b}, which we still cannot invoke directly since our twists are parameterized by $H^1(G_{\Q}, \mu_3)$ rather than $H^1(G_{\Q}, \mathbb{F}_3)$. We fix a nonzero integer $d$, the field extension $K/\Q$, and the set of places $\Vplac_0$ as in Section \ref{sec:Leg}.

\begin{mydef}
Let $H \geq 20$ be a real number. We define
\[
\alpha_0(H) := \exp^{(3)}\left(\frac{1}{3} \log^{(3)} H\right), \quad \alpha(H) := \exp\left(\exp^{(3)}\left(\frac{1}{4} \log^{(3)} H\right)^{-1}\right).
\]
Take $b$ to be a nonzero integer with all prime divisors smaller than $\alpha_0(H)$, and take $\Vplac$ to be the set of places in $\Vplac_0$ together with the primes dividing $b$. We will take $K(\Vplac)$ to be the maximal abelian extension of $K$ of exponent $3$ that is ramified only at the places in $\Vplac$.

For $i \in \Z_{\geq 0}$, we define the corresponding interval of primes
\[
\mathcal{P}_i(H) := \{p \not \in \Vplac_0 : \alpha_0(H) \alpha(H)^i \leq p < \alpha_0(H) \alpha(H)^{i + 1}\}.
\]
For each prime $p \ge \alpha_0(H)$ outside $\Vplac_0$, we make a choice of a prime $\lambda(p)$ of $K$ above $p$. We will select these primes so that, if we have two primes $p_1, p_2$ greater that $\alpha_0(H)$ such that $\Frob\, p_1$ and $\Frob\, p_2$ are the same conjugacy class in $\Gal(K(\Vplac)/\Q)$, then $\Frob\,\lambda(p_1)$ and $\Frob\,\lambda(p_2)$ have the same image in $\Gal(K(\Vplac)/\Q)$.

Then, given a positive integer $r$ and an increasing sequence $k_1 < \dots < k_r$ of nonnegative integers, and given a sequence $C_1, \dots, C_r$ of conjugacy classes of $\Gal(K(\Vplac)/\Q)$, we define a grid $(b, X_1, \dots, X_r)$ by taking
\[X_i = \{p_i \in \mathcal{P}_{k_i}(H)\,:\,\, \Frob\, p_i \text{ maps to } C_i\}.\]
Together with the map $\lambda$, this data defines a grid in the sense of Section \ref{ssec:alg_grids}. We also call it a \emph{filtered grid}.

Ignoring the $C_i$, we may define an \emph{unfiltered grid} by taking
\[X' = \prod_{i \le r} X'_i \quad\text{with}\quad X'_{i} = \mathscr{P}_{k_i}(H).\]
Given $x = (p_1, \dots, p_r)$ in $X'$, we take $n(x) = b \cdot p_1 \cdot \dots \cdot p_r$.

Each positive integer $n \le H$ is equal to $n(x)$ for at most one unfiltered grid constructed this way, and it lies in a unique filtered subgrid of this unfiltered grid.
\end{mydef}

We isolate some of the other properties we need from our grid.

\begin{mydef}
Let $H \ge 20$. Take $X'$ to be an unfiltered grid corresponding to $(b, k_1, \dots, k_r)$. Call $X'$ \emph{good of height $\le H$ } if
\begin{enumerate}
\item[(1)] \emph{(Reasonable amount of primes)} We have 
$$
\log \log H - (\log \log H)^{3/4} \leq r \leq \log \log H + (\log \log H)^{3/4}.
$$
\item[(2)] \emph{(Height of the grid)} We have $|n(x)| \le H$ for all $x$ in $X'$.
\item[(3)] \emph{(Few small primes)} The integer $b$ has at most $(\log \log H)^{\frac{1}{3} + \frac{1}{100}}$ distinct prime factors, and we have
\[\min \{ p \in \mathscr{P}_{k_i}(H)\} \ge \exp^{(3)}\left(\tfrac{1}{2}\log^{(3)}H\right) \quad\text{if}\quad i > (\log \log H)^{\tfrac{1}{2} + \tfrac{1}{100}}.\]
\end{enumerate}
\end{mydef}

\begin{lemma}
\label{lem:unfiltered_good}
There is some $C > 0$ depending only on $\Vplac_0$ such that, given $H \ge 20$, the number of positive integers $n \le H$ that are not in a good unfiltered grid of height $\le H$ is at most
\[CH \cdot \exp^{(2)} \left(\tfrac{1}{3} \log^{(3)} H\right)^{-1/3}.\]
\end{lemma}

\begin{proof}
Take $B = \exp^{(2)} \left(\tfrac{1}{3}\log^{(3)} H\right)^{2/3}$. Given a positive integer $n$, take $\text{rad}(n)$ to be the product of the primes dividing $n$. Every positive integers $n \le H$ with $n/\text{rad}(n) \ge B$ is the product of a squarefree number and a squarefull number of size at least $B$. Since the number of squarefull numbers less than $H_0$ for $H_0 \ge 3$ is $\mathcal{O}(H_0^{1/2})$, we find that the number of integers $n \le H$ with $n/\text{rad}(n) \ge B$ is
\[
\mathcal{O}\left(B^{-1/2}H\right).
\]
Here and throughout the proof, all implicit constants depend only on $\Vplac_0$.

Take $T$ to be the set of squarefree positive integers $n \le H$ that either have more than $(\log \log H)^2$ prime factors, or which sit in an unfiltered grid violating condition (3) of being good. We have
\[
|T| = \mathcal{O}\left(B^{-3/2}H\right)
\]
by \cite[Proposition 8.5]{Smi22b}. By separately considering the case that $n/\text{rad}(n) \ge B$ and the case $n/\text{rad}(n) \le B$, we find that the number of positive integers $n \le H$ with $\text{rad}(n)$ in $T$ is $\mathcal{O}\left(B^{-1/2}H\right)$.

Given an integer $n \ge H$ with $\text{rad}(n)$ outside $T$ and $n/\text{rad}(n) \le B$, we find that $n$ is in an unfiltered grid satisfying condition (2) so long as
\[
B < \alpha_0(H) \quad\text{and}\quad \alpha(H)^{(\log  \log H)^2}n \le H.
\]
The former holds so long as $H \gg 1$, and the latter certainly holds for all but at most $\mathcal{O}(B^{-1/2}H)$ positive integers less than $H$.

This leaves condition (1). From the moderate rate deviation principle \cite[Theorem 6]{Mehrdad Zhu 13} applied with $a_n = (\log \log H)^{3/4}$, we find that the number of positive integers $\le H$ with number of distinct prime divisors outside the range
\[
\left[\log \log H - \tfrac{1}{2} (\log \log H)^{3/4}, \, \log \log H + \tfrac{1}{2} (\log \log H)^{3/4}\right]
\]
is at most $H \cdot \exp\left(-\tfrac{1}{10} (\log \log H)^{1/2}\right)$ for sufficiently large $H$. Since we have already excluded integers in a grid not satisfying (3), we find that the result follows.
\end{proof}

\begin{lemma}
\label{lem:filtered_assu}
There is an absolute $c > 0$ such that we have the following:

Take $X'$ to be a good unfiltered grid of ideals of height $\le H$ corresponding to $(b, X'_1, \dots, X'_r)$ with $H \ge 20$. Define $\Vplac$ and $K(\Vplac)$ as from $b$ as above. Choose $\epsilon > 0$.

\begin{enumerate}
\item[(1)] Take $Y'$ to be the subset of $x = (p_1, \dots, p_r) \in X'$ such that, if we take $S(x) $ to be the set of $i \le r$ for which $\Frob\, p_i$ lies in $G_K$, the normal subgroup generated by the set
\[
\{\Frob\, p_i\,:\,\, i \in S(x)\}
\]
does not generate $\Gal(K(\Vplac)/K)$. Then
\[
|Y'| = \mathcal{O}\left(|X'| \cdot (\log H)^{-c}\right).
\]

\item[(2)] Given $\sigma$ in $\Gal(K/\Q)$, the number of $(p_1, \dots, p_r)$ in $X'$ such that we have
\[
\left|\#\{i \le r\,:\,\, \Frob\, p_i = \sigma\} - \frac{r}{\# \Gal(K/\Q)} \right| \ge \epsilon r
\]
is $\mathcal{O}\left(|X'|\cdot\exp\left( -(\log \log H)^{1 - \epsilon}\right)\right)$.

\item[(3)] For any given integer $a$, the number of $(p_1, \dots, p_r)$ in $X'$ such that
\[ 
\sum_{p \mid n(x)} \dim H^0(G_p, M_d) - \dim H^0(G_p, M_{-3d}) = a 
\]
is at most $\mathcal{O}\left(|X'|\cdot (\log \log H)^{-1/2}\right)$. 
\end{enumerate}

\noindent Here, all implicit constants are determined by $K$, $\Vplac$, and $\epsilon$.
\end{lemma}

\begin{proof}
Every maximal proper subgroup of $\Gal(K(\Vplac)/K)$ has index $3$ in the whole group, so every proper subgroup of $\Gal(K(\Vplac)/K)$ that is maximal among normal subgroups of $\Gal(K(\Vplac)/\Q)$ of this form has index dividng $3^4$ since $\Gal(K/\Q)$ has order at most $4$. We can bound the number of such subgroups by $\#\Gal(K(\Vplac)/K)$, which  is $\mathcal{O}(3^{4w})$, where $w$ is the number of prime divisors of $b$.

So, if $(p_1, \dots, p_r)$ lies in $Y'$, there is some extension  $L/K$ with $[L: K]$ dividing $3^4 = 81$ that is normal over $\Q$ and ramified only at places in $\Vplac$, and there is some conjugacy class $C$ in $\Gal(L/\Q)$, such that $\Frob\, p_i$ lies outside $C$ for each $p_i$. We say that $L$ is then \emph{bad} for $x$. The degree of $L$ is bounded by $2^2 \cdot 3^4 = 324$, and its discriminant is bounded by
\[
\mathcal{O}(\text{rad}(b)^{324}).
\]
Since we bounded the size of the prime divisors of $b$ in Lemma \ref{lem:unfiltered_good}, we find that this is at most
\[
\exp^{(3)}\left(\tfrac{2}{5}\log^{(3)} H\right) \quad\text{for } H \gg 1.
\]
Applying the Chebotarev density theorem \cite{LO} and condition (3) of Lemma \ref{lem:unfiltered_good} then suffices to show that, for $i > (\log  \log H)^{\tfrac{1}{2} + \tfrac{1}{100}}$,
\[
\# \{p \in X'_i\,:\,\, \Frob\, p \ne C \,\,\text{in}\,\,\Gal(L/\Q)\} \le \frac{323 + 1/2}{324} \cdot \# X'_i\quad \text{for }H \gg 1,
\]
where the implicit constant depends on $K$ and $\Vplac_0$. We note that we do not need to worry about Siegel zeros since every real quadratic field contained in $L$ is also contained in $K$, which we have fixed. For $H \gg 1$, the number of $x$ such that $L$ is bad for $x$ is at most
\[
\left(\frac{323 + 1/2}{324}\right)^{r - (\log  \log H)^{\frac{1}{2} + \frac{1}{100}}} \cdot |X'|,
\]
and we may bound $r$ using (1) of Lemma \ref{lem:unfiltered_good}. Meanwhile, the number of $L$ is at most $\mathcal{O}\left(\exp\left(5 (\log\log H)^{\frac{1}{3} + \frac{1}{100}}\right)\right)$. The first part follows.

Using the Chebotarev density theorem, given $\sigma$ in $\Gal(K/\Q)$, we have
\[
\left|\#\{ p \in X'_i\,:\,\, \text{Frob}\, p = \sigma\} - \frac{\# X'_i}{\# \Gal(K/\Q)}\right| \le \frac{\#X'_i}{\exp^{(3)}\left(\tfrac{1}{4} \log^{(3)} H\right)}
\]
so long as $H$ is sufficiently large. Parts (2) and (3) thus follow from consideration of the random variable
\[
V_1 + \dots + V_r \in \Z[\Gal(K/\Q)],
\]
where the $V_i$ are independent variables taking the uniform distribution on $\Gal(K/\Q)$. In particular, (2) follows from Hoeffding's inequality, and (3) follows from the local limit theorem.
\end{proof}

\subsection{The first moment}
Choose a good unfiltered grid $X'$ as above, and choose a point $x = (p_1, \dots, p_r)$ in $X'$. Take $X = X_1 \times \dots \times X_r$ to be the filtered subgrid of $X$ containing $x$. Given $y \in X'$, there are two  clear lower bounds for the size of $\Sel^{\sqrt{-3}} E_{-27d, n(y)}$. First, from the connecting map \eqref{eq:Edn3} applied to this  curve, we have the trivial bound
\[ 
\left|\Sel^{\sqrt{-3}} E_{-27d, n(y)}\right| \,\ge\, \left|H^0(G_{\Q}, M_d)\right|.
\]
In addition, from the Greenberg--Wiles formula \eqref{eq:Tama_relation}, we have the dual trivial bound
\[ 
\left|\Sel^{\sqrt{-3}} E_{-27d, n(y)}\right| \, = \, \left|\Sel^{\sqrt{-3}} E_{d, n(y)}\right| \cdot \mathcal{T}(E_{-27d, n(y)}) \ge \left|H^0(G_{\Q}, M_{-3d})\right| \cdot \mathcal{T}(E_{-27d, n(y)}).
\]
Here, we have used the fact that the Tamagawa ratio is constant across the filtered grid $X$, as follows from \eqref{eq:loc_Tama} together with the fact that the twists $E_{-27d, n(y)}$ are isomorphic at all primes in $\Vplac_0$.

One of these bounds is usually far larger than the other due to Lemma \ref{lem:filtered_assu} (3). If there is such a misbalance, the next proposition shows that most of the twists from the filtered grid have the minimal possible size.

\begin{theorem}
\label{tFirstMoment}
Choose $x$ and $X$ as above in a good unfiltered grid $X'$ of height $\le H$. We assume $x$ is outside the set $Y'$ defined in Lemma \ref{lem:filtered_assu} (1).

Define
\begin{multline*}
\kappa =  \left| H^0(G_{\Q}, M_d)\right| + \left| H^0(G_{\Q}, M_{-3d})\right| \cdot \mathcal{T}(E_{-27d, n(x)}) + \\
- \frac{|H^0(G_\Q, M_d)| \cdot |H^0(G_\Q, M_{-3d})| \cdot |\mathcal{M}_{0, \Vplac}(M_{-3d})|}{|\mathcal{R}|} - \frac{|H^0(G_\Q, M_d)| \cdot |H^0(G_\Q, M_{-3d})|}{|\mathcal{R}|}.
\end{multline*}
We then have
\[
\left|\kappa \cdot |X| \, -\, \sum_{y \in X} \left|\Sel^{\sqrt{-3}} E_{-27d, n(y)}\right| \right| = \mathcal{O}\left(|X| \cdot \alpha_0(H)^{-c}\right),
\]
where $c > 0$ is some absolute constant and the implicit constant depends just on $d$ and $\Vplac_0$.
\end{theorem}

\begin{proof}
We return to equation \eqref{eq:average}, so that our aim is to bound
\begin{align}
\label{eSelmerMomentGoal}
\left|\kappa \cdot |X| - \frac{1}{\left| \mathcal{R}\right|} \sum_{\mu \in \mathcal{M}_0(M_{-3d})} \,\sum_{\gamma \in \mathcal{R}} \, \sum_{y \in X} e\left((\gamma \cdot \mu)_y \right)\right|.
\end{align}
We start by applying Lemma \ref{lem:bilin_ign}, and we will analyse the contribution from each pair $(\mu, \gamma)$ depending on the three arising cases. We will see that the first two cases contribute to the main term, while the third case will be absorbed into the error term. Let us first dispel with the pairs $(\mu, \gamma)$ falling in the third case. By definition, we can find integers $j$ and $k$ and $\tau \in B(\lambda_j(p_j), \lambda_k(p_k))$ satisfying \eqref{eq:ignorable}. Fix $q_i \in X_i$ for every $i \not \in \{j, k\}$ and take 
$$
Z = X_j \times X_k \times \prod_{\substack{1 \leq i \leq r \\ i \not \in \{j, k\}}} \{q_i\}.
$$
Then, for pairs $(\mu, \gamma)$ in the third case, we claim the existence of $c > 0$ satisfying
\begin{align}
\label{eDoubleSum}
\left|\sum_{z \in Z} e\left((\gamma \cdot \mu)_z\right)\right| = \mathcal{O}\left(|X_j| \cdot |X_k| \cdot \alpha_0(H)^{-c}\right).
\end{align}
Once the claim is established, we see that the total contribution from the third case may be absorbed into the error term of the proposition by isolating their contribution through an application of the triangle inequality in \eqref{eSelmerMomentGoal}. 

In order to prove the claim \eqref{eDoubleSum}, we apply Lemma \ref{lem:bilinear_vary} to find functions $\alpha: X_j \rightarrow \mathbb{C}^\ast$ and $\beta: X_k \rightarrow \mathbb{C}^\ast$ both of magnitude at most $1$ such that
\[
e\left((\gamma \cdot \mu)_z\right) = \alpha(q_j) \beta(q_k) \prod_{\tau \in B(\lambda_j(q_j), \lambda_k(q_k))} e\left(b_{\tau}\left(\symb{\tau\lambda_j(q_j)}{\lambda_k(q_k)}\right)\right)
\]
for all $z = (q_1, \dots, q_r) \in Z$, where $b_{\tau} \in \Hom\left(\mu_3, \tfrac{1}{3}\Z/\Z\right)$ is given by
\[
b_{\tau} \,=\, - (\tau \gamma_j - \gamma_k) \cdot (\tau m_j - m_k).
\]
By \eqref{eq:ignorable}, we know that at least one $b_\tau$ is nonzero.

At this stage, we would like to apply the bilinear sieve for symbols as presented in \cite[Theorem 5.2]{Smi22a}. This equidistribution result uses a more general symbol $\symb{\mfp}{\mfq}_{\text{gen}}$, which is a function from $B(\mfp, \mfq)$ to the roots of unity. To convert between the symbol $\symb{\mfp}{\mfq}_{\text{gen}}$ and our symbol $\symb{\mfp}{\mfq}$, we use the isomorphism $\varphi$ on top of page 19 of \cite{Smi22a} to get an identification $\symb{\mfp}{\mfq} = \varphi(\symb{\mfp}{\mfq}_{\text{gen}}(1))$. Then \cite[Proposition 3.16(3)]{Smi22a} yields
\[
\symb{\tau \mfp}{\mfq} = \varphi(\symb{\tau \mfp}{\mfq}_{\text{gen}}(1)) = \varphi(\tau(\symb{\mfp}{\mfq}_{\text{gen}}(\tau))).
\]
The bilinear sieve \cite[Theorem 5.2]{Smi22a} shows that the vector $\symb{\mfp}{\mfq}_{\text{gen}}(\tau)$ as $\tau$ ranges over $B(\mfp, \mfq)$ is equidistributed. Therefore, by the above identity, the symbols $\symb{\tau \mfp}{\mfq}$ are independent and distributed uniformly at random as $\tau$ ranges over $B(\mfp, \mfq)$. Hence our claim \eqref{eDoubleSum} is a consequence of our assumption that at least one $b_\tau$ is nonzero.

It remains to handle the pairs $(\mu, \gamma)$ belonging to the first two cases. By virtue of $x$ lying outside the set $Y'$ defined in Lemma \ref{lem:filtered_assu} (1), we are allowed to apply Lemma \ref{lem:nosmallramonly}. In particular, equation \eqref{eq:M0_count} of that lemma gives
\begin{align}
\label{eM0count}
\left|\mathcal{M}_0(M_{-3d})\right| = \mathcal{T}(E_{-27d, n(x)}) \cdot \prod_{i \le r} \left|H^0(G_{p_i}, M_{-3d})\right| = \mathcal{T}(E_{-27d, n(x)}) \cdot |\mathcal{R}|.
\end{align}
Furthermore, if $\mu$ is not in $\mathcal{M}_1(M_{-3d})$, then it cannot parameterize a Selmer element due to equation \eqref{eq:nosmallramonly}. Therefore we may restrict the summation in \eqref{eDoubleSum} to $\mathcal{M}_1(M_{-3d})$. For the remaining pairs $(\mu, \gamma)$, all the conditions of Lemma \ref{lMainTerm} are satisfied, so we conclude that $e\left((\gamma \cdot \mu)_y \right) = 1$ for all $y$ and all such pairs. Therefore to finish the proof of the proposition, we need to count the total number of pairs $(\mu, \gamma)$ left, i.e. lying in the first or second case of Lemma \ref{lem:bilin_ign} and satisfying $\mu \in \mathcal{M}_1(M_{-3d})$.

From the first case we get a total of $|H^0(G_\Q, M_d)| \cdot |\mathcal{R}|$ pairs, while the second case gives a total of $|H^0(G_\Q, M_{-3d})| \cdot |\mathcal{M}_1(M_{-3d})|$. By equation \eqref{eM0count}, this last number is equal to
$$
|H^0(G_\Q, M_{-3d})| \cdot (\mathcal{T}(E_{-27d, n(x)}) \cdot |\mathcal{R}| - |\mathcal{M}_{0, \Vplac}(M_{-3d})| \cdot |H^0(G_\Q, M_d)|).
$$
From this count, we have to subtract pairs $(\mu, \gamma)$ that are in both cases. This amounts to subtracting $|H^0(G_\Q, M_{-3d})| \cdot |H^0(G_\Q, M_d)|$. Adding these counts together and dividing through $|\mathcal{R}|$ ends the proof of the proposition.
\end{proof}

Let us immediately take advantage of Theorem \ref{tFirstMoment} to deduce Proposition \ref{prop:blowup} from the introduction.

\begin{proof}[Proof of Proposition \ref{prop:blowup}]
If neither $d$ nor $-3d$ is a square, then Proposition \ref{prop:blowup} is a direct consequence of Theorem \ref{tFirstMoment}, the definition of the Tamagawa ratio and Lemma \ref{lem:filtered_assu} (3). If $d$ is a square, then we have
$$
\sum_{p \mid n(x)} \dim H^0(G_p, M_d) - \dim H^0(G_p, M_{-3d}) = \sum_{p \mid n(x)} \mathbf{1}_{p \equiv 1 \bmod 3}.
$$
In this case an application of Turan's trick yields Proposition \ref{prop:blowup}. A similar argument works if $-3d$ is a square.
\end{proof}

\subsection{Equidistribution of the Cassels--Tate pairing}
We shall now prove Theorem \ref{thm:main}. It is at this stage that we shall benefit from our work in Section \ref{sec:trilinear}, which we will eventually combine with Proposition \ref{pCT}. Recall that we introduced a real number $\alpha_0$ in the statement of Theorem \ref{thm:main}. 

\begin{theorem}
\label{tCTReduction1}
Fix $W \in \pm 1$ and fix a nonzero integer $d$. There exists $C > 0$ such that the following holds. 
Let $X$ be a good unfiltered grid of height $H > C$. Then, for any integer $r \ge 0$ satisfying $(-1)^r = W$, we have
\[
\left|\frac{|\{x \in X : r_3(E_{d, n(x)}) = r, w(E_{d, n(x)}) = W\}|}{|\{x \in X : w(E_{d, n(x)}) = W\}|} - \alpha_0 \cdot 3^{\frac{-r(r-1)}{2}} \cdot \prod_{k=1}^r \left(1 - 3^{-k}\right)^{-1}\right| \leq \frac{C}{(\log^{(2)} H)^{\frac{1}{4}}}.
\]
\end{theorem}

Note that Theorem \ref{tCTReduction1} readily implies Theorem \ref{thm:main} thanks to Lemma \ref{lem:unfiltered_good}. It remains to prove Theorem \ref{tCTReduction1}. The principal bottlenecks for improving our error term in Theorem \ref{tCTReduction1} stem from Lemma \ref{lem:unfiltered_good} and Lemma \ref{lem:filtered_assu}.

We prove Theorem \ref{tCTReduction1} by reducing to grids consisting of three intervals where we are able to apply our trilinear sieve result. In order to achieve such a setup, we perform two successive reduction steps. 

From now on, we let $X$ be a good unfiltered grid of height $H$. We take $x = (p_1, \dots, p_r) \in X$ and we let $Y$ be the filtered subgrid of $X$ corresponding to $x$. We consider the following properties of a filtered subgrid $Y$ (compare with Lemma \ref{lem:filtered_assu}):
\begin{itemize}
\item[(G1)] the normal subgroup generated by $\{\mathrm{Frob} \ p_i : i \in S(x)\}$ equals $\Gal(K(\Vplac)/K)$,
\item[(G2)] for every $\sigma \in \Gal(K/\Q)$
$$
\left|\#\{i \leq r : \mathrm{Frob} \ p_i = \sigma\} - \frac{r}{\#\Gal(K/\Q)}\right| < \frac{r}{100},
$$
\item[(G3)] we have
$$
\left|\sum_{i \leq r} \dim H^0(G_{p_i}, M_d) - \dim H^0(G_{p_i}, M_{-3d})\right| \geq (\log \log H)^{1/4}.
$$
\end{itemize}
Observe that the root number is constant on $Y$, and write $W$ for this number. We write $G$ for the number of grid classes of $Y$, and we take representatives $y_1, \dots, y_G$ for the grid classes of $Y$. Attached to $Y$, we define two distinct integers $d_1, d_2 \in \{d, -27d\}$ in such a way that 
$$
\sum_{i \leq r} \dim H^0(G_{p_i}, M_{d_2}) \leq \sum_{i \leq r} \dim H^0(G_{p_i}, M_{d_1}). 
$$
We say that a grid class $[y_j]$ is a \emph{reasonable} grid class of $Y$ if
\[
|[y_j]| \geq \frac{|Y|}{G \cdot (\log H)} \quad \text{ and } \quad \Sel^{\sqrt{-3}} E_{d_1, n(y)} = 0 \text{ for all } y \in [y_j].
\]
Note that the latter condition does not depend on the choice of $y \in [y_j]$ by Proposition \ref{prop:48}. In order to perform our first reduction step, we will fix a (reasonable) grid class. The benefit of doing so is that the Cassels--Tate pairings are then defined on the same space.

\begin{mydef}
\label{dV}
Fix a grid class $[y_j]$ and some element $y \in [y_j]$. Write 
$$
V := \mfB_{y, M_{d_2}}^{-1} \left(\Sel^{\sqrt{-3}} E_{d_2, n(y)}\right) \subseteq H^1_\Vplac(G_{\Q}, M_{d_2}) \oplus \bigoplus_{i=1}^r M_{d_2}(-1)^{G_{p_i}},
$$
which does not depend on $y \in [y_j]$ thanks to Proposition \ref{prop:48}. Write $t := \dim_{\FF_3} V$ and let $\phi_1, \dots, \phi_t$ be a basis of $V$. Since the pairing $\CTP_{d_1, n(y)}$ is alternating, we can naturally attach an element of $(\wedge^2 V)^\vee := \Hom(\wedge^2 V, \Q/\Z)$ to it.

If $d_1$ is a square, then $\chi_{3, d_1 n(y)^2/16}$ is always a Selmer element and we may assume without loss of generality that this is $\phi_1$. This element is also always in the kernel of $\CTP_{d_1, n(y)}$.

We say that an element $\omega \in \wedge^2 V$ is a test vector if
\begin{itemize}
\item in case $d_1$ is not a square, then we demand that $\omega$ is not zero,
\item in case $d_1$ is a square, then we demand that $\omega = \sum_{2 \leq i < j \leq t} c_{i, j} \phi_i \wedge \phi_j$ for some coefficients $c_{i, j} \in \mathbb{F}_3$ not all equal to $0$.
\end{itemize}
\end{mydef}

We are now ready to state our first reduction step. One highlight of Theorem \ref{tCTReduction2} is the exceptionally strong error term, which is due to our trilinear sieve result. Having access to such a strong error term is in fact of vital importance, as we have to multiply this error by $|\wedge^2 V|$ when we reduce from Theorem \ref{tCTReduction2} to Theorem \ref{tCTReduction1}. Since $\dim V$ is roughly of order $(\log^{(2)} H)^{1/2}$ in our situation, this necessitates stronger analytic tools than the bilinear sieve results previously employed in \cite{Smi22a, Smi22b} or the Chebotarev density theorem \cite{KoyPag22}.

\begin{theorem}
\label{tCTReduction2}
Fix a nonzero integer $d$. There exists $C > 0$ such that the following holds. 
Let $X$ be a good unfiltered grid of height at most $H > C$, and let $Y$ be a filtered subgrid satisfying the conditions (G1)-(G3). Let $[y]$ be a reasonable grid class of $Y$. Write $V$ for the vector space associated to $[y]$ as in Definition \ref{dV}. Then we have for all test vectors $\omega \in \wedge^2 V$
\[
\left|\sum_{z \in [y]} \exp(2\pi i \cdot \CTP_{d_1, n(z)}(\omega))\right| \leq \frac{C \cdot |[y]|}{\exp^{(3)}\left(\frac{1}{5} \log^{(3)} H\right)}.
\]
\end{theorem}

\begin{proof}[Proof that Theorem \ref{tCTReduction2} implies Theorem \ref{tCTReduction1}] 
Let $X$ be a good unfiltered grid of height at most $H$. We split $X$ as the disjoint union over its filtered subgrids $Y$. We recall that the root number is constantly equal to some number $W$ on $Y$. If $Y$ satisfies the conditions (G1)-(G3), we claim that for all $r$ satisfying $(-1)^r = W$
\begin{align}
\label{eReduceToFiltered}
\left|\frac{|\{y \in Y : r_3(E_{d, n(y)}) = r\}|}{|Y|} - \alpha_0 \cdot 3^{\frac{-r(r-1)}{2}} \cdot \prod_{k=1}^r \left(1 - 3^{-k}\right)^{-1}\right| \leq \frac{C}{(\log^{(2)} H)^{1/4}}
\end{align}
for some constant $C > 0$ depending only on $d$. Note that the claim readily implies Theorem \ref{tCTReduction1} by virtue of Lemma \ref{lem:filtered_assu}, so it suffices to prove the claim.

Recall that we attached two distinct integers $d_1, d_2 \in \{d, -27d\}$ to $Y$ in such a way that 
$$
\sum_{i \leq r} \dim H^0(G_{p_i}, M_{d_2}) \leq \sum_{i \leq r} \dim H^0(G_{p_i}, M_{d_1}), 
$$
hence
$$
\sum_{i \leq r} \dim H^0(G_{p_i}, M_{d_2}) + (\log^{(2)} H)^{1/4} \leq \sum_{i \leq r} \dim H^0(G_{p_i}, M_{d_1})
$$
by condition (G3). With these choices, we return to equation \eqref{eTamagawa} to deduce that 
\begin{align}
\label{eTamagawaControl}
\mathcal{T}(E_{d_2, n(y)}) \gg 3^{(\log^{(2)} H)^{1/4}}, \quad \quad \mathcal{T}(E_{d_1, n(y)}) \ll 3^{-(\log^{(2)} H)^{1/4}}.
\end{align}
Therefore we expect $\Sel^{\sqrt{-3}} E_{d_2, n(y)}$ to be large and $\Sel^{\sqrt{-3}} E_{d_1, n(y)}$ to be the minimal possible size. To get a handle on the $3$-Selmer group of $E_{d_1, n(y)}$, our goal will be to calculate the Cassels--Tate pairing $\CTP_{d_1, n(y)}$ on $\Sel^{\sqrt{-3}} E_{d_2, n(y)}$ (recall the exact sequence \eqref{eq:Sel3decomp}). 

In order to do so, we recall that $G$ equals the number of grid classes for $Y$ so that
$$
G \leq e^{C_1 r^2} \leq e^{C_2 (\log^{(2)} H)^2}
$$
for some constants $C_1, C_2 > 0$ depending only on the starting tuple. By construction we have
\begin{align}
\label{eGridClassesCover}
\bigcup_{j = 1}^G [y_j] = \{y \in Y : w(E_{d, n(y)}) = W\} = Y.
\end{align}
Thanks to equation (\ref{eGridClassesCover}), Theorem \ref{tFirstMoment} and equation \eqref{eTamagawaControl}, the union of the grid classes that are not reasonable is negligible compared to $|Y|$. Therefore it suffices to prove that there exist constants $c, C > 0$ such that for all reasonable grid classes $[y]$
$$
\left|\frac{|\{z \in [y] : r_3(E_{d, n(z)}) = r\}|}{|[y]|} - \alpha_0 \cdot 3^{\frac{-r(r-1)}{2}} \cdot \prod_{k=1}^r \left(1 - 3^{-k}\right)^{-1}\right| \leq \frac{C}{\exp\left(c \log^{(2)} H\right)}.
$$
Indeed, this readily implies the inequality \eqref{eReduceToFiltered}.

We now analyze the distribution of the Cassels--Tate pairing on reasonable grid classes $[y]$. Then the formula for $r_3(E_{d, n(z)})$ becomes
\begin{align}
\label{er3Formula}
r_3(E_{d, n(z)}) = s(d) + \dim_{\FF_3} \Sel^3 E_{d_1, n(z)} = s(d) + \dim_{\FF_3}(\ker \CTP_{d_1, n(z)}),
\end{align}
where $s(d) = -1$ if $d_1$ is a square and $s(d) = 0$ otherwise. One benefit of working with grid classes is that we are able to fix a uniform basis for $\Sel^{\sqrt{-3}} E_{d_2, n(z)}$ thanks to Proposition \ref{prop:48}. Let $V$ be the vector space from Definition \ref{dV} and $\phi_1, \dots, \phi_t$ be the corresponding basis. Since $[y]$ is reasonable, we have
$$
t \gg (\log^{(2)} H)^{1/4}
$$
by equations \eqref{eq:Tama_relation} and \eqref{eTamagawaControl}. Recall that the $3$-parity conjecture is known over $\Q$, hence $(-1)^{t + s(d)} = W$. Therefore $t + s(d)$ has the same parity as $r$ by our assumption on $W$. Writing $P^{\text{Alt}}(n | t)$ for the probability that a uniformly at random selected $t \times t$ alternating matrix with coefficients in $\mathbb{F}_3$ has kernel of dimension $n$, we see that it suffices to show that
\begin{align}
\label{eFinalReduction}
\left||\{z \in [y] : r_3(E_{d, n(z)}) = r\}| - P^{\text{Alt}}(r | t + s(d)) \cdot |[y]|\right| \leq \frac{C \cdot |[y]|}{\exp^{(3)}\left(\frac{1}{6} \log^{(3)} H\right)}
\end{align}
for some constant $C > 0$ depending only on $d$. Indeed, to prove this reduction step, we use the explicit formulas for $P^{\text{Alt}}(n | t)$ in \cite[p. 9]{Smi22b} (the formula there is stated for $\mathbb{F}_2$ but this is readily adapted to $\mathbb{F}_q$ for any $q$).

We are now ready to establish that equation \eqref{eFinalReduction} is a consequence of Theorem \ref{tCTReduction2}, which would complete the proof of our reduction step. In order to prove equation \eqref{eFinalReduction}, we employ orthogonality of characters on $\wedge^2 V$ to rewrite the left hand side as
$$
\left|\sum_{z \in [y]} \sum_{\substack{\chi \in (\wedge^2 V)^\vee \\ \dim(\chi) = r - s(d) \\ s(d) = -1 \Rightarrow \phi_1 \in \mathrm{LKer}(\chi)}} \hspace{-0.55cm} \frac{1}{|\wedge^2 V|} \sum_{\omega \in \wedge^2 V} \overline{\chi(\omega)} \exp(2 \pi i \cdot \CTP_{d_1, n(z)}(\omega)) - \sum_{z \in [y]} P^{\text{Alt}}(r | t + s(d))\right|,
$$
where $\mathrm{LKer}(\chi)$ is the left kernel of $\chi$ viewed as a pairing $V \times V \rightarrow \mathbb{C}^\ast$, and $\dim(\chi)$ is the dimension of $\mathrm{LKer}(\chi)$.

If $s(d) = 0$, then the contribution from $\omega = 0$ cancels out with the term $P^{\text{Alt}}(r | t + s(d)) = P(r | t)$. Then, by the triangle inequality, it suffices to bound
$$
\sum_{\substack{\chi \in (\wedge^2 V)^\vee \\ \dim(\chi) = r}} \hspace{-0.1cm} \frac{1}{|\wedge^2 V|} \sum_{\substack{\omega \in \wedge^2 V \\ \omega \neq 0}} \left|\sum_{z \in [y]} \exp(2 \pi i \cdot \CTP_{d_1, n(z)}(\omega))\right| \leq \sum_{\substack{\omega \in \wedge^2 V \\ \omega \neq 0}} \left|\sum_{z \in [y]} \exp(2 \pi i \cdot \CTP_{d_1, n(z)}(\omega))\right|.
$$
This sum is over all test vectors $\omega$, so we can apply Theorem \ref{tCTReduction2}. Summing the error from Theorem \ref{tCTReduction2} over all choices of $\omega$ shows that equation \eqref{eFinalReduction} holds.

If $s(d) = -1$, then we recall that $\phi_1$ is in the left kernel of $\CTP_{d_1, n(z)}$. In this case, we see that the contribution from those $\omega$ of the shape
\begin{align}
\label{eTrivialOmega}
\omega = \sum_{i = 2}^t c_i \phi_1 \wedge \phi_i \quad \text{ for some } c_i \in \mathbb{F}_3
\end{align}
cancels out the contribution from the term $P^{\text{Alt}}(r | t + s(d)) = P^{\text{Alt}}(r | t - 1)$.

For all $\omega$ not of the shape \eqref{eTrivialOmega}, we may add some $\omega'$ of the shape \eqref{eTrivialOmega} so that $\omega + \omega'$ is a test vector without affecting $\CTP_{d_1, n(z)}$ as $\phi_1$ is in the left and right kernel. Then proceeding as the case where $d_1$ is not a square finishes the proof of the reduction step.
\end{proof}

Inspired by equation (\ref{er3Formula}), we shall start calculating the Cassels--Tate pairing on reasonable grid classes. To do so, we will fix all our primes $Y_i$ for all $i \in \{1 \leq j \leq r\} - S$, except for the precisely three indices in $S$ to which we shall ultimately apply the trilinear sieve. Let us now say more about these indices $S$.

\begin{mydef}
Let $[y]$ be a reasonable grid class, write $V$ for the vector space from Definition \ref{dV} and take a test vector $\omega \in \wedge^2 V$. Thanks to our assumption (G1), it follows from Lemma \ref{lem:nosmallramonly} that equation \eqref{eq:nosmallramonly} holds. By putting our basis $\phi_1, \dots, \phi_t$ of $V$ in row echelon form and using equation \eqref{eq:nosmallramonly}, we find an injective map $s: \{1, \dots, t\} \rightarrow \{1, \dots, r\}$ such that 
$$
\pi_{s(i)}(\phi_j) \neq 0 \Longleftrightarrow i = j
$$
for all $1 \leq i, j \leq t$. We will abuse notation by continuing to write $\phi_1, \dots, \phi_t$ for this new basis. We will now construct three integers $i_1, i_2, i_3$ and we will put $S := \{i_1, i_2, i_3\}$.

We uniquely expand $\omega$ as
\[
\omega = \sum_{1 \leq i < j \leq t} c_{i, j} \phi_i \wedge \phi_j
\]
with $c_{i, j} \in \mathbb{F}_3$, and with $c_{1, j} = 0$ if $d_1$ is a square. By definition of a test vector, we may find some $j_1 < j_2$ with $c_{j_1, j_2} \neq 0$. We take $i_2 = s(j_1)$ and $i_3 = s(j_2)$.

If $d_1$ is a square, recall that $\chi_{3, d_1 n^2/16}$ is always a Selmer element and we assumed that it equals $\phi_1$. In this case, we define $i_1 = s(1)$. If $d_1$ is not a square, we take $i_1$ in such a way that $C_{i_1}$ projects non-trivially in $\Gal(\Q(\sqrt{d_1})/\Q)$. This is possible thanks to condition (G2). Note that this forces $\pi_{i_1}(\phi_j) = 0$ for all $j$ by the Selmer conditions.

If $i_1, i_2, i_3$ are constructed as above and $S = \{i_1, i_2, i_3\}$, then we say that $S$ is a set of \emph{variable indices} for $([y], \omega)$.
\end{mydef}

Our goal is to fix elements $y_i \in Y_i$ for all $i \not \in S$, and then get equidistribution of the Cassels--Tate pairing by varying over the intervals $Y_i$ with $i \in S$. Given some point $P \in \prod_{i \not \in S} Y_i$, we say that $P$ is large if 
\[
|[y] \cap \pi^{-1}(P)| \geq \frac{\left|[y]\right|}{\exp^{(3)}\left(\frac{1}{4} \log^{(3)} H\right) \cdot \prod_{i \not \in S} |Y_i|},
\]
where $\pi$ is the projection map from $Y$ to $\prod_{i \not \in S} Y_i$. Bounding the points $P$ that are not large trivially, it remains to prove equidistribution of the Cassels--Tate pairing on $[y] \cap \pi^{-1}(P)$. We will phrase this as our final reduction step.

\begin{theorem}
\label{tCTReduction3}
Fix a nonzero integer $d$. There exist $c, C > 0$ such that the following holds. 
Let $X$ be a good unfiltered grid of height at most $H > C$, and let $Y$ be a filtered subgrid satisfying the conditions (G1)-(G3). Let $[y]$ be a reasonable grid class of $Y$. Write $V$ for the vector space associated to $[y]$ as in Definition \ref{dV}. Let $\omega \in \wedge^2 V$ be a test vector and let $S$ be a set of variable indices for $([y], \omega)$. Then for all large points $P \in \prod_{i \not \in S} Y_i$
\[
\left|\sum_{\substack{z \in [y] \\ \pi(z) = P}} \exp(2\pi i \cdot \CTP_{d_1, n(z)}(\omega))\right| \leq \frac{C \cdot |[y] \cap \pi^{-1}(P)|}{\left(\exp^{(3)}\left(\frac{1}{3} \log^{(3)} H\right)\right)^c}.
\]
\end{theorem}

\begin{proof}[Proof that Theorem \ref{tCTReduction3} implies Theorem \ref{tCTReduction2}] 
We split the sum in Theorem \ref{tCTReduction2} over all $P \in \prod_{i \not \in S} Y_i$. If $P$ is not large, we bound trivially. After summing up the contribution of such $P$ we remain within the error term of Theorem \ref{tCTReduction2}. For the large $P$, we apply Theorem \ref{tCTReduction3}.
\end{proof}

We are now ready to embark on the proof of Theorem \ref{tCTReduction3} with the aid of the trilinear sieve.

\begin{proof}[Proof of Theorem \ref{tCTReduction3}]
Write $m$ for the integer obtained by multiplying out the coordinates of $P$ and recall that $S = \{i_1, i_2, i_3\}$. We isolate the sum over $y_1$ to get that the sum in the theorem is bounded by
\[
\sum_{y_2 \in Y_{i_2}} \sum_{y_3 \in Y_{i_3}} \left|\sum_{\substack{y_1 \in Y_{i_1} \\ \pi^{-1}(y_1, y_2, y_3) \in [y]}} \exp(2\pi i \cdot \CTP_{d_1, b m y_2y_3y_1}(\omega))\right|.
\]
After applying the Cauchy--Schwarz inequality, we see that it suffices to bound
\[
\sum_{y_2 \in Y_{i_2}} \sum_{y_3 \in Y_{i_3}} \sum_{\substack{y_1 \in Y_{i_1}, y_1' \in Y_{i_1} \\ \pi^{-1}(y_1, y_2, y_3) \in [y] \\ \pi^{-1}(y_1', y_2, y_3) \in [y]}} \exp\left(2\pi i \cdot \left(\CTP_{d_1, b m y_2y_3y_1}(\omega) - \CTP_{d_1, b m y_2y_3y_1'}(\omega)\right)\right).
\]
We now apply Proposition \ref{pCT}. Taking $T = \{1, j_1, j_2\}$ if $d_1$ is square and $T = \{j_1, j_2\}$ otherwise, we may write
\[
\omega = c_{j_1, j_2} \phi_{j_1} \wedge \phi_{j_2} + \sum_{k \not \in T} c_{j_1, k} \phi_{j_1} \wedge \phi_k + \sum_{k \not \in T} c_{k, j_2} \phi_k \wedge \phi_{j_2} + \sum_{k, \ell \not \in T, k < \ell} c_{k, \ell} \phi_k \wedge \phi_\ell,
\]
where we have taken $c_{j, k} = -c_{k, j}$ in the case that $j > k$. By the choice of the variable indices, we have
\[
\pi_{i_k}(\phi_j) = 0 
\]
for $j$ outside $T$ and $k = 1, 2, 3$, for $j = j_1$ and $k = 1,3$, and for $j = j_2$ and $k = 1, 2$.

Write $m_1$ for $c_{j_1, j_2}\pi_{i_2}(\phi_{j_1})$ and $m_2$ for $\pi_{i_3}(\phi_{j_2})$. These are both nonzero. Hiding all the other Cassels--Tate pairings in coefficients $a_{123}$, we thus rewrite the above as
\[
\sum_{y_2 \in Y_{i_2}} \sum_{y_3 \in Y_{i_3}} \sum_{\substack{y_1 \in Y_{i_1}, y_1' \in Y_{i_1} \\ \pi^{-1}(y_1, y_2, y_3) \in [y] \\ \pi^{-1}(y_1', y_2, y_3) \in [y]}} a_{123} \Redei{\chi_{3, y_1/y_1'}}{\phi + \mfB_{\lambda_{i_2}(y_2), M_{d_2}}(m_1)}{\phi' + \mfB_{\lambda_{i_3}(y_3), M_{d_2}}(m_2)},
\]
where $\phi$ and $\phi'$ are fixed cocycles and where $a_{123} = a_{12} \cdot a_{13} \cdot a_{23}$ with $a_{12}$ depending only on $y_1$, $y_1'$ and $y_2$, with $a_{13}$ depending only on $y_1$, $y_1'$ and $y_3$, and with $a_{23}$ depending only on $y_2$ and $y_3$. Furthermore, the condition of being in a grid class only depends on two coordinates at a time so we can absorb it in $a_{123}$. Applying Cauchy--Schwarz twice more, it therefore suffices to bound
\begin{align}
\label{eFinalSum}
\sum_{\substack{y_1 \in Y_{i_1} \\ y_1' \in Y_{i_1}}} \sum_{\substack{y_2 \in Y_{i_2} \\ y_2' \in Y_{i_2}}} \sum_{\substack{y_3 \in Y_{i_3} \\ y_3' \in Y_{i_3}}} a_{123}' \Redei{\chi_{3, y_1/y_1'}}{\phi_2(y_2, y_2')}{\phi_3(y_3, y_3')}
\end{align}
with 
\begin{gather*}
\phi_2(y_2, y_2') = \mfB_{\lambda_{i_2}(y_2), M_{d_2}}(m_1) - \mfB_{\lambda_{i_2}(y_2'), M_{d_2}}(m_1) \\
\phi_3(y_3, y_3') = \mfB_{\lambda_{i_3}(y_3), M_{d_2}}(m_2) - \mfB_{\lambda_{i_3}(y_3'), M_{d_2}}(m_2)
\end{gather*}
and with $a_{123}' = a_{12}'(y_1, y_1', y_2, y_2') a_{13}'(y_1, y_1', y_3, y_3') a_{23}'(y_2, y_2', y_3, y_3')$. The desired oscillation of equation \eqref{eFinalSum} is now a consequence of Theorem \ref{thm:trilinear}.
\end{proof}

\end{document}